\documentclass[times,sort&compress,3p]{elsarticle} 
\journal{Journal of Multivariate Analysis}
\usepackage{amsmath, amssymb, enumerate, amsthm}

\usepackage{bbm}
\usepackage{graphicx,psfrag,epsf}
\usepackage{appendix}
\usepackage{natbib}
\usepackage[normalem]{ulem} 
\usepackage{url} 
\usepackage{multirow,booktabs} 
\usepackage{algorithm}
\usepackage{algpseudocode}

\graphicspath{ {Figures/} }

\usepackage{forloop}
\usepackage{caption,subcaption}
\usepackage{pgffor}

\usepackage{xcolor}

\newtheorem{theorem}{Theorem}
\newtheorem{lemma}{Lemma}
\newtheorem*{lemma*}{Lemma}

\newtheorem*{assumption*}{Partial Exchangeability Assumption (PEA)}
\newtheorem{proposition}{Proposition}
\newtheorem*{proposition*}{Proposition}
\newtheorem{corollary}{Corollary}
\newtheorem*{example*}{Example}
\newtheorem{example}{Example}
\newtheorem{remark}{Remark}
\newtheorem{definition}{Definition}

\def\d{\mathrm{d}}

\providecommand{\bs}[1]{\boldsymbol{#1}}
\providecommand{\bss}[1]{\boldsymbol{\mathrm{#1}}}
\providecommand{\Id}[1]{\boldsymbol{\mathrm{I}_{#1}}}

\def\t{\boldsymbol{\tau}}
\def\th{\boldsymbol{\hat{\tau}}}
\def\tt{\boldsymbol{\tilde{\tau}}}

\def\eqdis{\stackrel{\mbox{\em\tiny d}}{=}}

\def\S{\boldsymbol{\Sigma}}
\def\Sh{\boldsymbol{\hat{\Sigma}}}
\def\St{\boldsymbol{\tilde{\Sigma}}}

\def\Ta{\boldsymbol{\Theta}}

\def\T{\boldsymbol{\mathrm{T}}}
\def\R{\boldsymbol{\mathrm{R}}}
\def\Tt{\boldsymbol{\tilde{\mathrm{T}}}}
\def\Th{\boldsymbol{\hat{\mathrm{T}}}}

\def\D{\boldsymbol{\Delta}}
\def\G{\boldsymbol{\Gamma}}

\def\B{\boldsymbol{\mathrm{B}}}
\def\U{\boldsymbol{\mathrm{U}}}
\def\X{\boldsymbol{\mathrm{X}}}
\def\SS{\boldsymbol{\mathrm{S}}}

\providecommand{\E}{\rm{E}}

\DeclareMathOperator*{\argmin}{\arg\,\min}

\makeatletter
\def\namedlabel#1#2{\begingroup
   \def\@currentlabel{#2}%
   \label{#1}\endgroup
}
\makeatother

\allowdisplaybreaks

\begin{document}

\begin{frontmatter}

\title{Detection of block-exchangeable structure in large-scale correlation matrices}

\author[A1]{Samuel Perreault\corref{mycorrespondingauthor}}
\author[A1]{Thierry Duchesne}
\author[A2]{Johanna G. Ne\v{s}lehov\'a}

\address[A1]{D{\'e}partement de math{\'e}matiques et de statistique, Universit{\'e} Laval, Qu{\'e}bec, Canada}
\address[A2]{Department of Mathematics and Statistics, McGill University, Montr{\'e}al, Canada}

\cortext[mycorrespondingauthor]{Corresponding author. Email address: \url{samuel.perreault.3@ulaval.ca}}
  
\begin{abstract}
Correlation matrices are omnipresent in multivariate data analysis. When the number $d$ of variables is large, the sample estimates of correlation matrices are typically noisy and conceal underlying dependence patterns. We consider the case when the variables can be grouped into $K$ clusters with exchangeable dependence; this assumption is often made in applications, e.g., in finance and econometrics. Under this partial exchangeability condition, the corresponding correlation matrix has a block structure and the number of unknown parameters is reduced from $d(d-1)/2$ to at most $K(K+1)/2$. We propose a robust algorithm based on Kendall's rank correlation to identify the clusters without assuming the knowledge of $K$ a priori or anything about the margins except continuity. The corresponding block-structured estimator performs considerably better than the sample Kendall rank correlation matrix when $K < d$. The new estimator can also be much more efficient in finite samples even in the unstructured case $K = d$, although there is no gain asymptotically. When the distribution of the data is elliptical, the results extend to linear correlation matrices and their inverses. The procedure is illustrated on financial stock returns.
\end{abstract}


\end{frontmatter}

\section{Introduction} \label{sec:intro}

Relationships between the components of a random vector $\X = (X_1,\dots, X_d)$ are of prime interest in many fields where statistical methods are used. Traditionally, this dependence  is summarized through a correlation matrix. When $\X$ is multivariate Normal, the classical choice is the linear correlation matrix. When multivariate Normality fails, as is frequent, e.g., in risk management, linear correlation can be grossly misleading and may not even exist \citep{Embrechts/McNeil/Straumann:2002}. For this reason, it is safer to use a rank correlation matrix such as  the matrix of  pair-wise Kendall's taus or Spearman's rhos.

In high dimensions, empirical correlation matrices typically conceal underlying dependence patterns. This is due to their sheer size and to the inherent imprecision of the estimates, especially when the sample size is small compared to dimension~$d$. 
For example, consider the log-returns of $107$ stocks included in the NASDAQ100 index from January 1 to September 30, 2017, giving $187$ observations. Hardly any pattern is visible in the left panel of Figure~\ref{fig:stocks-matrices}, which shows the empirical Kendall rank correlation matrix based on residuals from a fitted stochastic volatility model.

\begin{figure}[t]
	\centering
	\begin{minipage}{.05\textwidth}
		\centering
	   	\includegraphics[height=0.2\textheight]{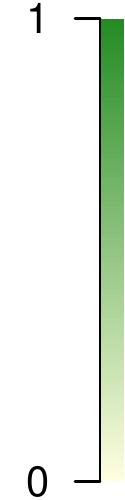}\\
	\end{minipage}
	\begin{minipage}{.9\textwidth}
		\centering
		\includegraphics[width=.28\textwidth]{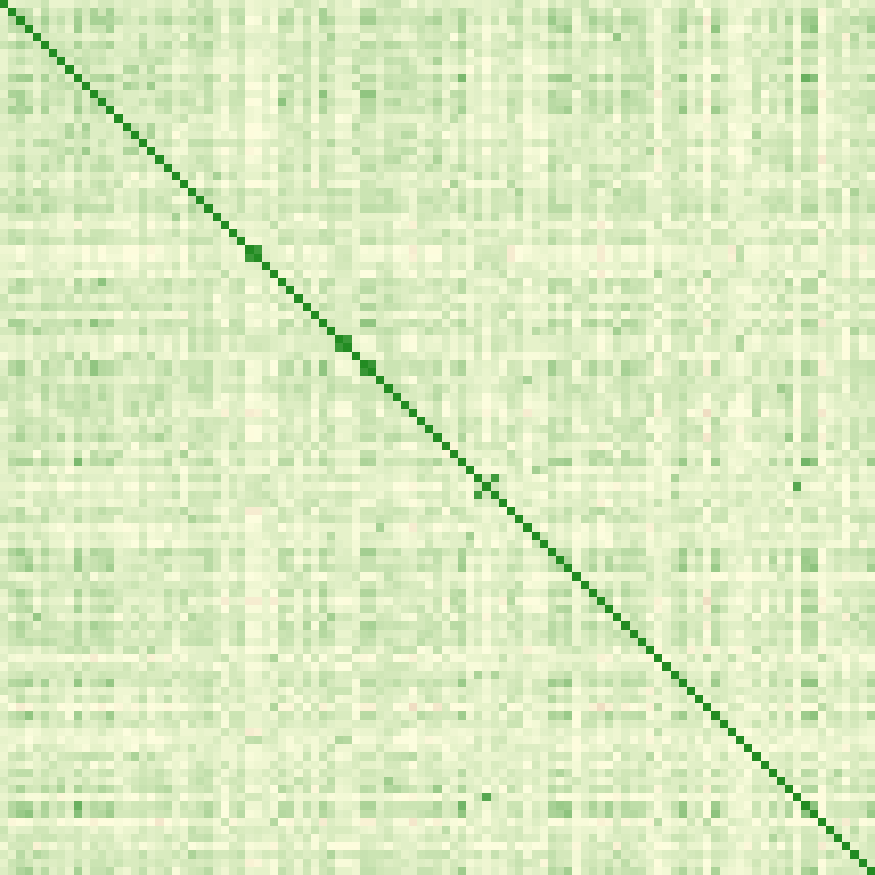} 
		\hspace{1mm}
		\includegraphics[width=.28\textwidth]{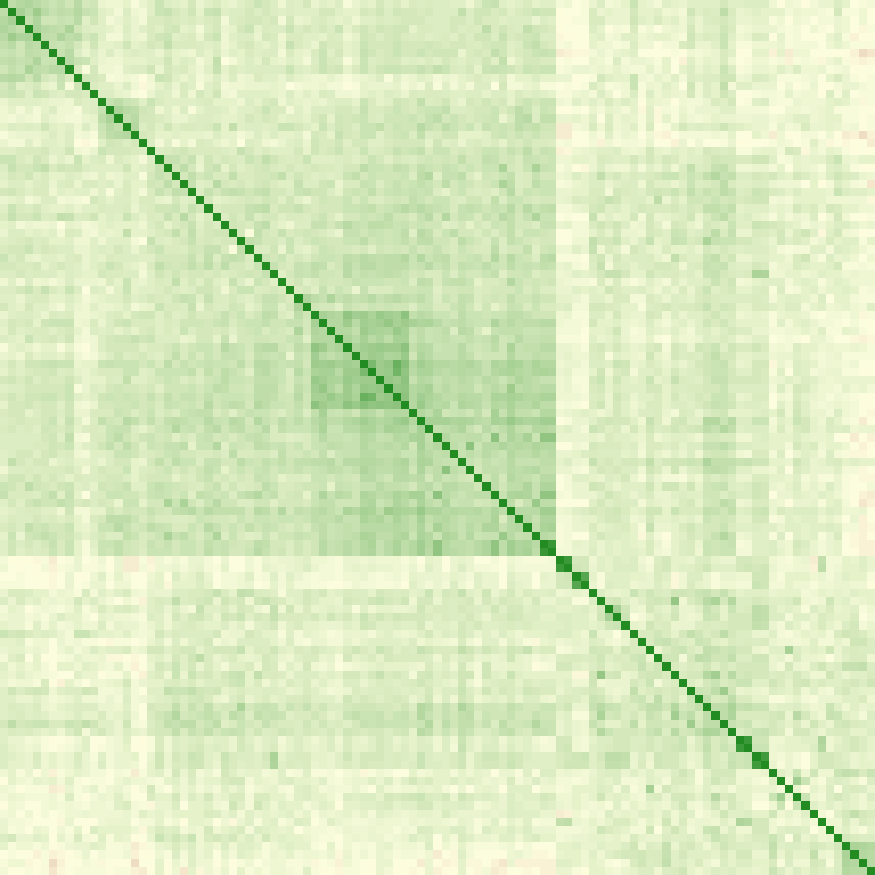}
		\hspace{1 mm}
		\includegraphics[width=.28\textwidth]{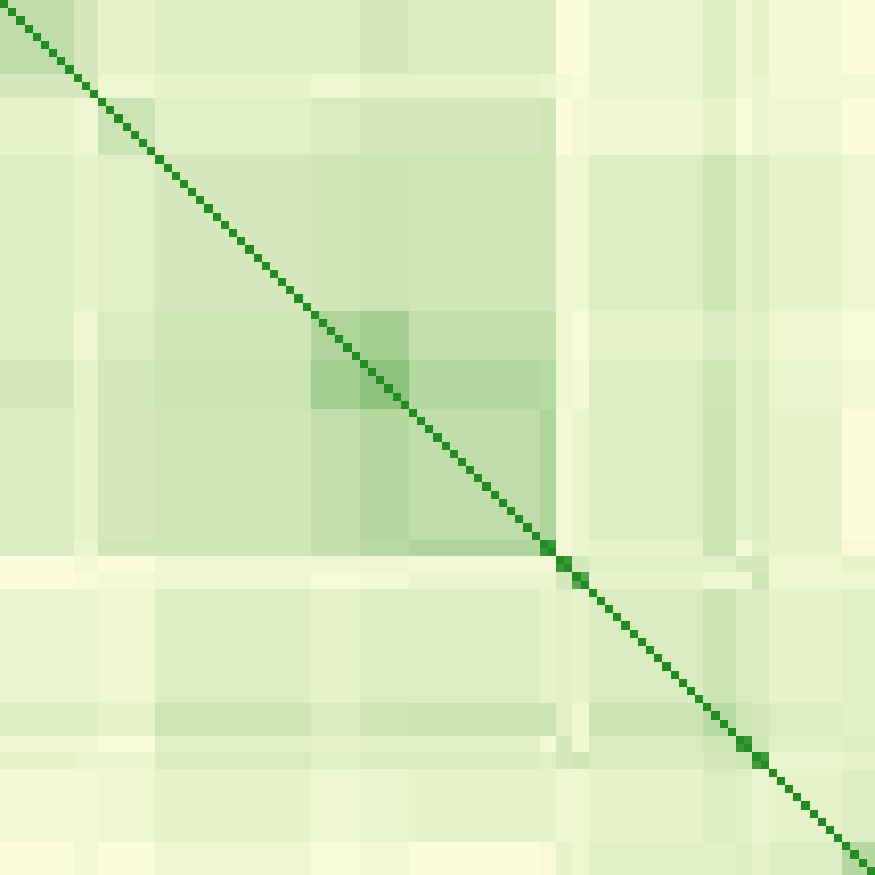}
	\end{minipage}
\caption{The empirical Kendall's tau matrix of $107$ stocks included in the NASDAQ100 index in the original labeling (left) and after relabeling (middle). The right panel shows the improved estimate obtained from Algorithm \ref{algo:path} and structure selection with $\alpha = .5$.}
\label{fig:stocks-matrices}
\end{figure}

Noisiness of sample correlation matrices is well documented. Several strategies have been proposed to remedy for it, most notably shrinkage \citep{Ledoit/Wolf:2004,Schafer/Strimmer:2005}. Alternative procedures developed in the context of graphical models consist of decomposing a noisy inverse covariance matrix into a low-rank matrix and a sparse matrix \citep{Chandrasekaran/Parrilo/Willsky:2012, Ma/Xue/Zou:2013,Agarwal/Negahban/Wainwright:2012}.

We follow a different path in this article. Motivated by the above NASDAQ example and many others, we focus on applications in which it makes sense to assume that the correlation matrix has a block structure. By this we mean that the variables can be grouped into $K$ disjoint clusters in such a way that for any two clusters $A$ and $B$, and any $X_i \in A$ and $X_j \in B$, the correlation between $X_i$ and $X_j$ satisfies $\rho(X_i, X_j) = \rho_{AB}$. In other words, all variables within each cluster are equicorrelated and the between-cluster correlation depends only on the clusters but not the representatives. This assumption is a way to reduce {the number of unknown pair-wise correlations} from $d(d-1)/2$ to at most $K(K+1)/2$. Correlation matrices with a block structure occur in portfolio and credit risk modeling, where variables can be grouped by industry classifications or risk types; see, e.g., the block-structured DECO model \citep{Engle/Kelly:2012}.  They also arise in the modeling of categorical, clustered, or gene expression data, and in genomics studies. In the NASDAQ100 example, a block structure emerges upon relabeling of the variables, as shown in the middle panel of Figure~\ref{fig:stocks-matrices}, though it is still noisy.

This article describes a technique for learning the cluster structure from data and shows how to use the latter to devise a more efficient estimator of the correlation matrix. No prior knowledge of the clusters, their number or composition is assumed. We only require that
the dependence within each cluster is exchangeable. The procedure we propose is an iterative algorithm  that is similar to, but different from, agglomerative clustering.  In contrast to model-based clustering which aims to cluster together observations from the same subpopulation of a multivariate mixture distribution, the current proposal aims at identifying elements of a correlation matrix that are equal. The algorithm also outputs an improved estimate of the correlation matrix which has a block structure, and an estimate of its asymptotic covariance matrix. In the above example of stock returns, the relabeling in the middle panel was done using the clusters identified through the proposed algorithm; the improved estimate of the correlation matrix is displayed in the right panel. As we prove asymptotically and illustrate via simulations, the improvement of the estimator can be substantial, in particular for $K \ll d$. Even in the unstructured case when $K=d$  and there is no gain asymptotically, the new estimator can perform substantially better in finite samples due to a bias-variance tradeoff, particularly when $n$ is small compared to $d$.

All procedures developed in this paper are based on the matrix $\T$  of pair-wise Kendall rank correlations, which turned out to be slightly more convenient than Spearman's rank correlation matrix, for reasons stated in Section~\ref{sec:model}. A~clear advantage of this approach over the linear correlation matrix is that Kendall's tau is margin-free, well-defined and well-behaved irrespective of the distribution of $\X$, making our methodology nonparametric and margin-free. In particular, it is not assumed that the variables in the same cluster are equally distributed; we only require that the marginal distributions are continuous. No Normality assumption is imposed either. However, when $\X$ is multivariate Normal, or more generally elliptical, there is a one-to-one relationship between $\T$ and the linear correlation matrix, and the improved estimator of $\T$ developed in this paper may be used to obtain more efficient estimators of the linear correlation matrix and its inverse.

Beyond the estimation of correlation itself, our procedure can be used as a first step in building complex dependence models. When $d$ is large, a model for the distribution of $\X$ needs to be both flexible and parsimonious. Within the Normal or  elliptical model, this means that the number of free parameters in the correlation matrix needs to be reduced, and the block structure identified through our algorithm can serve precisely this purpose. Outside the Normal model, dependence in $\X$ can be conveniently described through copulas, for joint distribution of $\X$ can be rewritten, for all $x_1,\ldots, x_d \in \mathbb{R}$, as
\begin{align} 
\label{eq:copula}
	\Pr(X_1 \leq x_1, \dots, X_d \leq x_d) = C\left\{ F_1(x_1),\ldots, F_d(x_d) \right\},
\end{align}
where $F_1,\ldots, F_d$ are the univariate marginal distributions of $\X$ and $C$ is a copula, i.e., a joint distribution function with standard uniform marginals \citep{Sklar:1959}. To achieve flexibility and parsimony when $d$ is large, the complexity of the problem needs to be reduced through an ingenious construction of $C$; examples are vines  \citep{Kurowicka/Joe:2011}, factor models \citep{Krupskii/Joe:2013,Hua/Joe:2017}, or hierarchical constructions \citep{Mai/Scherer:2012,Brechmann:2014,Joe:2015}. The cluster algorithm proposed in this paper is particularly well suited for such approaches: Equicorrelated clusters can first be identified through it and modeled by exchangeable lower-dimensional copulas. Dependence between clusters can then be achieved subsequently through vines or factors.~As such, this paper contributes to the emerging literature on structure learning for copula models; clustering algorithms have been employed very recently to learn the structure of nested (also called hierarchical) Archimedean copulas \cite{Okhrin/Okhrin/Schmid:2013,Gorecki/Hofert/Holena:2016,Gorecki/Hofert/Holena:2017a,Gorecki/Hofert/Holena:2017b}.

The article is organized as follows. In Section~\ref{sec:model}, we discuss the partial exchangeability assumption and its implications. In Section~\ref{sec:improved}, we construct an improved estimator of $\T$ assuming a known cluster structure, derive its asymptotic distribution, and show that it has smaller asymptotic variance than the empirical Kendall rank correlation matrix when $K < d$. The algorithm through which $K$ and the cluster structure can be learned from data, and which outputs an improved estimator of $\T$, is then introduced in Section~\ref{sec:learning}. In Section~\ref{sec:bijective}, we discuss implications for the estimation of the linear correlation matrix.  The new methodology is studied through simulations in Section~\ref{sec:simulation} and illustrated on NASDAQ100 stock returns in Section~\ref{sec:stock}. Section~\ref{sec:conclusion} concludes the paper. The estimation of the covariance matrix of the estimator of $\T$ and proofs are relegated to the Appendix. Additional material from the simulation study and the data illustration as well as  the link to the \textsf{R}-code to implement the method may be found in the Online Supplement.

\section{Partial exchangeability assumption} \label{sec:model}

Throughout, let $\X=(X_1,\dots,X_d)$ be a random vector with continuous univariate marginals, denoted $F_1,\ldots, F_d$. In this case, the copula $C$ in Sklar's decomposition \eqref{eq:copula} is unique; in fact, it is the joint distribution of the vector $(F_1(X_1),\dots, F_d(X_d))$. The following partial exchangeability assumption plays a central role in this paper.
\begin{assumption*}
\namedlabel{ass:main}{PEA} 
For $j\in\{1,\dots, d\}$, let $U_j = F_j(X_j)$. A partition $\mathcal{G} = \{\mathcal{G}_1, \dots, \mathcal{G}_{K}\}$ of $\{ 1,\dots,d \}$ satisfies the Partial Exchangeability Assumption (PEA) if for any $u_1,\dots, u_d \in [0,1]$ and any permutation $\pi$ of $1,\dots, d$ such that for all $j \in\{1,\dots, d\}$ and all $k \in \{1,\dots, K\}$, $ j \in \mathcal{G}_k$ if and only if $\pi(j) \in \mathcal{G}_k$, one has
\begin{align*}
C(u_1, \dots, u_d) = C(u_{\pi(1)}, \dots, u_{\pi(d)})
\end{align*}
or equivalently, $(U_1, \dots, U_d) \eqdis (U_{\pi(1)}, \dots, U_{\pi(d)})$, where $\eqdis$ denotes equality in distribution.
\end{assumption*}

To understand the \ref{ass:main}, note first that the partition $\mathcal{G}=\{\{1,\dots,d\}\}$ with $K = 1$ satisfies the \ref{ass:main} only if $C$ is fully exchangeable, meaning that for all $u_1,\dots, u_d \in [0,1]$ and any permutation $\pi$ of $1,\dots, d$, $C(u_1,\dots, u_d) = C(u_{\pi(1)}, \dots, u_{\pi(d)})$; examples of fully exchangeable copulas are Gaussian or Student $t$ with an equicorrelation matrix, and all Archimedean copulas. When $K > 1$, the \ref{ass:main} is a weaker version of full exchangeability. A partition $\mathcal{G}$ for which \ref{ass:main} holds divides $X_1,\dots, X_d$ into clusters such that the copula $C$ is invariant under within-cluster permutations. In particular, for any $k \in \{1,\dots, K\}$, the copula of $(X_j, j \in \mathcal{G}_k)$ is fully exchangeable. In contrast to many standard clustering contexts, the \ref{ass:main} does not imply that variables in the same cluster are equally distributed, because it only concerns the copula $C$ of $\X$ and not the marginals $F_1,\dots, F_d$. {As explained in Section~\ref{sec:learning}, while the \ref{ass:main} may hold for several partitions, the coarsest partition $\mathcal{G}$ that satisfies the \ref{ass:main} is unique}, viz.
\begin{equation}
\label{eq:coarsestG}
\mathcal{G} = \argmin_{\mathcal{G}^* \text{ satisfies the \ref{ass:main}}} (|\mathcal{G}^*|).
\end{equation}

Finally, observe that the \ref{ass:main} is not restrictive in any way. In the completely unstructured case, $\mathcal{G}$ in Eq.~\eqref{eq:coarsestG} is $\{\{1\},\dots, \{d\}\}$. However, as we demonstrate in Section~\ref{sec:simulation}, there may still be a substantial advantage in considering a coarser partition for inference purposes in finite samples. When the partition $\mathcal{G}$ in Eq.~\eqref{eq:coarsestG} is such that $|\mathcal{G}| = K < d$, {the number of distinct rank correlations is reduced from $d(d-1)/2$ to 
\begin{align}
\label{eq:nblocks}
L = K(K-1)/2 + \sum_{i=1}^K \boldsymbol{1}(|\mathcal{G}_k| > 1).
\end{align}
}
Several commonly used models make an implicit or explicit use of this kind of {complexity reduction}; examples include latent variable models such as frailty or random effects models, Markov random fields or graphical models, hierarchical copula models \citep{Brechmann:2014,Mai/Scherer:2012}, factor copulas \citep{Gregory/Laurent:2004,Krupskii/Joe:2015}, or nested (hierarchical) Archimedean copulas \citep{Joe:2015}. 
\begin{definition}
For a partition $\mathcal{G}$ that satisfies the \ref{ass:main}, we write $X_i \sim X_j$ whenever $i,j\in \mathcal{G}_k$ for some $k \in \{1,\dots, K\}$. Furthermore, the cluster membership matrix $\D$ is a $d\times d$ matrix whose $(i,j)$th entry is given, for all $i,j \in \{1,\dots, d\}$, by $\Delta_{ij}= \mathbf{1}(X_i \sim X_j)$.
\end{definition}
Next, let $\T$ be the $d\times d$ matrix of pair-wise Kendall correlation coefficients. Specifically, for all $i,j \in \{1,\dots, d\}$, the $(i,j)$th entry of $\T$ is the population version of Kendall's tau between $X_i$ and $X_j$, viz.
\begin{align*}
\T_{ij} = \tau(X_i,X_j) =\Pr\{(X_i - X_i^*)(X_j - X_j^*) > 0\} - \Pr\{(X_i - X_i^*)(X_j - X_j^*) < 0\},
\end{align*}	
i.e., the difference between the probabilities of concordance and discordance of $(X_i, X_j)$ and its independent copy $(X_i^*, X_j^*)$. Because $\tau(X_i, X_j)$ depends only on the copula $C_{ij}$ of $(X_i, X_j)$, viz.
\begin{align}\label{eq:taucop}
\tau(X_i, X_j) =	-1+4 \int C_{ij}(u_i,u_j)\ \d C_{ij}(u_i,u_j),
\end{align}
see, e.g., \cite{Nelsen:1999}, it is not surprising that under the \ref{ass:main}, several entries in $\T$ are identical. This is specified in the next result, which follows directly from the \ref{ass:main} and Eq.~\eqref{eq:taucop}.
\begin{proposition} \label{prop:T-equalities}
Suppose that the partition $\mathcal{G}$ of $\{1,\dots, d\}$ satisfies the \ref{ass:main} and that $X_{i_1} \sim X_{i_2}$ and $X_{j_1} \sim X_{j_2}$, where $i_1 \neq j_1$ and $i_2 \neq j_2$. Then the copulas $C_{i_1j_1}$ and $C_{i_2j_2}$ are identical and, consequently, $\T_{i_1j_1} = \T_{i_2j_2}$.
\end{proposition}
\begin{remark}\em
It follows from Proposition \ref{prop:T-equalities} that whenever $X_{i_1} \sim X_{i_2}$ and $X_{j_1} \sim X_{j_2}$ where $i_1 \neq j_1$ and $i_2 \neq j_2$, any copula-based measure of association $\kappa$ will satisfy $\kappa(X_{i_1}, X_{j_1}) = \kappa(X_{i_2}, X_{j_2})$; examples of $\kappa$ include Spearman's rho or Gini's gamma \citep{Nelsen:1999}. The reason why we focus on Kendall's tau in this paper is that the asymptotic and finite-sample variance of the empirical Kendall's tau matrix have a tractable form, more so for instance than that of Spearman's rho \citep{Borkowf:2002}. Having said that, the procedures proposed here could in principle be extended to other measures of association.
\end{remark}

Suppose now that a partition $\mathcal{G}$ with $|\mathcal{G}| > 1$ satisfies the \ref{ass:main}. Proposition \ref{prop:T-equalities} then implies that when $X_i \sim X_j$ for some $i \neq j$, the $i$th and $j$th rows and columns in $\T$ are identical, once the diagonal entries are aligned. Consequently, if the variables are relabeled so that the clusters are contiguous, then the cluster membership matrix $\D$ is block-diagonal and $\T$ is a block matrix. As an illustration, consider the following example, which we shall use throughout the paper.
\begin{figure}[t!]
	\centering
	\begin{minipage}{.05\textwidth}
	   	\includegraphics[height=0.18\textheight]{scale-short}\\
	\end{minipage}
	\hspace{2mm}
	\begin{minipage}{.225\textwidth}
	    \centering
    	    \includegraphics[width=1\linewidth]{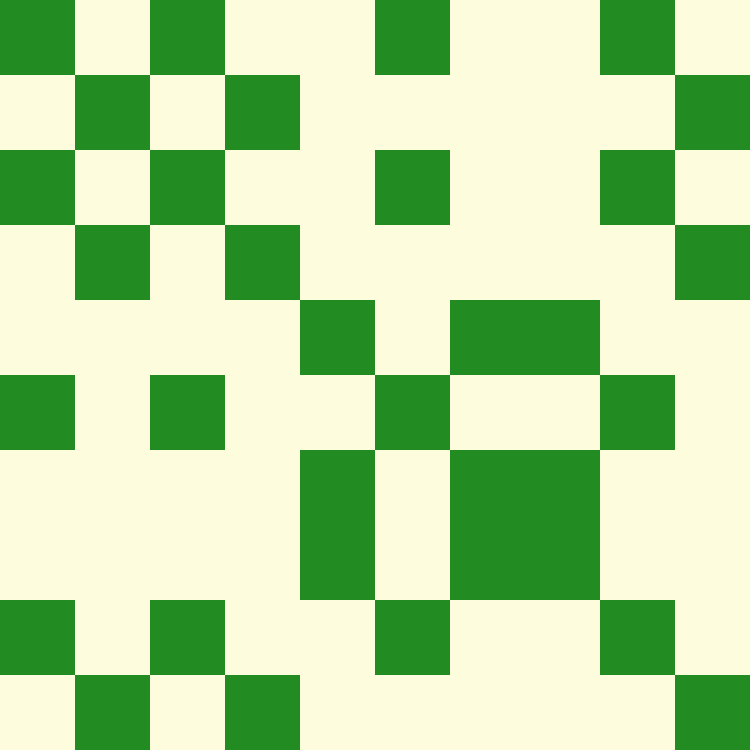}
    	    $\D^*$	
	\end{minipage}
	\begin{minipage}{.225\textwidth}
	    \centering
    	    \includegraphics[width=1\linewidth]{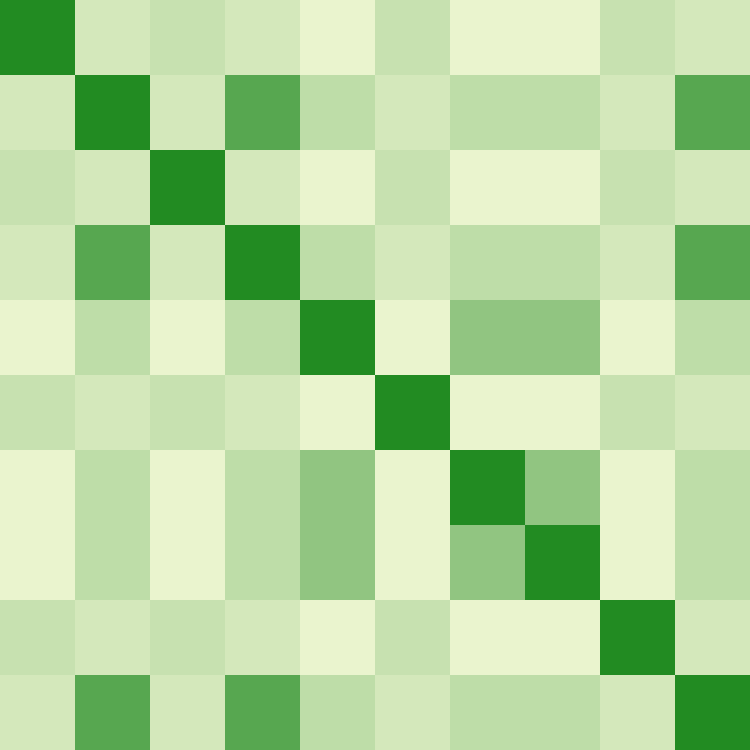}
    	    $\T^*$
	\end{minipage}
	\begin{minipage}{.225\textwidth}
	    \centering
    	    \includegraphics[width=1\linewidth]{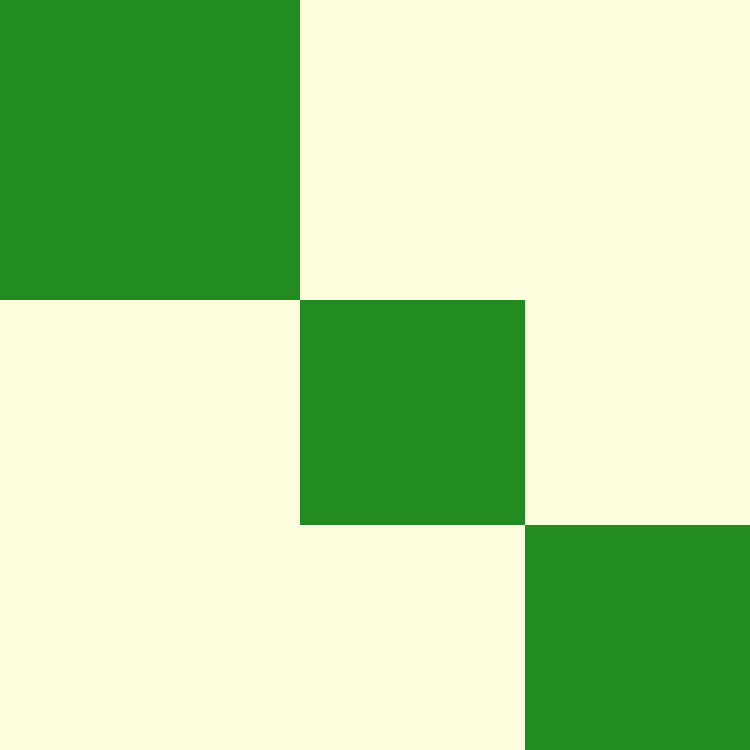}
    	    $\D$
	\end{minipage}
	\begin{minipage}{.225\textwidth}
	    \centering
    	    \includegraphics[width=1\linewidth]{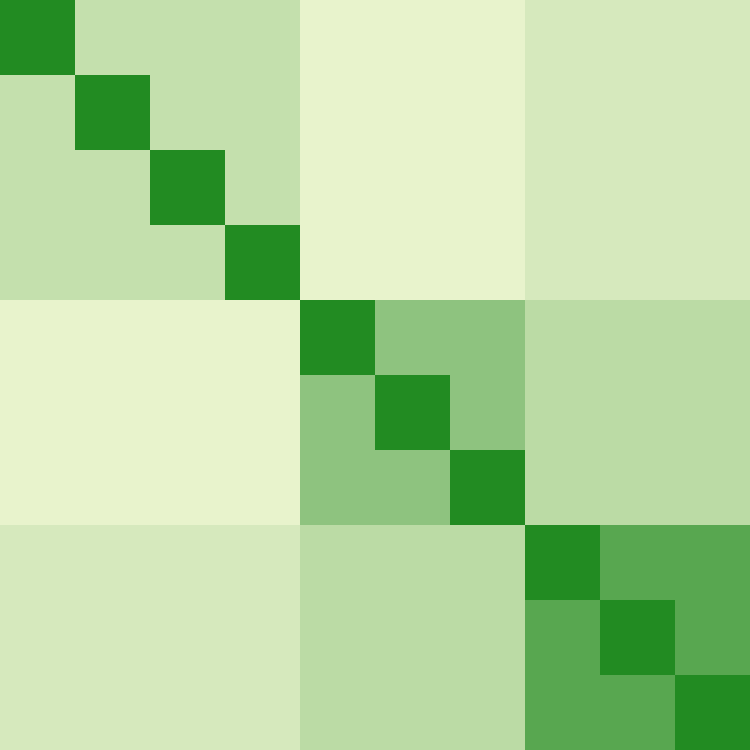}
    	    $\T$
	\end{minipage}
	\caption{Cluster membership and Kendall correlation matrices before ($\D^*$ and $\T^*$) and after ($\D$ and $\T$) relabeling of the variables.} \label{fig:ex-T}
\end{figure}

\begin{example}\label{ex:1}
\em
Consider a random vector $\X^*$ of dimension $d=10$ such that the partition $\mathcal{G}^*=\{\mathcal{G}_1^*, \mathcal{G}_2^*, \mathcal{G}_3^*\}$ of size $K = 3$ given by $\mathcal{G}_1^* = \{ 1, 3, 6, 9 \}$, $\mathcal{G}_2^* = \{ 5, 7, 8 \}$, and $\mathcal{G}_3^* = \{ 2, 4, 10 \}$ satisfies the \ref{ass:main}. The corresponding cluster membership matrix $\D^*$ and the matrix $\T^*$ of pair-wise Kendall correlations are shown in Figure~\ref{fig:ex-T}; the exact distribution of $\X^*$ does not matter at this point. In this case, the clusters are not contiguous, i.e., $\D^*$ is not block diagonal, and the block structure of $\T^*$ is not easily seen.  Once the variables are relabeled as $\X = (X_1,\dots, X_{10}) = (X^*_1, X^*_3, X^*_6, X^*_9,X^*_5,X^*_7, X^*_8, X^*_2, X^*_4, X^*_{10})$, the partition that satisfies the \ref{ass:main} becomes $\mathcal{G} = \{\mathcal{G}_1, \mathcal{G}_2, \mathcal{G}_3\}$, where
\begin{align} 
\label{eq:ex-G}
	\mathcal{G}_1 = \{1,2,3,4\}, \quad \mathcal{G}_2 = \{5,6,7\}, \quad  \mathcal{G}_3 = \{8,9,10\}.
\end{align}
The clusters are now contiguous, $\D$ is block-diagonal, and $\T$ has an apparent block structure; see Figure~\ref{fig:ex-T}. Every time we revisit this example, we work with the relabeled vector $\X$.
\end{example}
Although we use examples in which the matrix $\D$ is block-diagonal for illustrative purposes, contiguity of the clusters is not required.  In fact, given that $\mathcal{G}$ is unknown, the variables are unlikely to be labeled so that $\T$ has an apparent block structure. To describe the latter, we first need additional notation. To this end, let $\R$ be an arbitrary symmetric $d\times d$ matrix. The entries above the main diagonal can be stacked in a vector, say $\boldsymbol{\rho}$, of length $p = d(d-1)/2$. Note that the diagonal elements of $\R$ play no role at this point. The particular way the vectorization is done is irrelevant, as long as it is the same throughout. For example, one may use the lexicographical ordering viz.
\begin{align} 
\label{eq:vectorization}
\boldsymbol{\rho} = 
(\R_{12} , \dots, \R_{1d} ,\R_{23} ,\dots,\R_{2d} , \dots, \R_{(d-1) d})^\top.
\end{align}
For  arbitrary $r \in \{1,\dots, p\}$, $(i_r,j_r)$ refers to the pair of indices $i_r < j_r$ such that $\boldsymbol{\rho}_r = \R_{i_rj_r}$. Now any partition $\mathcal{G} = \{ \mathcal{G}_1,\dots, \mathcal{G}_K\}$ of $\{1,\dots, d\}$ induces a partition of the elements of $\boldsymbol{\rho}$, or, equivalently, of $\{1,\dots, p\}$.
For any $k_1 \le k_2 \in \{1,\dots, K\}$, let
\begin{align}
\label{eq:Bk}
	\mathcal{B}_{k_1k_2} &= \{ r \in \{1,\dots, p\} : (i_r,j_r) \in (\mathcal{G}_{k_1} \times \mathcal{G}_{k_2}) \cup (\mathcal{G}_{k_2} \times \mathcal{G}_{k_1}) \}.
\end{align}
Note that the total number of nonempty blocks $\mathcal{B}_{k_1k_2}$ is $L$ given in Eq.~\eqref{eq:nblocks}
because when $k_1 = k_2$, $\mathcal{B}_{k_1k_2}$ is nonempty only if $|\mathcal{G}_{k_1} | > 1$.
Referring to the sets $\mathcal{B}_{k_1k_2}$ using a single index, the partition of $\{1,\dots, p\}$ is then given by
\begin{align} 
\label{eq:B-cal}
	\mathcal{B}_{\mathcal{G}} = \{\mathcal{B}_1, \ldots, \mathcal{B}_L \}.
\end{align}
In analogy to the cluster membership matrix $\D$, we define a $p \times L$ block membership matrix $\B$; for all $r \in \{1,\dots, p\}$ and $\ell\in \{1,\dots, L\}$, its $(r,\ell)$th entry is given~by
$\B_{r \ell} = \boldsymbol{1}(r \in \mathcal{B}_\ell)$. 
Finally, define the set $\mathcal{T}_{\mathcal{G}}$ of all symmetric matrices with a block structure given by $\mathcal{B}_{\mathcal{G}}$, viz.
\begin{align} 
\label{eq:T-cal}
	\mathcal{T}_{\mathcal{G}} = \{\R \in \mathbb{R}^{d \times d}: \R \text{ symmetric and } \forall_{\ell \in\{1,\dots, L\}} \ r,s \in \mathcal{B}_{\ell} \;\; \Rightarrow \;\; \R_{i_rj_r} = \R_{i_sj_s} \}.
\end{align}
Note that only the elements of $\R$ that are above the main diagonal enter the definition of $\mathcal{T}_{\mathcal{G}}$ in Eq.~\eqref{eq:T-cal}.

Now suppose that $\mathcal{G}$ is such that the \ref{ass:main} holds and that the elements above the main diagonal of $\T$ are stacked in $\t$.  By Proposition \ref{prop:T-equalities}, for any $\ell \in \{1,\dots, L\}$ and $r,s \in \mathcal{B}_{\ell}$, $\t_r = \t_s$, or, equivalently, $\T_{i_r j_r} = \T_{i_sj_s}$. This means that $\T \in \mathcal{T}_{\mathcal{G}}$; when no confusion can arise, we will also write $\t \in \mathcal{T}_{\mathcal{G}}$. Consequently, there are only $L$ distinct elements in $\t$. Storing these in a vector $\t^* \in [-1,1]^{L}$, we thus have
$\t = \B \t^*$.
This means that when \ref{ass:main} holds, the number of free parameters in $\T$ is reduced from $d(d-1)/2$ to $L$ given in Eq.~\eqref{eq:nblocks}. We revisit Example \ref{ex:1} to illustrate.
\begin{example} \em
Consider the matrix $\T$ corresponding to $\X$ in Example \ref{ex:1}, and stack it in a vector $\t$ of length $p=45$ as in Eq.~\eqref{eq:vectorization}. Because there are $K=3$ clusters given in Eq.~\eqref{eq:ex-G}, $\t^*$ has length $L = 6$ and the cluster structure $\mathcal{G}$ reduces the number of free parameters in $\T$ from $45$ to $6$. The six distinct blocks are visible in the right panel of Figure~\ref{fig:ex-T}.
\end{example}

\section{Improved estimation of $\T$} \label{sec:improved}

Suppose that  $\X_1, \ldots ,\X_n$ is a random sample from $\X$ and, for each $i \in \{ 1, \ldots, n\}$, set  $\X_i = (X_{i1},\ldots,X_{id})^\top$. The classical nonparametric estimator of $\T$ is $\Th$; for $i,j \in \{1,\dots, d\}$, its $(i,j)$th entry is given by
\begin{align*} 
	\Th_{ij} = - 1 + \frac{4}{n(n-1)} \sum_{r \neq s} \boldsymbol{1}(X_{ri} \leq X_{si}) \boldsymbol{1}(X_{rj} \leq X_{sj}).
\end{align*}
As explained in Section~\ref{sec:model}, if the \ref{ass:main} holds for some partition $\mathcal{G}$ with $|\mathcal{G}| < d$, the number of free parameters in $\T$ reduces from $d(d-1)/2$ to $L$. We now show that an a priori knowledge of $\mathcal{G}$ leads to a more efficient estimator of $\T$.

Recall first that for all $i \neq j \in \{1,\dots, d\}$, $\Th_{ij}$ is a $U$-statistic and thus unbiased and asymptotically Normal \citep{Hoeffding:1948}. The behavior of $\Th$ was studied in \cite{ElMaache/Lepage:2003} and \cite{Genest/Neslehova/BenGhorbal:2011}; results pertaining to the closely related coefficient of agreement appear in \cite{Ehrenberg:1952}. If $\t$ and $\th$ denote the vectorized versions of $\T$ and $\Th$ respectively, one has, as $n \to \infty$,
\begin{align}
\label{eq:asstau}
\sqrt{n} \, (\th - \t) \rightsquigarrow \mathcal{N}(\boldsymbol{0}_p,\S_\infty),
\end{align}
where $\rightsquigarrow$ denotes convergence in distribution and $\boldsymbol{0}_p$ is the $p$-dimensional vector of zeros. Expressions for the asymptotic variance $\S_\infty$ as well as the finite-sample variance $\S$ of $\th$ have been derived in \cite{Genest/Neslehova/BenGhorbal:2011}.

The asymptotic Normality of $\th$ suggests to base inference on the following loss function $\mathcal{L} : [-1,1]^d \to [0,\infty)$,
\begin{align} 
\label{eq:loss}
	\mathcal{L}(\boldsymbol{t} \mid \th, \S) = (\th - \boldsymbol{t})^\top \S^{-1} (\th - \boldsymbol{t}).
\end{align}
$\mathcal{L}$ is the Mahalanobis distance between $\boldsymbol{t}$ and $\th$ that accounts for the heterogeneous variability of the entries of $\th$. The fact that the finite-sample variance $\S$ is unknown is irrelevant for now; it will only become a concern in Section~\ref{sec:learning}.

Considering an arbitrary $\boldsymbol{t} \in [-1,1]^d$, it is obvious that $\mathcal{L}$ attains its minimum at $\th$ since $\mathcal{L}(\boldsymbol{t}|\th,\S) \geq 0 = \mathcal{L}(\th|\th,\S)$. Now suppose that $\mathcal{G}$ is a partition of $\{1,\dots, d\}$ such that the \ref{ass:main} holds. Unless $|\mathcal{G}|=d$, it is extremely unlikely that $\Th \in \mathcal{T}_{\mathcal{G}}$. However, we can introduce these structural constraints implied by $\mathcal{G}$ into the estimation procedure.
\begin{theorem} \label{thm:maxlike}
Suppose that $\mathcal{G}$ is a partition of $\{1,\dots, d\}$ such that the \ref{ass:main} holds. Let $\B$ be the block membership matrix corresponding to $\mathcal{B}_{\mathcal{G}}$ in Eq.~\eqref{eq:B-cal}. Then for $\mathcal{L}(\boldsymbol{t}|\th, \S)$ as in Eq.~\eqref{eq:loss},
\begin{align} 
\label{eq:maxlike}
\tt(\th \mid \mathcal{G}) = \argmin_{\boldsymbol{t} \in \mathcal{T}_{\mathcal{G}}} \mathcal{L}(\boldsymbol{t} \mid \th, \S) = \G \th,
\end{align}
where $\G = \B \B^+$ and ${^+}$ denotes the Moore--Penrose generalized inverse. Furthermore, for any $\ell \in \{1,\dots, L\}$ and $r \in \mathcal{B}_\ell$, $\tt(\th|\mathcal{G})_r = \sum_{s \in \mathcal{B}_\ell} {\th}_s / |\mathcal{B}_\ell |$.
\end{theorem}
\begin{remark}\em
The block-wise averages that appear in $\tt$ are akin to the pair-wise linear correlation averaging from \cite{Elton/Gruber:1973} and \cite{Ledoit/Wolf:2003a}, and more particularly the block DECO in \cite{Engle/Kelly:2012}, which uses block-wise averages of linear correlations. {Averaging of pair-wise Kendall's tau in order to fit nested Archimedean copula models has recently been employed in  \cite{Okhrin/Tetereva:2017}.} Also note that the \ref{ass:main} plays a crucial role in the proof of Theorem \ref{thm:maxlike}. It is needed to invoke Proposition \ref{prop:S-inv} and Lemma~\ref{lem:maxlike} and obtain a finite-sample variance matrix $\S$ with a structure that will lead to the required simplifications. In particular, it is not enough to assume that $\t \in \mathcal{T}_{\mathcal{G}}$ for some partition $\mathcal{G}$.
\end{remark}

\begin{figure}[t!]
\centering
	\begin{minipage}{.05\textwidth}
	   	\includegraphics[height=0.2\textheight]{scale-short}\\
	\end{minipage}
	\hspace{2mm}
	\begin{minipage}{.25\textwidth}
	    \centering
    	    \includegraphics[width=1\linewidth]{example-T}\\
    	    $\T$	
	\end{minipage}
	\begin{minipage}{.25\textwidth}
	    \centering
    	    \includegraphics[width=1\linewidth]{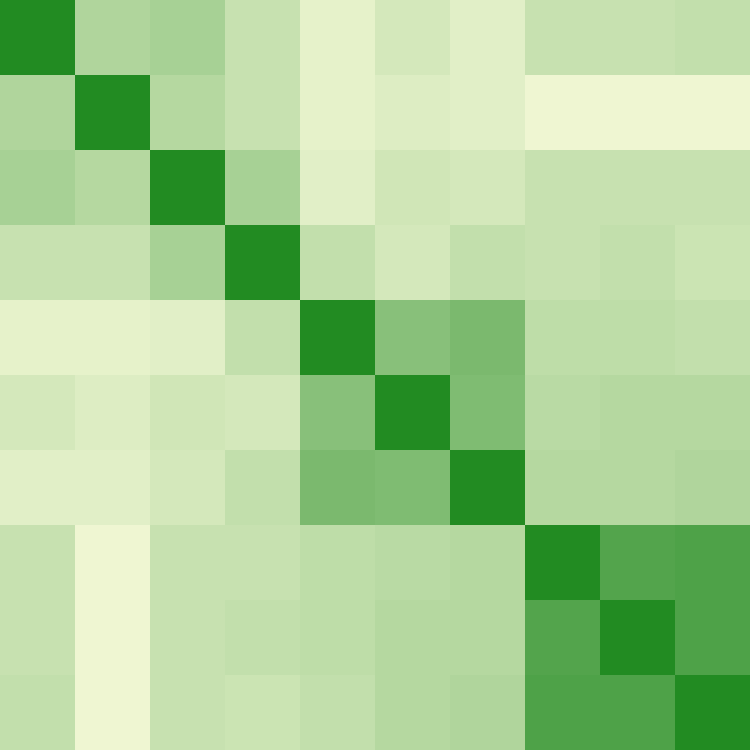}	\\
    	    $\Th$
	\end{minipage}
	\begin{minipage}{.25\textwidth}
	    \centering
    	    \includegraphics[width=1\linewidth]{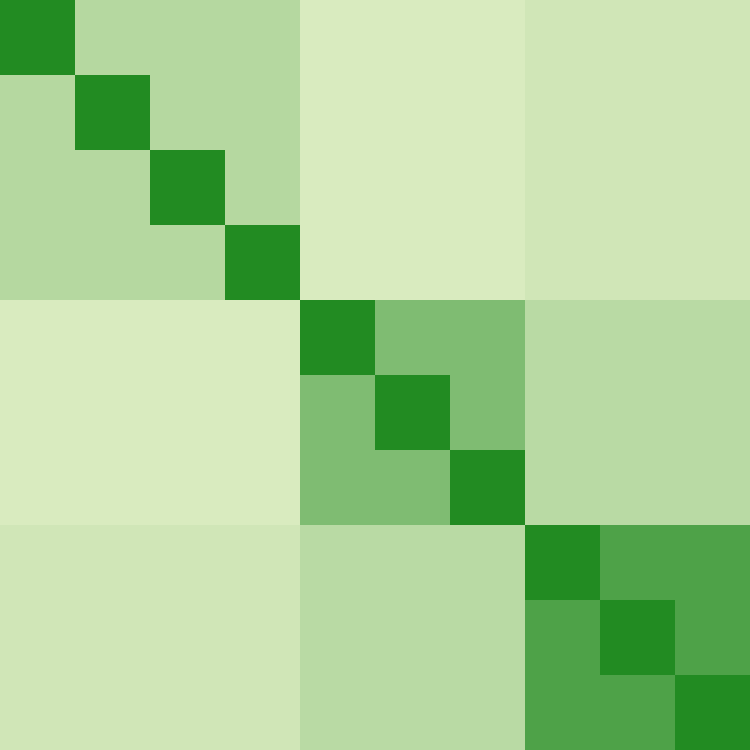}	\\
    	    $\Tt$
	\end{minipage}
	\caption{The matrices $\T$, $\Th$ and $\Tt$ in Example~\ref{ex:3}.} \label{fig:ex-T-tilde}
\end{figure}

When it introduces no confusion, we refer to $\tt(\th|\mathcal{G})$ as $\tt$  and to its matrix version $\Tt(\Th|\mathcal{G})$ as $\Tt$. What is crucial in Theorem \ref{thm:maxlike} is that  $\tt$ consists of the cluster averages of the elements of $\th$ and as such does not involve the unknown finite-sample variance $\S$ of $\th$, so an estimator of $\S$ is not needed to compute $\tt$. The information contained in $\mathcal{G}$ is introduced by projecting $\Th$ onto $\mathcal{T}_{\mathcal{G}}$. The resulting estimator $\Tt$ is expected to be closer to the original matrix $\T$ because the entries that estimate a same value are averaged over, thus reducing the estimation variance.  In fact, for any $r \in \{1,\dots, p\}$, the asymptotic variance of $\tt_r$ is less than or equal to that of $\th_r$ as a result of the following theorem.
\begin{theorem} \label{thm:reduced-variance}
Suppose that $\mathcal{G}$ is a partition of $\{1,\dots, d\}$ such that the \ref{ass:main} holds. Let $\B$ be the block membership matrix corresponding to $\mathcal{B}_{\mathcal{G}}$ in Eq.~\eqref{eq:B-cal}, and $\G =\B \B^+$ and $\tt = \G\th$ be as in Theorem~\ref{thm:maxlike}. Then the following statements hold:
\begin{itemize}
\item[(i)] As $n\to \infty$, $\sqrt{n} \, (\tt - \t) \rightsquigarrow \mathcal{N}\left( \mathbf{0}, \G\S_{\infty} \right)$.
\item[(ii)] The matrix $\S_\infty-\G\S_\infty$ is nonnegative definite.
\end{itemize}
\end{theorem}
To conclude this section, we illustrate $\tt$ using the setup in Example \ref{ex:1}.
\begin{example}\label{ex:3}\em
Consider a random sample of size $n=70$ from the vector $\X$ in Example \ref{ex:1}; we used $\X$ to be Normally distributed with Kendall correlation matrix $\T$ displayed in Figure~\ref{fig:ex-T}. Figure~\ref{fig:ex-T-tilde} displays $\T$, $\Th$ and $\Tt$. For this one simulated sample, it is clear that $\Tt$ is a better estimate of $\T$ than $\Th$.
\end{example}

\section{Learning the structure $\bs{\mathcal{G}}$} \label{sec:learning}

Because the cluster structure $\mathcal{G}$ is typically unknown, the improved estimator $\tt$ derived in Section~\ref{sec:improved} cannot be directly used. We now propose a way to learn $\mathcal{G}$ from data and to obtain an improved estimator of $\t$ as a by-product. To do so, we first identify $d$ candidate structures in Section~\ref{subsec:alg1} and then choose one among them in Section~\ref{subsec:selection}.

\subsection{Creating a set of candidate structures} \label{subsec:alg1}

If the \ref{ass:main} holds for some partition $\mathcal{G}$, it holds for any refinement thereof, defined below.
\begin{definition}\label{def:condition}
Let $\mathcal{G} = \{ \mathcal{G}_1,\dots, \mathcal{G}_K\}$ be a partition of $\{1,\dots, d\}$. A refinement of $\mathcal{G}$  is a partition $\mathcal{G}^* = \{ \mathcal{G}_1,\dots, \mathcal{G}_{K^*}\}$ of $\{1,\dots, d\}$ such that $K^* > K$ and 
for any $k^* \in \{1,\dots, K^*\}$ there exists $k \in \{1,\dots, K\}$ such that $\mathcal{G}_{k^*} \subseteq \mathcal{G}_{k}$.
\end{definition}
The block structure implied by a refinement of $\mathcal{G}$ is consistent with the block structure implied by $\mathcal{G}$. This is formalized in the next proposition, which follows easily from Eq.~\eqref{eq:T-cal}.
\begin{proposition} \label{prop:refinement}
For partitions $\mathcal{G}$, $\mathcal{G}^*$ of $\{1,\dots, d\}$, $\mathcal{G}^*$ is a refinement of $\mathcal{G}$ if and only if $\mathcal{T}_{\mathcal{G}} \subseteq \mathcal{T}_{\mathcal{G}^*}$.
\end{proposition}
Proposition \ref{prop:refinement} implies that if $\T \in \mathcal{T}_{\mathcal{G}}$ for some partition $\mathcal{G}$, then for any refinement $ \mathcal{G}^*$ thereof, $\T \in \mathcal{T}_{\mathcal{G}^*}$. 
\begin{example}\label{ex:4}\em
Consider the partition $\mathcal{G}$ given in Eq.~\eqref{eq:ex-G} in Example \ref{ex:1}. The partition $\mathcal{G}^* = \{\mathcal{G}_1^*,\dots, \mathcal{G}_4^*\}$ with $\mathcal{G}_1^* = \{1,2 \}$, $\mathcal{G}_2^* = \{3,4 \}$,  $\mathcal{G}_3^* = \{5,6,7 \}$,  $\mathcal{G}_4^* = \{ 8,9,10 \}$ is a refinement of $\mathcal{G}$ since $\mathcal{G}_1^*, \mathcal{G}_2^* \subseteq \mathcal{G}_1$, $\mathcal{G}_3^* \subseteq \mathcal{G}_2$ and $\mathcal{G}_4^* \subseteq \mathcal{G}_3$. Consequently, $\mathcal{G}^*$ satisfies the \ref{ass:main} as well. Figure~\ref{fig:ex-T-tilde-sub} shows the cluster membership matrices $\D$ and $\D^*$ corresponding to $\mathcal{G}$ and $\mathcal{G}^*$, respectively. Also displayed are the estimates $\Tt(\Th|\mathcal{G}^*)$ and $\Tt(\Th |\mathcal{G})$; one can see that the block structure of the former is embedded in the latter but not conversely.
\begin{figure}[t!]
	\centering
	\begin{minipage}{.05\textwidth}
	   	\includegraphics[height=0.18\textheight]{scale-short}\\
	\end{minipage}
	\hspace{2mm}
	\begin{minipage}{.225\textwidth}
	    \centering
    	    \includegraphics[width=1\linewidth]{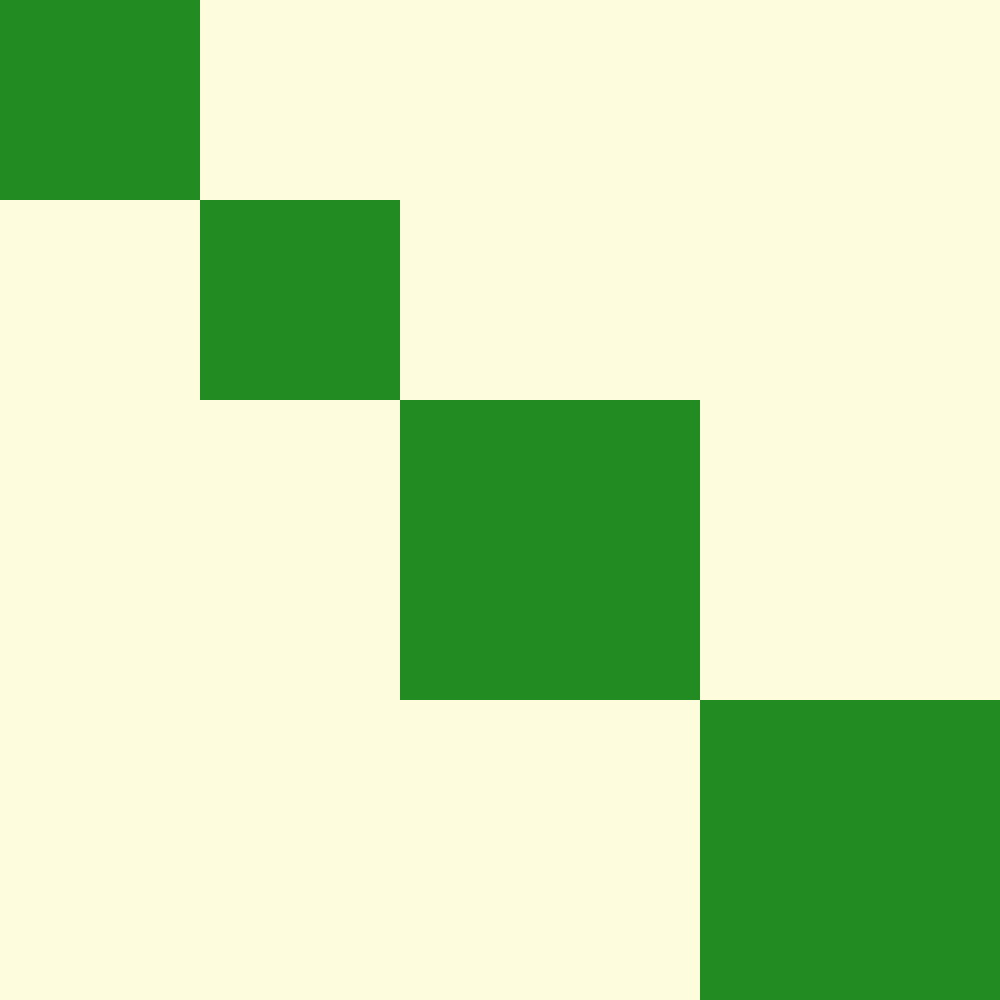}\\
    	    $\D^*$
	\end{minipage}
	\begin{minipage}{.225\textwidth}
	    \centering
    	    \includegraphics[width=1\linewidth]{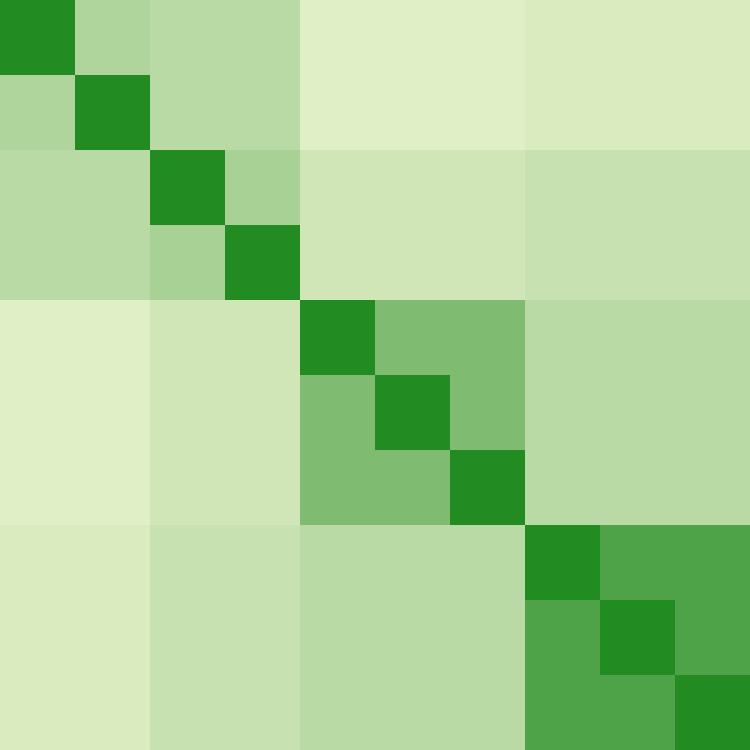}	\\
    	    $\Tt(\Th|\mathcal{G}^*)$
	\end{minipage}
	\begin{minipage}{.225\textwidth}
	    \centering
    	    \includegraphics[width=1\linewidth]{example-Delta}\\
    	    $\D$
	\end{minipage}
	\begin{minipage}{.225\textwidth}
	    \centering
    	    \includegraphics[width=1\linewidth]{example-T-tilde}\\
    	    $\Tt(\Th|\mathcal{G})$
	\end{minipage}
	\caption{The estimates $\Tt(\Th|\mathcal{G}^*)$ and $\Tt(\Th|\mathcal{G})$ from Example \ref{ex:4}, and the cluster membership matrices $\D$ and $\D^*$ of the partitions $\mathcal{G}$ and $\mathcal{G}^*$, respectively.} \label{fig:ex-T-tilde-sub}
\end{figure}
\end{example}

While the partition that satisfies the \ref{ass:main} may not be unique, the following holds.
\begin{proposition}\label{prop:uniqueG}
The coarsest partition $\mathcal{G}$ that satisfies the \ref{ass:main}, i.e., $\mathcal{G}$ given by Eq.~\eqref{eq:coarsestG}, is unique.
\end{proposition}

The fact that any refinement of $\mathcal{G} $ in Eq.~\eqref{eq:coarsestG} also satisfies the \ref{ass:main} motivates us to start with the finest possible partition $\mathcal{G}^{(d)} = \{ \{1\},\dots,\{d\} \}$
 for which the \ref{ass:main} always holds, and to merge the clusters one at a time in a way that resembles hierarchical agglomerative clustering. Specifically, we will create a path $\mathcal{P} = \{\mathcal{G}^{(d)},\dots, \mathcal{G}^{(1)}\}$ through the set of all possible partitions of $\{1,\dots, d\}$ with  $|\mathcal{G}^{(i)}| = i$ for $i \in \{ 1,\dots, d\}$, with the aim that $\mathcal{G}$ given in Eq.~\eqref{eq:coarsestG} is an element of $\mathcal{P}$. The construction of $\mathcal{P}$ is motivated by the following result.
 \begin{proposition}\label{prop:qform}
 Let $\mathcal{G}$ be an arbitrary partition of $\{1,\dots, d\}$, $\B$ be the block membership matrix corresponding to $\mathcal{B}_{\mathcal{G}}$ in Eq.~\eqref{eq:B-cal}, and $\G =\B \B^+$, $\tt = \G \th$ as in Theorem~\ref{thm:maxlike}. If $\S_\infty$ is positive definite, and $\hat \S$ is an estimator of $\S$ such that { $\hat \S^{-1}$ exists} and, as $n \to \infty$, $ n \hat \S \to \S_\infty$ element-wise in probability, the following holds.
 \begin{itemize}
 \item[(i)] If $\mathcal{G}$ fulfils the \ref{ass:main}, $(\th - \tt)^\top (\hat\S^{-1}/n) (\th -\tt) \to 0$ in probability.
 \item[(ii)] If $\mathcal{G}$ does not fulfill the \ref{ass:main}, $(\th - \tt)^\top (\hat \S^{-1}/n) (\th -\tt) \to \t^\top (\Id{p} - \G) \S_\infty^{-1} (\Id{p} - \G) \t$ in probability; if $\G \t \neq \t$, the limit is strictly positive.
\end{itemize}
\end{proposition}

The construction of $\mathcal{P}$ relies on slowly introducing information through constraints under which the loss function $\mathcal{L}$ in Eq.~\eqref{eq:loss} is minimized. To do this, the estimation of the unknown finite-sample variance $\S$ of $\th$ needs to be considered. While $\S$ does not appear in the estimator $\tt$ in Theorem \ref{thm:maxlike}, it is relevant for the construction of $\mathcal{P}$. In \ref{app:sigma}, we detail how to obtain an estimator of $\S$ for a given partition $\mathcal{G}$: {We first compute the plug-in estimator $\Sh$ of $\S$ following \cite{Genest/Neslehova/BenGhorbal:2011} and then average out certain entries of $\Sh$ using the block structure of $\S$ induced by $\mathcal{G}$. The resulting estimator, which also depends on $\th$, is denoted $\St(\Sh|\th,\mathcal{G})$ or simply $\St$ when no confusion can arise. When $n$ is small compared to $d$ or when $K$ is large, not enough averaging is employed and $\St$ may be too noisy. In such cases, we apply Steinian shrinkage following \cite{Devlin/Gnanadesikan/Kettenring:1975}, viz.
\begin{align*}
	\St(\Sh \mid \th,\mathcal{G},w) = (1 - w) \St(\Sh\mid \th,\mathcal{G}) + w \St_{\rm{diag}},
\end{align*}
where $\St_{\rm{diag}}$ is a diagonal matrix whose diagonal entries are those of $\St(\Sh|\th,\mathcal{G})$ and $w \in [0,1]$ is the shrinkage intensity parameter.} As shown in \ref{subapp:sigma-tilde}, $n\tilde\S( \Sh\ | \th, \mathcal{G}, w ) \to \S_{\infty}$ element-wise in probability if $w \to 0$ as $n\to\infty$. {When $K$ is large, we suggest using a value of $w$ close to $1$ so that a lot of shrinkage is applied. The extreme case $w=1$ is recommended when $d$ is large, as the estimation and storage of $\S$ becomes virtually impossible otherwise.}

Now suppose that the $i$th partition $\mathcal{G}^{(i)}$ has been selected; let $\St_w^{(i)} = \tilde\S(\Sh\ | \th, \mathcal{G}^{(i)}, w)$ denote the corresponding estimate of $\S$. To select the $(i-1)$st cluster structure $\mathcal{G}^{(i-1)}$, merge two clusters at a time and choose the optimal merger, in the sense that
\begin{align} 
\label{eq:cheap}
	\mathcal{G}^{(i-1)} = \argmin_{\substack{\mathcal{G}^*: \mathcal{T}_{\mathcal{G}^*}  \subset \mathcal{T}_{\mathcal{G}^{(i)}} , \: |\mathcal{G}^*| = i-1}} \mathcal{L} \Big(\tt(\th |\mathcal{G}^*)\mid \th,\St_w^{(i)} \Big),
\end{align}
for $\tt(\cdot | \cdot)$ given in Eq.~\eqref{eq:maxlike}. The minimization in Eq.~\eqref{eq:cheap} is done by simply going through all $i(i-1)/2$ possible mergers; $\mathcal{T}_{\mathcal{G}^*}  \subset \mathcal{T}_{\mathcal{G}^{(i)}}$ indicates that $\mathcal{G}^{(i)}$ must be a refinement of $\mathcal{G}^*$, so that the previously introduced equality constraints are carried. We then update the estimate of $\t$ to $\tt(\th |\mathcal{G}^{(i-1)})$ as in Theorem \ref{thm:maxlike}, the estimate of $\S$ to $\St_w^{(i-1)} = \St(\Sh\ | \th, \mathcal{G}^{(i-1)}, w)$ and iterate the above steps until $i=1$. The entire procedure is formalized in Algorithm \ref{algo:path} below.

\begin{algorithm} 
\caption{Construction of the path $\mathcal{P}$.}
\label{algo:path}
\begin{algorithmic}[1]
\State  Fix $w \in [0,1]$ and set  $\mathcal{G}^{(d)} = \{ \{1\},\dots,\{d\}\}$, $\St_w^{(d)} = \tilde\S(\Sh|\th, \mathcal{G}^{(d)},w)$. \Comment{Initialization}
\For {$i \in \{d-1, \dots, 1\}$}
\State Set 
$
\mathcal{G}^{(i)} = \argmin\limits_{\mathcal{G}^*: \mathcal{T}_{\mathcal{G}^*} \subset \mathcal{T}_{\mathcal{G}^{(i+1)}}, \: | \mathcal{G}^*| = i} \mathcal{L} \Big( \tt(\th|\mathcal{G}^*) \mid \th,\St_w^{(i+1)}\Big)$.
\Comment{Structure selection}
\State  Set $\St_w^{(i)} = \tilde\S(\Sh| \th, \mathcal{G}^{(i)},w)$. \Comment{Update}
\EndFor
\State  Return $\mathcal{P} = \{\mathcal{G}^{(d)}, \dots, \mathcal{G}^{(1)}\}$. \Comment{Output}
\end{algorithmic}
\end{algorithm}

Algorithm \ref{algo:path} returns a sequence $\mathcal{P}=\{\mathcal{G}^{(d)}, \dots, \mathcal{G}^{(1)}\}$ of decreasingly complex structures; note that $\mathcal{G}^{(i)}$ is a refinement of $\mathcal{G}^{(i-1)}$ for all $i\in\{2,\dots, d\}$. The partitions $\mathcal{G}^{(d)} = \{ \{1\},\dots,\{d\}\}$ and $\mathcal{G}^{(1)} = \{\{1,\dots, d\}\}$ are inevitable outputs of the algorithm; this is why we refer to $\mathcal{P}$ as a \textit{path} we took to go from $\mathcal{G}^{(d)}$ to $\mathcal{G}^{(1)}$ in the space of all partitions. If present on the path, $\mathcal{G}$ in Eq.~\eqref{eq:coarsestG} is always $\mathcal{G}^{(K)}$, where $K = |\mathcal{G}|$. An illustration is provided in Example~\ref{ex:5} below.

\begin{example}\label{ex:5}\em
Figure~\ref{fig:ex-path} shows the application of Algorithm \ref{algo:path} on $\Th$ constructed from the random sample from $\X$ in Example \ref{ex:3}; the true cluster structure $\mathcal{G}$ in Eq.~\eqref{eq:coarsestG} is given by Eq.~\eqref{eq:ex-G}. It indeed lies on the path; it corresponds to  $\D^{(3)}$.
\begin{figure}[t!]
	\centering
	\begin{subfigure}[ht!]{0.15\textwidth}
    	\centering
        \includegraphics[width=1\linewidth]{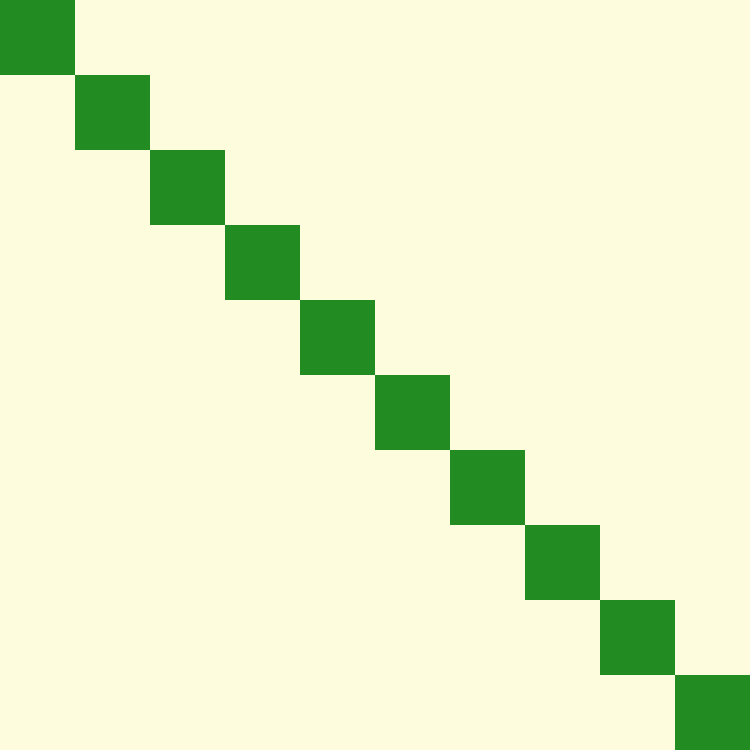}\\
        \vspace{.1cm}
        \includegraphics[width=1\linewidth]{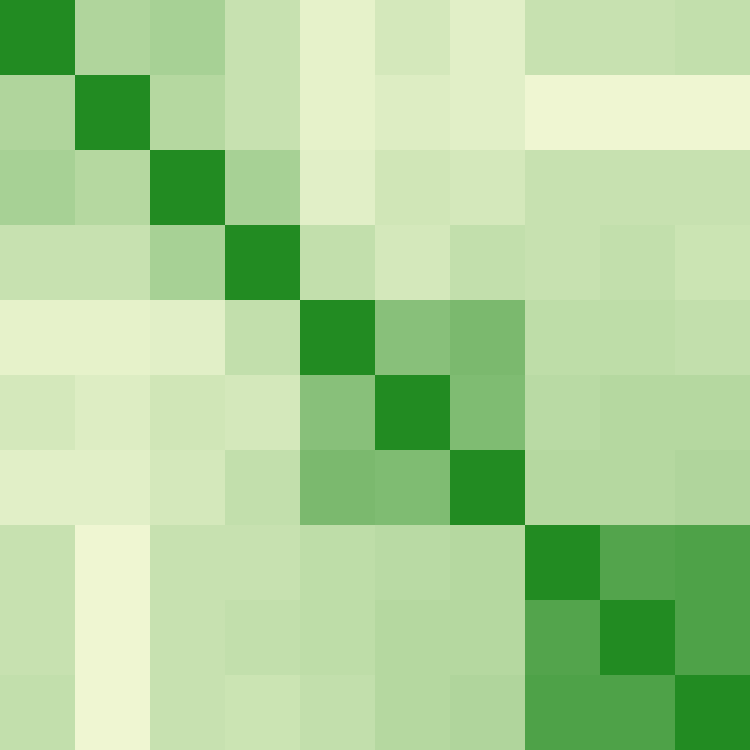}\\
        \vspace{.1cm}
        $(\D^{(10)},\Tt^{(10)})$
	\end{subfigure}\quad $\cdots$ \quad
	\begin{subfigure}[ht!]{0.15\textwidth}
    	\centering
        \includegraphics[width=1\linewidth]{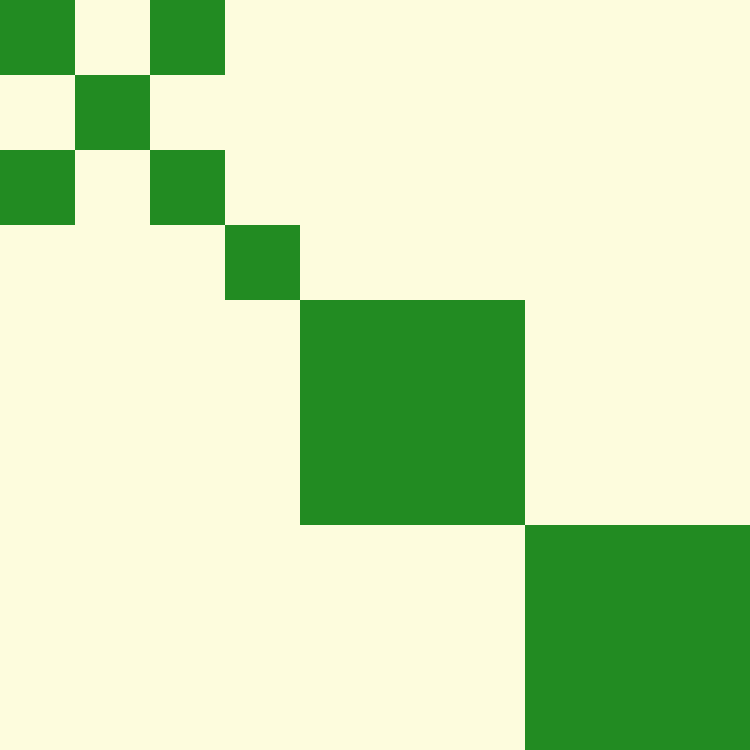}\\
        \vspace{.1cm}
        \includegraphics[width=1\linewidth]{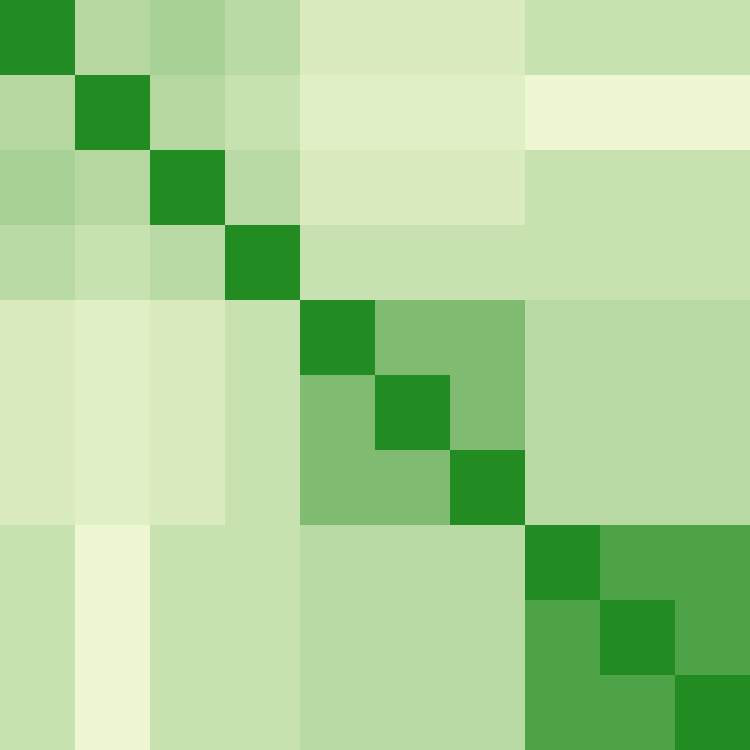}\\
        \vspace{.1cm}
        $(\D^{(5)},\Tt^{(5)})$
	\end{subfigure}
	\begin{subfigure}[ht!]{0.15\textwidth}
    	\centering
        \includegraphics[width=1\linewidth]{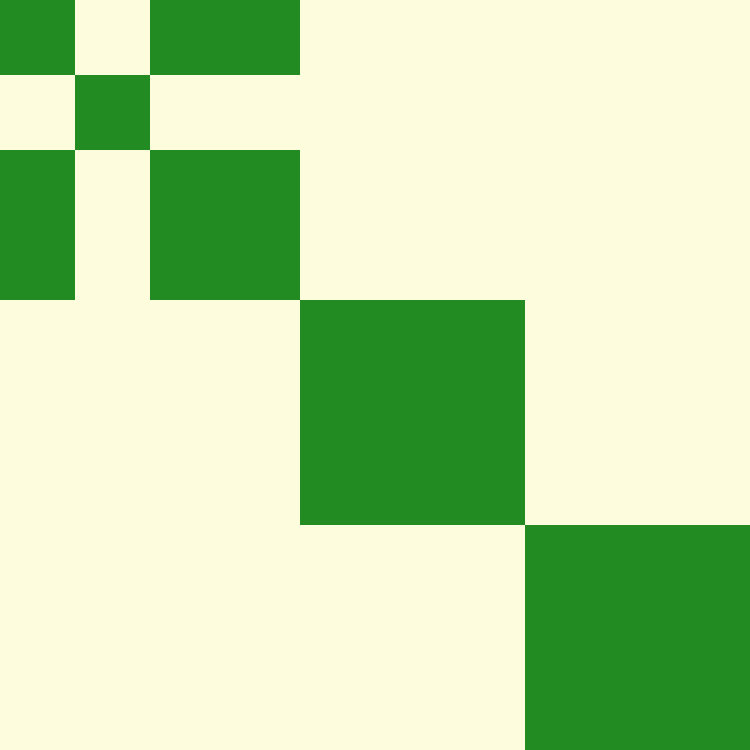}\\
        \vspace{.1cm}
        \includegraphics[width=1\linewidth]{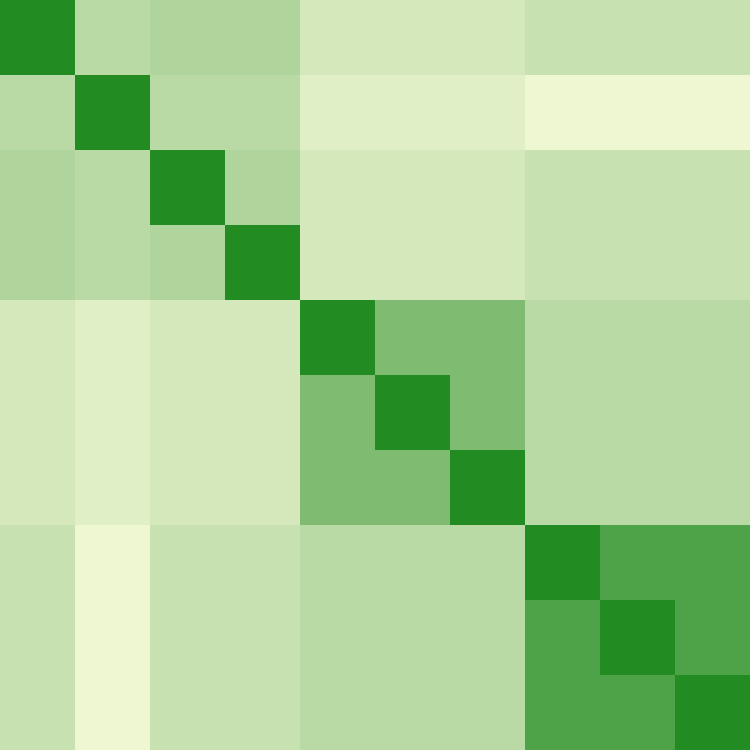}\\
        \vspace{.1cm}
        $(\D^{(4)},\Tt^{(4)})$
	\end{subfigure}
	\begin{subfigure}[ht!]{0.15\textwidth}
    	\centering
        \includegraphics[width=1\linewidth]{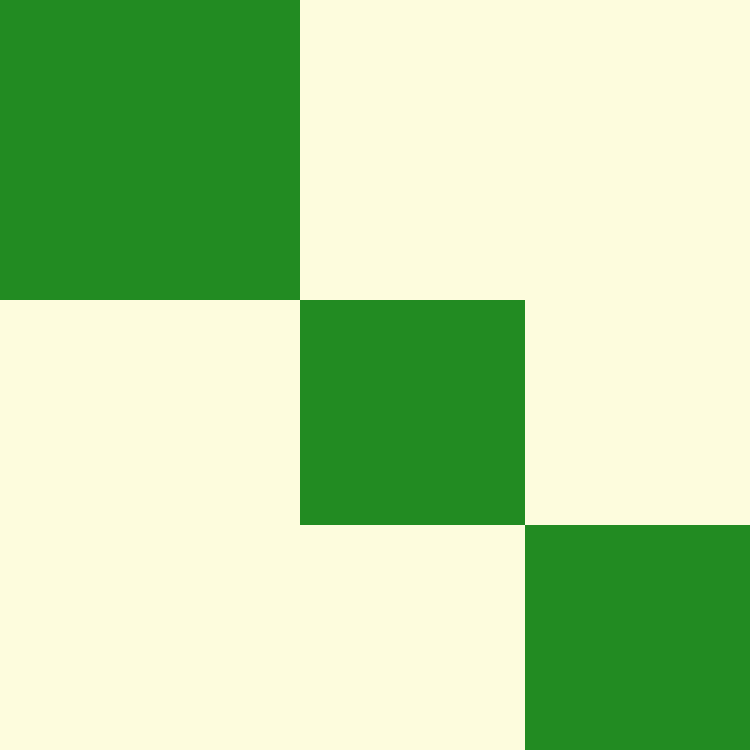}\\
        \vspace{.1cm}
        \includegraphics[width=1\linewidth]{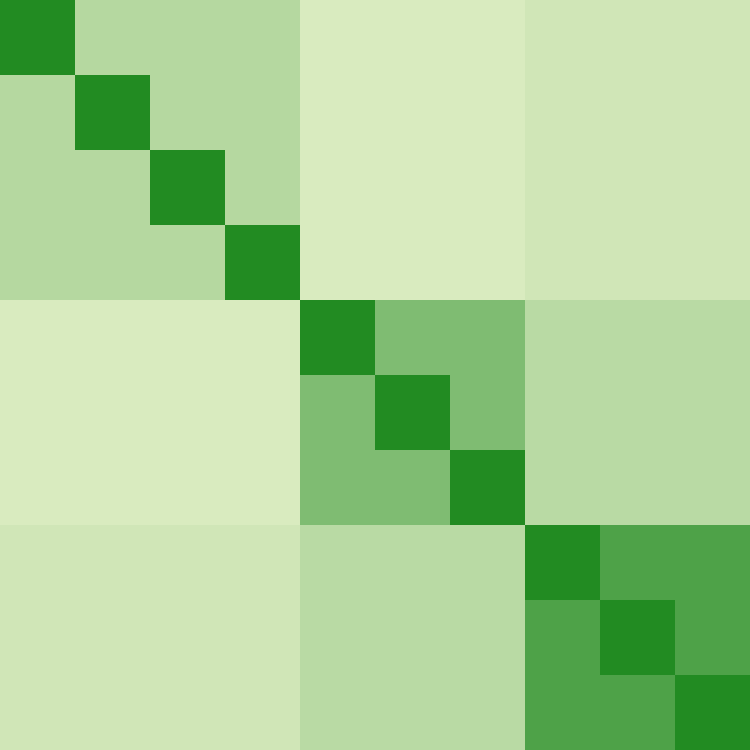}\\
        \vspace{.1cm}
        $(\D^{(3)},\Tt^{(3)})$
	\end{subfigure}
	\begin{subfigure}[ht!]{0.15\textwidth}
    	\centering
        \includegraphics[width=1\linewidth]{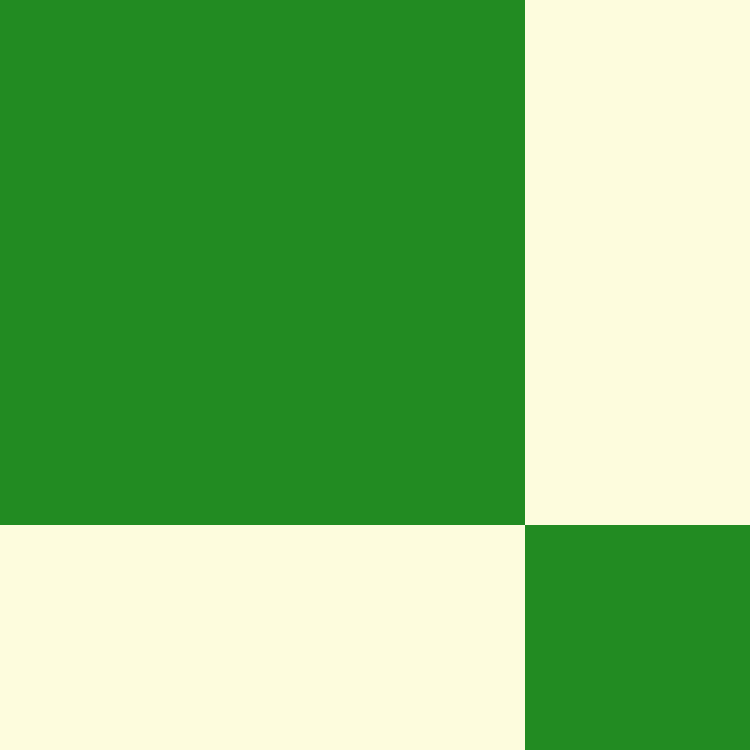}\\
        \vspace{.1cm}
        \includegraphics[width=1\linewidth]{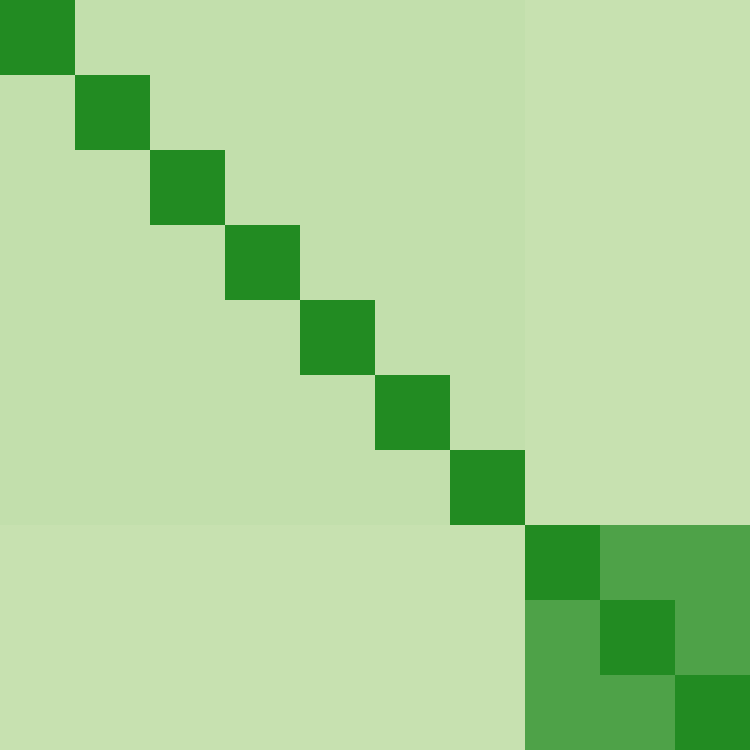}\\
        \vspace{.1cm}
        $(\D^{(2)},\Tt^{(2)})$
	\end{subfigure}
	\begin{subfigure}[ht!]{0.15\textwidth}
    	\centering
        \includegraphics[width=1\linewidth]{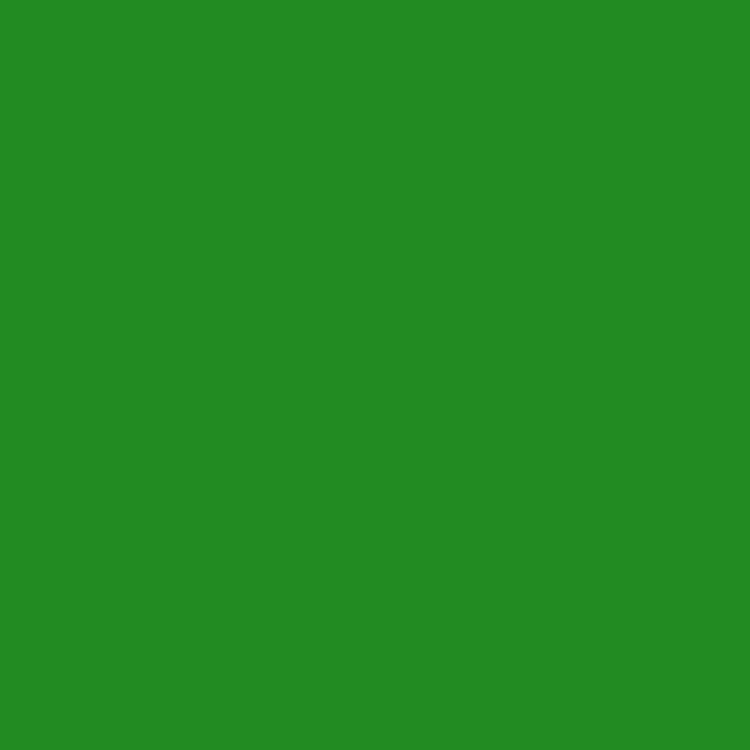}\\
        \vspace{.1cm}
        \includegraphics[width=1\linewidth]{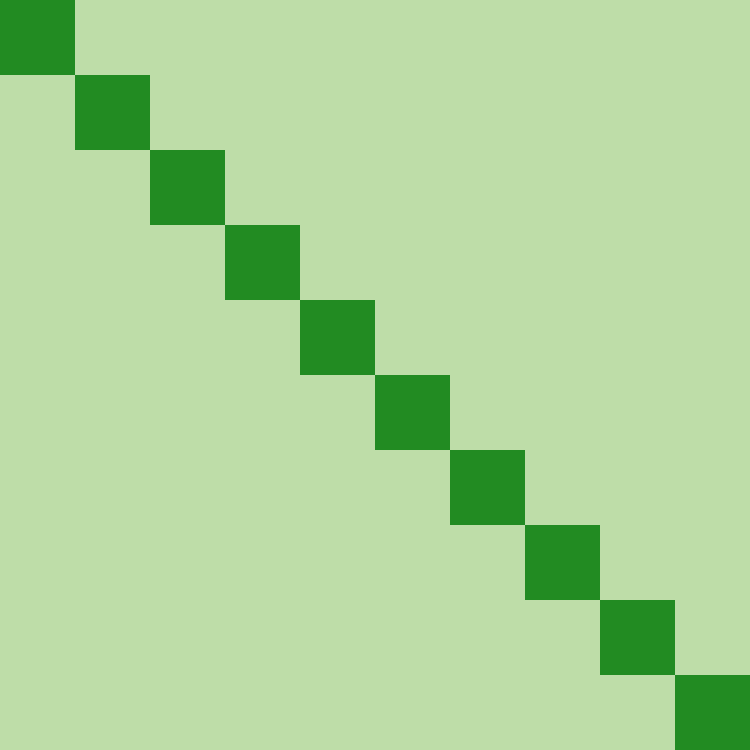}\\
        \vspace{.1cm}
        $(\D^{(1)},\Tt^{(1)})$
	\end{subfigure}
\caption{The pairs of matrices $(\D^{(i)},\Tt^{(i)})$, $i\in\{10,5,\dots,1\}$ corresponding to the path $\mathcal{P}$ obtained from Algorithm \ref{algo:path}  in Example \ref{ex:5}.}\label{fig:ex-path}
\end{figure}
\end{example}

The following corollary to Proposition \ref{prop:qform}, whose proof is deferred to \ref{app:proofs}, establishes that the path includes $\mathcal{G}$ with probability $1$ as $n$ tends to infinity.
\begin{corollary} \label{cor:on-path}
Assume that $\S_\infty$ is positive definite, and that $\G \t = \t$, where $\G =\B \B^+$ and $\B$ is a block membership matrix, holds if and only if $\B$ corresponds to a partition which satisfies the \ref{ass:main}. Let $\mathcal{G}$ be as in Eq.~\eqref{eq:coarsestG}, $\mathcal{P}$ the path returned by Algorithm \ref{algo:path}, and suppose that $w \to 0$ as $n\to\infty$. Then as $n\to \infty$,  $\Pr(\mathcal{G} \in \mathcal{P}) \to 1$.
\end{corollary}

From each $\mathcal{G}^{(i)}\in \mathcal{P}$, we can compute the cluster membership matrix $\D^{(i)}$, the improved estimate $\tt^{(i)} = \tt(\th|\mathcal{G}^{(i)})$ as in Theorem \ref{thm:maxlike}, its matrix version $\Tt^{(i)}$, and an estimate $\St^{(i)}$ of $\S$ upon setting $w=0$. Letting $\B^{(i)}$ be the block membership matrix corresponding to $\mathcal{B}_{\mathcal{G}^{(i)}}$ in Eq.~\eqref{eq:B-cal} and $\G^{(i)} = \B^{(i)}{\B^{(i)}}^{+}$, a consistent estimator of the covariance matrix of $\tt^{(i)}$ is then $\G^{(i)} \St^{(i)}\G^{(i)}$, which simplifies to $\G^{(i)} \St^{(i)}$ by Lemma~\ref{lem:gamma-sigma-rel}.

\begin{remark} \em
Consider a partition $\mathcal{G}$ for which the \ref{ass:main} holds, and let $\mathcal{G}^\dag$ be a refinement thereof. Let $\B$ and $\B^\dag$ be the block membership matrices derived from Eq.~\eqref{eq:B-cal} corresponding to $\mathcal{G}$ and $\mathcal{G}^\dag$, respectively, and set $\G = \B \B^+$ and $\G^\dag = \B^\dag (\B^\dag)^+$. Then because $\S^{-1} \in \mathcal{S}_{\mathcal{G}} \subset \mathcal{S}_{\mathcal{G}^\dag}$, where $\mathcal{S}_{\mathcal{G}}$ and $\mathcal{S}_{\mathcal{G}^\dag}$ are as defined in \ref{subapp:sigma-structure}, Lemma~\ref{lem:gamma-sigma-rel} applies and $(\Id{p} - \G^\dag)^\top \S \G^\dag = \boldsymbol{0}$. Furthermore, for $\tt^\dag = \G^\dag \th$ and $\tt = \G \th$, $\G^\dag \tt = \tt$. Hence,
\begin{align}
\label{eq:loss-decomp}
(\th - \tt)^\top \S^{-1} (\th - \tt) = (\th - \tt^\dag)^\top \S^{-1} (\th - \tt^\dag) + (\tt^\dag-\tt)^\top \S^{-1} (\tt^\dag-\tt).
\end{align}
Now set $K= | \mathcal{G}|$. If $\mathcal{G}^{(K)}, \ldots, \mathcal{G}^{(d)}$ is a sequence of partitions such that $\mathcal{G}^{(K)} = \mathcal{G}$, $\mathcal{G}^{(d)} = \{ \{1\},\dots , \{d\}\}$, and for each $i \in\{ K,\dots, d-1\}$, $\mathcal{G}^{(i+1)}$ is a refinement of $\mathcal{G}^{(i)}$. A successive application of Eq.~\eqref{eq:loss-decomp} then gives
$$
	(\th - \tt^{(K)})^\top \S^{-1} (\th - \tt^{(K)})= \sum_{i=K+1}	^{d} (\tt^{(i)} -\tt^{(i-1)})^\top \S^{-1} (\tt^{(i)} -\tt^{(i-1)}) 
$$
In particular, for any $i\in\{K,\dots, d-1\}$, $\mathcal{L}(\tt^{(i)}|\th,\S) \ge\mathcal{L}(\tt^{(i+1)}|\th,\S)$. This motivates that in Algorithm \ref{algo:path}, only two clusters are merged at a time.
\end{remark}

\subsection{Structure selection} \label{subsec:selection}

Proposition \ref{prop:qform} suggests that the loss will increase sharply when the clustering has become too coarse. The following result offers a way to determine when this sharp increase may have occurred.
\begin{proposition} \label{prop:stop}
Let $\mathcal{G}=\{\mathcal{G}_1,\dots, \mathcal{G}_K\}$ be a partition of $\{1,\dots, d\}$ satisfying the \ref{ass:main}, $\B$ the block membership matrix corresponding to $\mathcal{B}_{\mathcal{G}}$ in Eq.~\eqref{eq:B-cal}, and $\G = \B \B^{+}$, $\tt = \G \th$ as in Theorem \ref{thm:maxlike}.
If $\S_\infty$ is positive definite, and $\hat \S$ is any estimator of $\S$ such that { $\hat \S^{-1}$ exists} and, as $n \to \infty$, $ n \hat \S \to \S_\infty$ element-wise in probability, then, as $n\to \infty$,
$
\mathcal{L}(\tt | \th, \hat \S)=
(\th -\tt)^{\top} \hat\S^{-1} (\th -\tt) \rightsquigarrow \chi_{p -L}^2,
$
where $L$ is the number of distinct blocks given in Eq.~\eqref{eq:nblocks}.
\end{proposition}
At each iteration of Algorithm~\ref{algo:path}, $\S$ is estimated by $\St_w^{(i)}$. Proposition \ref{prop:stop} and Eq.~\eqref{eq:Stildeconv} suggest using $\mathcal{L}(\tt^{(i)} | \th, \St_w^{(i)})$ to get a rough idea of when too much clustering has been applied through
\begin{align} 
\label{eq:alpha}
	\alpha^{(i)} = \Pr \Big\{ \chi_{p - L_i}^2 > \mathcal{L}(\tt^{(i)} \mid \th, \St_w^{(i)}) \Big\},
\end{align}
where $L_i$ is the number of blocks given in Eq.~\eqref{eq:nblocks} corresponding to the $i$th partition $\mathcal{G}^{(i)}$.
For $n$ large enough, we expect that a sharp decrease in $\alpha^{(i)}$ will occur at the first $i$ such that $\mathcal{T}_{\mathcal{G}} \not\subseteq \mathcal{T}_{\mathcal{G}^{(i)}}$, i.e., when the $\G$ matrix corresponding to $\mathcal{G}^{(i)}$ becomes inadmissible. We do not use the criterion \eqref{eq:alpha} as a formal $p$-value, but rather as a tool that can help with structure selection. In Appendix C.2 
in the Online Supplement, we present a naive automated selection procedure based on Eq.~\eqref{eq:alpha}, which we refer to as Algorithm 
C1.

\begin{example}\label{ex:6}\em
Consider again the random sample of size $n=70$ from $\X$ in Example \ref{ex:3}.  We computed $\alpha^{(i)}$, $i\in\{1,\dots,10\}$, given by Eq.~\eqref{eq:alpha} for the path obtained with Algorithm \ref{algo:path} in Example \ref{ex:5}. As can be seen in Figure~\ref{fig:ex-alphas}, the gap between $\alpha^{(3)}$ and $\alpha^{(2)}$ strongly suggests that the best structure is $\mathcal{G}^{(3)}$, which is indeed the true structure in this case.
\begin{figure}[t!]
\centering
        \includegraphics[width=.7\linewidth]{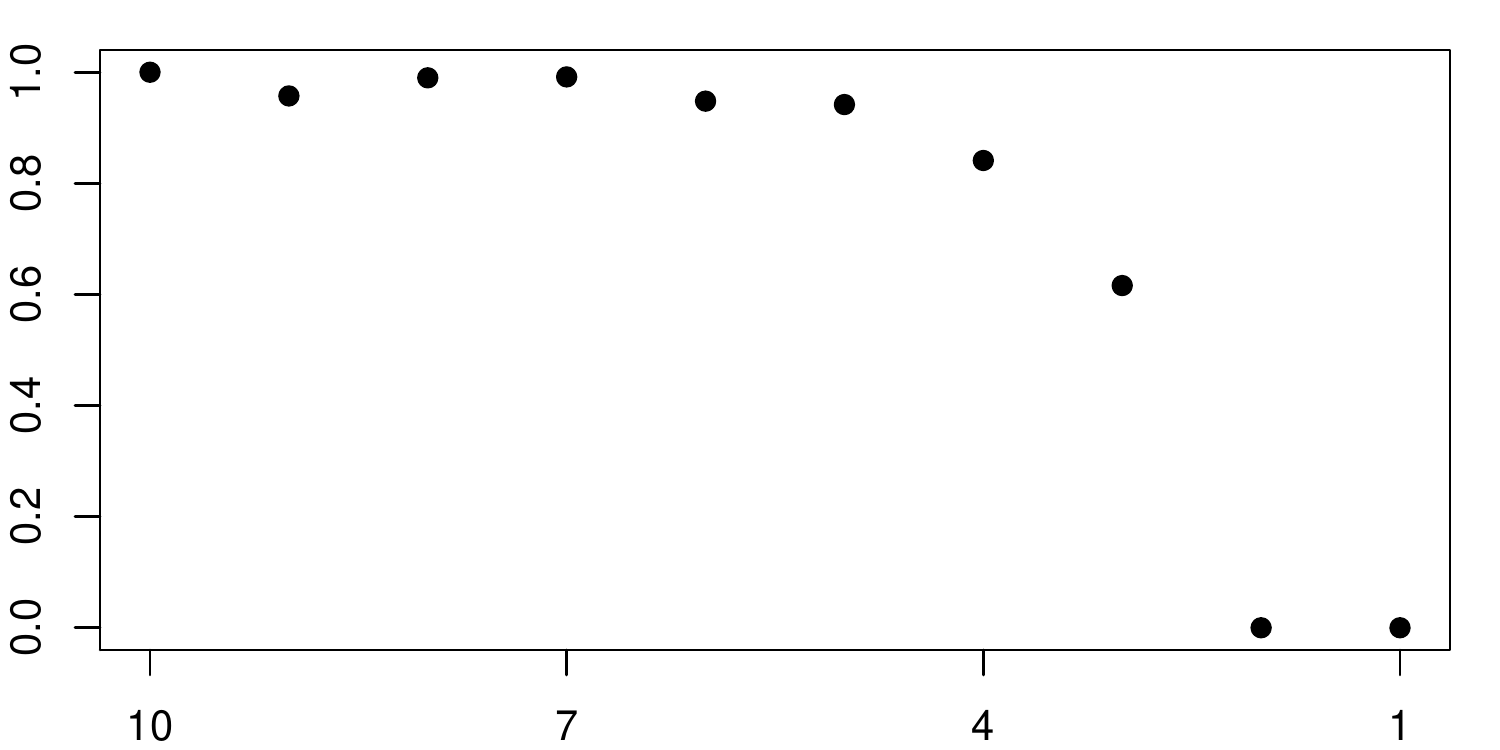}
\caption{The pairs $(i,\alpha^{(i)})$, $i \in\{ 10, \dots,1\}$, computed in Example \ref{ex:6}.}\label{fig:ex-alphas}
\end{figure}
\end{example}

\section{Estimation of linear correlation} \label{sec:bijective}

We now show how the \ref{ass:main} can be used to obtain improved estimates of the classical linear correlation matrix $\bss{P}$ with entries $\bss{P}_{ij} = {\rm cor}(X_i, X_j)$, $i,j \in\{ 1,\dots, d\}$ when $\X$ is elliptical. Recall that an absolutely continuous random vector $\bss{X}$ follows an elliptical
distribution with mean vector $\bss{\mu} \in \mathbb{R}^d$, positive definite $d\times d$ dispersion matrix $\bss{D}$ and density generator $g:[0,\infty)\to[0,\infty)$, in notation $\X \sim \mathcal{E}( \bss{\mu}, \bss{D}, g)$, if its density $f$ satisfies, for all $\bss{x} \in \mathbb{R}^d$,
\begin{equation}\label{eq:denselli}
f(\bss{x})=|\bss{D}|^{-1/2} g
\{(\bss{x}-\boldsymbol{\mu})^\top
\bss{D}^{-1}(\bss{x}-\bss{\mu})/2 \},
\end{equation}
where $|\bss{D}|$ denotes the determinant of $\bss{D}$  \cite{Fang/Kotz/Ng:1990,Fang/Zhang:1990}. Eq.~\eqref{eq:denselli} means that the level curves of $f$ are concentric ellipses centred at $\bss{\mu}$. Well-known examples of elliptical distributions are the multivariate Normal, Student $t$ or generalized hyperbolic distributions; for their use in finance and risk modeling, see \cite{McNeil/Frey/Embrechts:2015}.

Note that when $\X \sim  \mathcal{E}( \bss{\mu}, \bss{D}, g)$, $\bss{\mu}$ and $\bss{D}$ are not necessarily the mean and covariance matrix of $\X$, respectively; these moments may not even exist. However, if $\E(X_i^2) < \infty$ for all $i \in \{1,\dots, d\}$, ${\rm E}(\X) = \bss{\mu}$ and  there exists a constant $c >0$ such that ${\rm cov}(\X) = c \bss{D}$, see Theorem 2.6.4 in  \cite{Fang/Zhang:1990}. Consequently, if the linear correlation matrix $\bss{P}$ of $\X$ exists, one has, for all $i,j\in\{1,\dots, d\}$,
$\bss{P}_{ij} = \bss{D}_{ij}/\sqrt{\bss{D}_{ii}\bss{D}_{jj}}$.
Surprisingly, for all $i\neq j \in \{1,\dots, d\}$, the correlation coefficient $\bss{P}_{ij}$ is in one-to-one relationship with Kendall's correlation $\tau(X_i, X_j)$ \cite{Fang/Fang:2002, Lindskog/McNeil/Schmock:2002}, viz.
\begin{align}\label{eq:tauelli}
\bss{P}_{ij} =  \sin  ( \pi\T_{ij}/2 ).
\end{align}
Because the map in Eq.~\eqref{eq:tauelli}  is a bijection, it can be used to construct an estimator of $\bss{P}$, given, for all $i \neq j\in\{1,\dots, d\}$, by $\hat{\bss{P}}_{ij} = \sin(\pi \hat \T_{ij} /2)$. As illustrated by \cite{Lindskog/McNeil/Schmock:2002}, the resulting estimator $\hat{\bss{P}}$ can be considerably more efficient than the sample correlation matrix, especially when the margins of $\bs{X}$ are heavy-tailed. Recently, $\hat{\bss{P}}$ has been employed, e.g., in the context of nonparanormal graphical models \citep{Liu/Han/Yuan/Lafferty/Wasserman:2012,Xue/Zou:2012}, and Gaussian or elliptical copula regression \citep{Cai/Zhang:2017,Zhao/Genest:2017}.

{Now suppose that $\mathcal{G}$ is a partition of $\{1,\dots, d\}$ so that the \ref{ass:main} holds. Because $\bs{X}$ is elliptical, this is equivalent to $\T \in \mathcal{T}_{\mathcal{G}}$, or, in view of Eq.~\eqref{eq:tauelli}, to $\bss{P} \in \mathcal{T}_{\mathcal{G}}$. Because $\Tt$ is a more efficient estimator of $\T$ by Theorem \ref{thm:reduced-variance}, the delta method implies that $\tilde{\bss{P}} \in  \mathcal{T}_{\mathcal{G}}$ obtained by using $\Tt_{ij}$ in Eq.~\eqref{eq:tauelli} is a more efficient estimator of $\bss{P}$ than $\hat{\bss{P}}$. Moreover, if $\bss{P}$ is positive definite, it follows from Lemma~\ref{lem:T-inv} in \ref{subapp:proof-T-inv} that the precision matrix $\bs{\Omega} = \bss{P}^{-1}$ has the same block structure as $\bss{P}$, i.e., $\bs{\Omega} \in \mathcal{T}_{\mathcal{G}}$. As an estimator of $\bs{\Omega}$, one may thus use $\tilde{\bs{\Omega}}=\tilde{\bss{P}}^{-1}$  directly if the latter is positive definite; it then follows from Lemma~\ref{lem:T-inv} that $\tilde{\bs{\Omega}} \in \mathcal{T}_{\mathcal{G}}$. Otherwise, $\tilde{\bss{P}}$ can first be made positive definite using one of the shrinkage methods described, e.g., in \cite{Rousseeuw/Molenberghs:1993}; its inverse can be further improved by averaging out the entries block-wise to obtain a matrix in $\mathcal{T}_{\mathcal{G}}$. 

\section{Simulation study} \label{sec:simulation}

We first investigate whether the true cluster structure is on the path $\mathcal{P}$ returned by Algorithm \ref{algo:path} and how often the criterion (\ref{eq:alpha}) selects the true cluster structure given the latter is in $\mathcal{P}$. The full description and results of this simulation are provided in the Online Supplement; here we only present the main conclusions for the sake of brevity.

Based on $500$ simulation runs and samples of various sizes from the Normal and the Cauchy copula with $d=20$ and Kendall's rank correlation matrix $\T \in \{ \T_1, \ldots, \T_4\}$ displayed in Figure~\ref{fig:gen-mats}, we found that irrespectively of the dependence structure, the number of paths that contain the true $\mathcal{G}$ given by Eq.~\eqref{eq:coarsestG} increases with the sample size. Also, $\mathcal{G} \in \mathcal{P}$ more often when the blocks are clearly separated, e.g., when $\T$ equals $\T_1$ or $\T_2$. The shrinkage parameter $w$ has little impact unless the sample size is small and $w\in \{0, 0.25\}$; the optimal choice overall seems to be $w=0.75$.

We also assessed the selection of the most appropriate cluster structure with Algorithm C1 based on Eq.~\eqref{eq:alpha}. As expected, the percentage of simulations in which the latter algorithm identifies $\mathcal{G}$ given by Eq.~\eqref{eq:coarsestG}  gets closer to $1-\alpha$ as the sample size increases, leading to the recommendation to choose a small value of $\alpha$ once $n$ is sufficiently large. However, the meaning of large depends on how well the blocks in $\T$ are separated. For example, when $\T = \T_1$ and $\T=\T_3$, $\alpha=0.05$ becomes the best option when $n=125$, and $n=500$, respectively. For smaller sample sizes, it is better to choose a larger value of $\alpha$.
\begin{figure}[t!]
	\centering
	\begin{minipage}{.05\textwidth}
		\centering
	   	\includegraphics[width=1\linewidth]{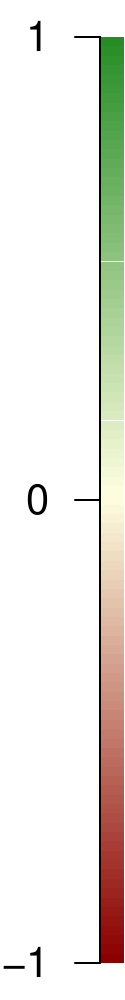}\\
	\end{minipage}
	\begin{minipage}{.9\textwidth}
	\centering
	\begin{minipage}{.23\textwidth}
		\centering
    	\includegraphics[width=1\linewidth]{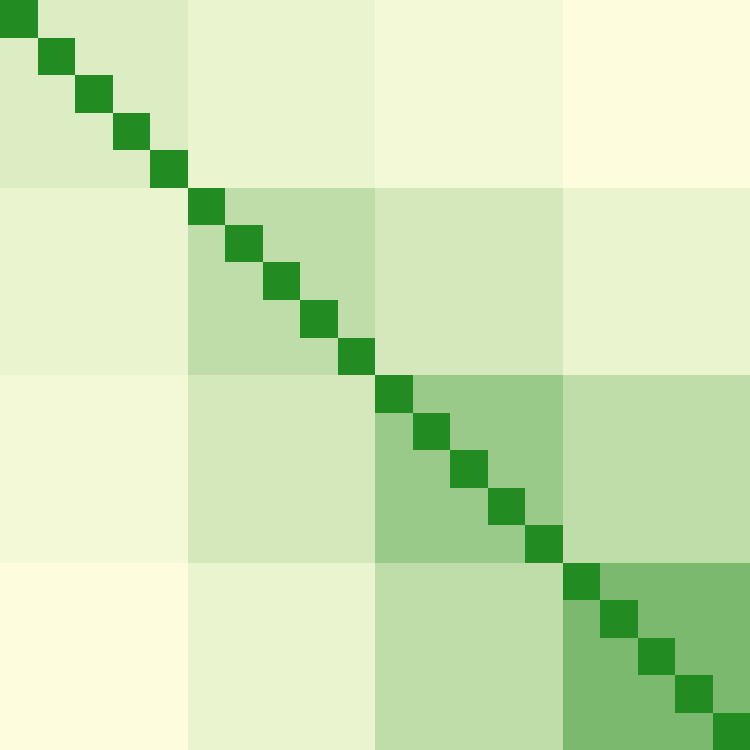}\\
    	$\T_1$\\
	\end{minipage}
	\begin{minipage}{.23\textwidth}
		\centering
    	\includegraphics[width=1\linewidth]{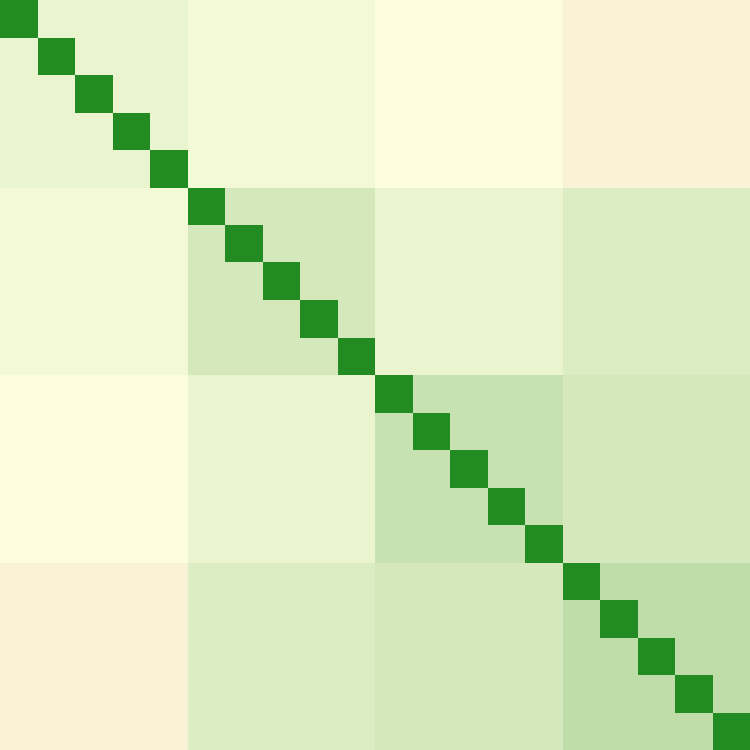}\\
    	$\T_2$
	\end{minipage}
	\begin{minipage}{.23\textwidth}
		\centering
    	\includegraphics[width=1\linewidth]{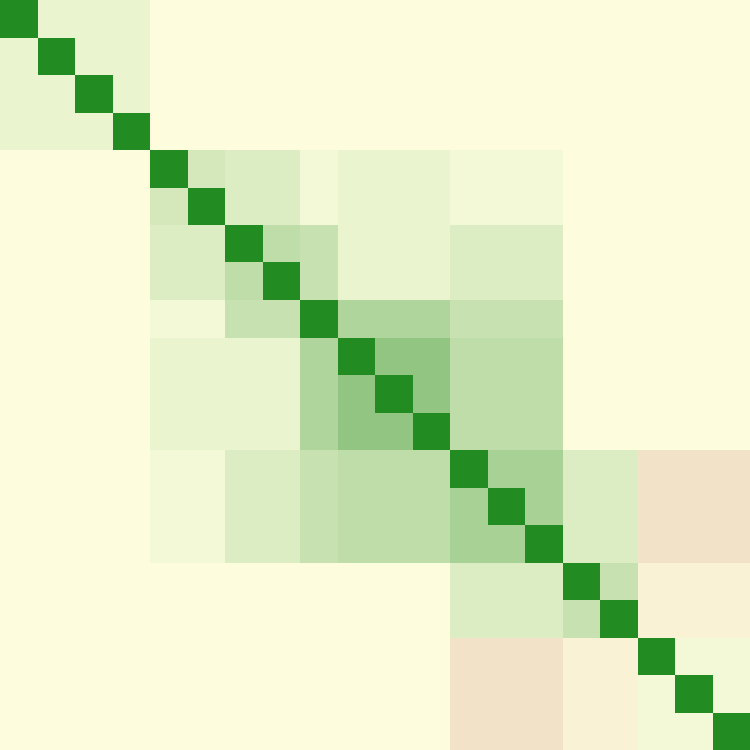}\\
    	$\T_3$
	\end{minipage}
	\begin{minipage}{.23\textwidth}
		\centering
    	\includegraphics[width=1\linewidth]{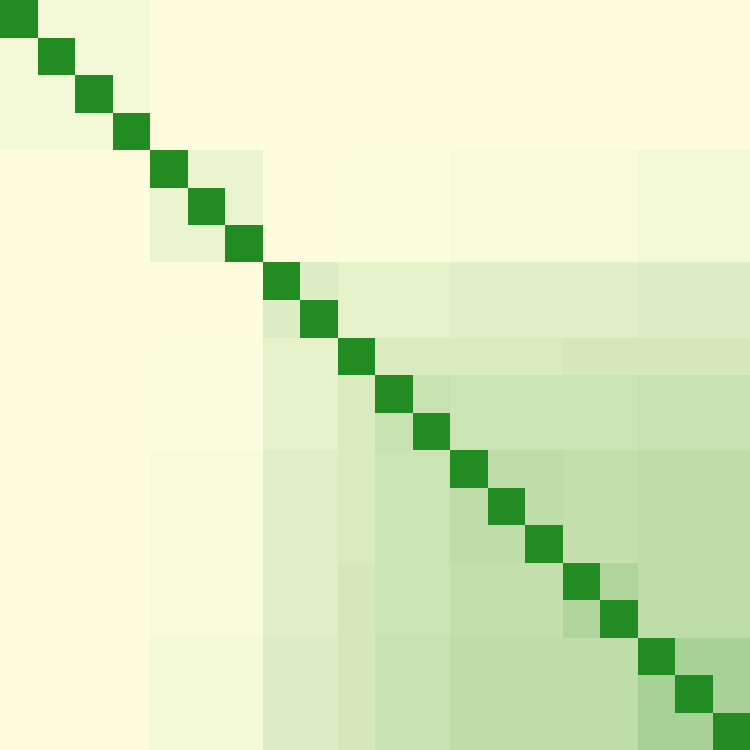}\\
    	$\T_4$
	\end{minipage}
	
	\vspace{.15cm}
	\begin{minipage}{.23\textwidth}
		\centering
    	\includegraphics[width=1\linewidth]{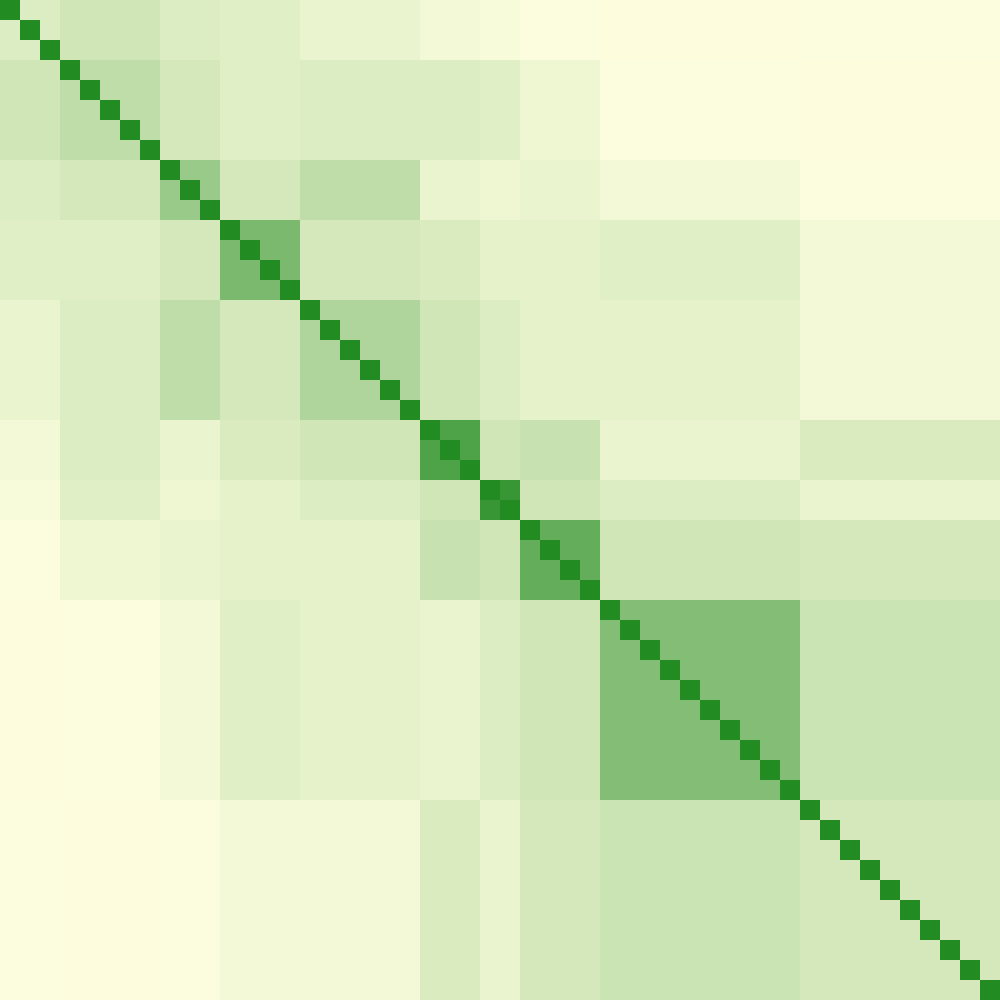}\\
    	$\T_5$
	\end{minipage}
	\begin{minipage}{.23\textwidth}
		\centering
    	\includegraphics[width=1\linewidth]{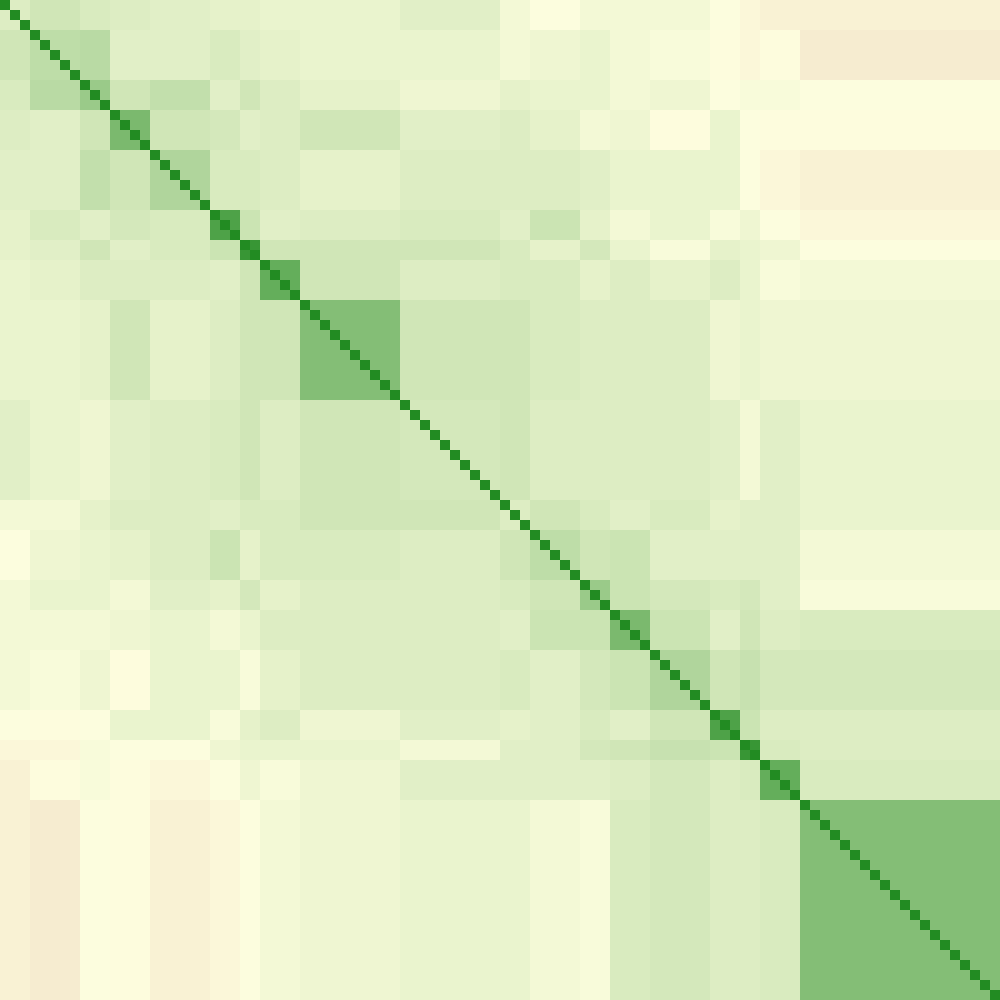}\\
    	$\T_6$
	\end{minipage}
	\begin{minipage}{.23\textwidth}
		\centering
    	\includegraphics[width=1\linewidth]{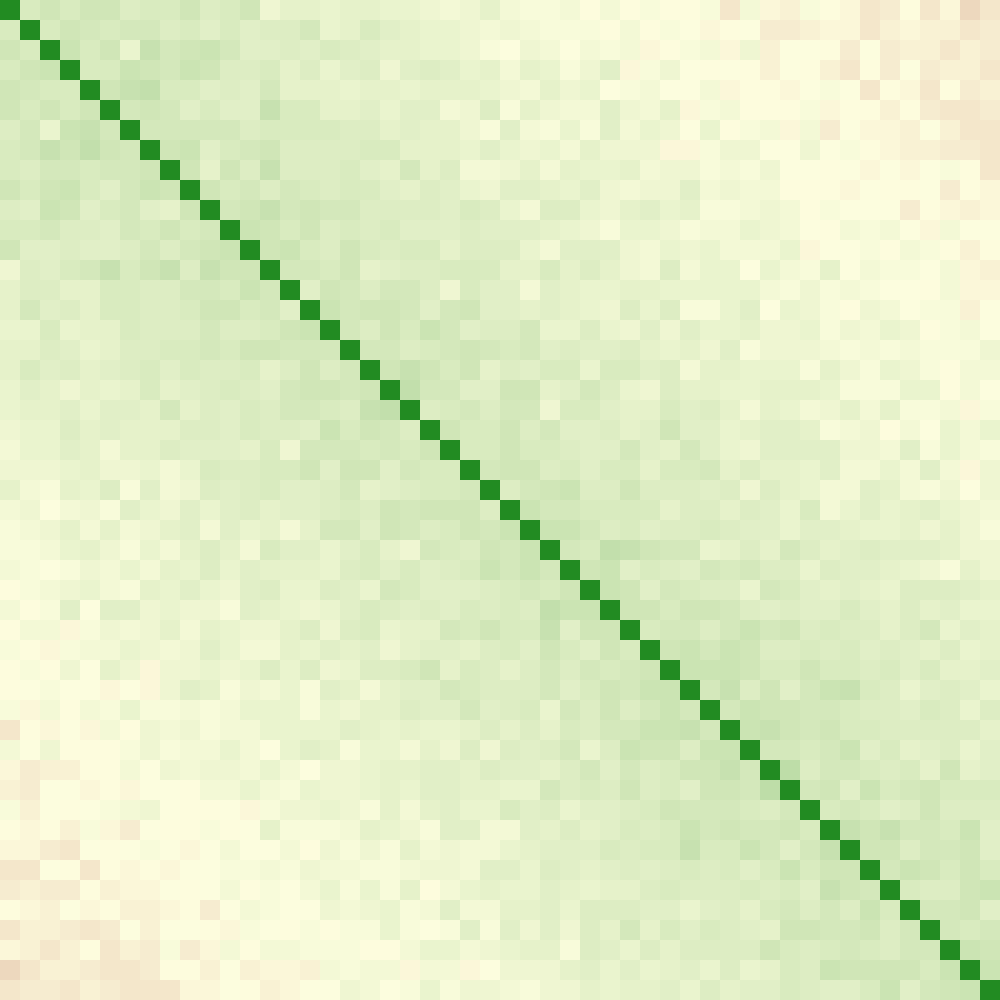}\\
    	$\T_7$
	\end{minipage}
	\begin{minipage}{.23\textwidth}
		\centering
    	\includegraphics[width=1\linewidth]{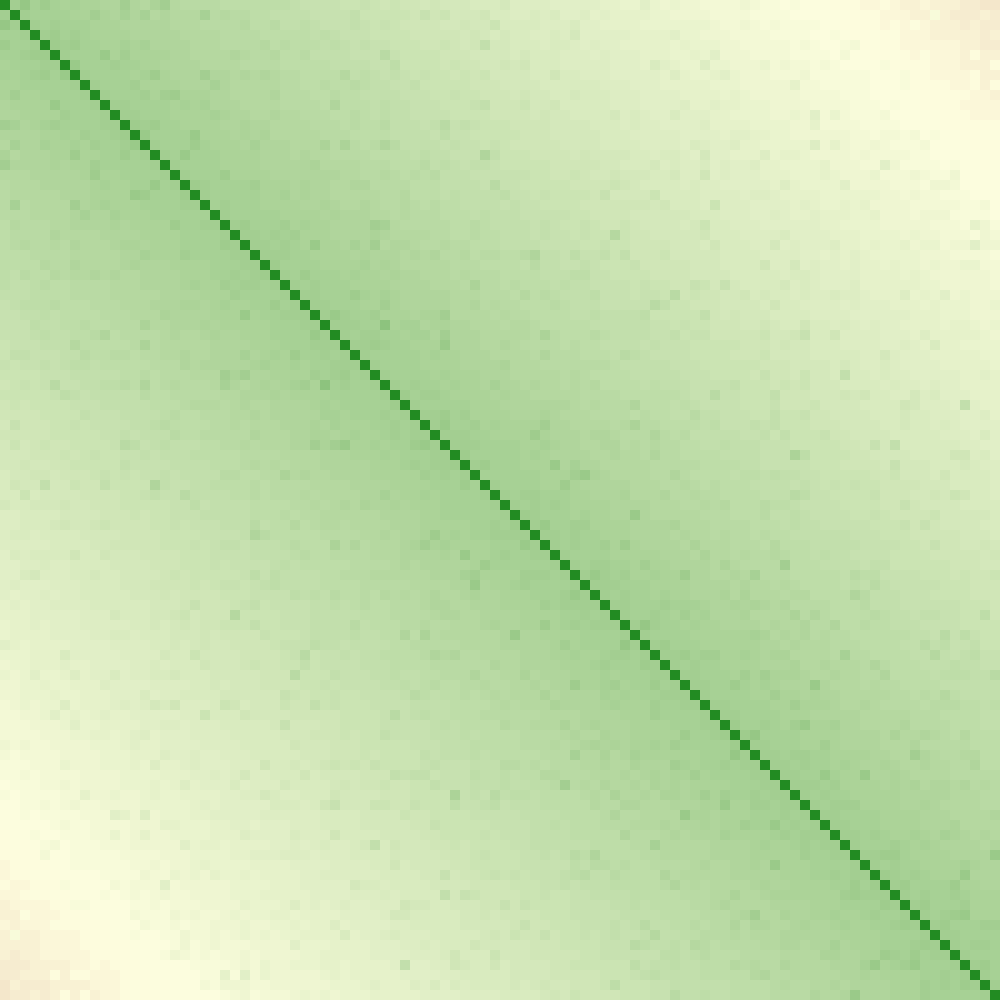}\\
    	$\T_8$
	\end{minipage}
	\end{minipage}
	\caption{The matrices $\T_1,\dots, \T_8$ used in the simulation study.}\label{fig:gen-mats}
\end{figure}
Even though Algorithm \ref{algo:path} in combination with Algorithm C1 may not identify the true coarsest cluster structure properly, the estimator $\Tt$ that it returns is still considerably more efficient than $\Th$. We explore this by first calculating the oracle statistic
\begin{align*}
\nu_2 = 1 -  \min \big \{ ||\tt^{(j)} - \t||_2^2 : {\mathcal{G}^{(j)} \in \mathcal{P}} \big \}/||\th - \t||_2^2.
\end{align*}
To assess the performance of Algorithm C1, we also calculate
\begin{align*} 
\xi(\alpha) = 1 - ||\t^\bullet - \t||_2^2/||\th - \t||_2^2,
\end{align*}
where $\t^\bullet$ corresponds to the cluster structure selected by Algorithm C1 with level $\alpha$. Note that  $\nu_2$ and $\xi(\alpha)$ are always between $0$ (worst) and $1$ (best). The values of these statistics averaged over the $500$ simulation runs are $\bar \nu_2$ and $\bar \xi(\alpha)$.
\begin{table}[t!]
\small
	\caption{Average values of $\nu_2$ and $\xi(\alpha)$ for various $\alpha$ for $500$ simulation runs with $w=0.75$.} \label{tab:means}	
	\centering
	\begin{tabular}{*{1}{l}*{15}{c}}
	\toprule
& \multicolumn{3}{c}{$\T_1$} & & \multicolumn{3}{c}{$\T_2$} & & \multicolumn{3}{c}{$\T_3$} & & \multicolumn{3}{c}{$\T_4$}\\
\cmidrule(lr){2-4} \cmidrule(lr){6-8} \cmidrule(lr){10-12} \cmidrule(lr){14-16}
		Measure\textbar$n$ & 125 & 250 & 500 & $\quad$ & 125 & 250 & 500 & $\quad $ & 125 & 250 & 500 &$ \quad$ & 125 & 250 & 500\\
		\midrule
		$\bar{\nu}_2$  & 0.65 & 0.66 & 0.67 & & 0.70 & 0.80 & 0.80 & & 0.51 & 0.59 & 0.63 & & 0.51 & 0.52 & 0.56\\
		\midrule
		$\bar\xi(0.05)$	 & 0.61 & 0.66 & 0.67 & & 0.61 & 0.78 & 0.79 & & 0.35 & 0.47 & 0.61 & & 0.37 & 0.39 & 0.44\\
$\bar\xi(0.10)$	 & 0.61 & 0.65 & 0.66 & & 0.63 & 0.78 & 0.78 & & 0.38 & 0.50 & 0.61 & & 0.39 & 0.41 & 0.46\\
$\bar\xi(0.25)$	 & 0.59 & 0.61 & 0.63 & & 0.64 & 0.75 & 0.75 & & 0.41 & 0.52 & 0.59 & & 0.41 & 0.42 & 0.47\\
$\bar\xi(0.50)$	 & 0.53 & 0.55 & 0.54 & & 0.61 & 0.70 & 0.69 & & 0.43 & 0.51 & 0.53 & & 0.41 & 0.42 & 0.45\\
	\end{tabular}
	\begin{tabular}{*{1}{l}*{15}{c}}
	\toprule
& \multicolumn{3}{c}{$\T_5$} & & \multicolumn{3}{c}{$\T_6$} & & \multicolumn{3}{c}{$\T_7$} & & \multicolumn{3}{c}{$\T_8$}\\
\cmidrule(lr){2-4} \cmidrule(lr){6-8} \cmidrule(lr){10-12} \cmidrule(lr){14-16}
		Measure\textbar$n$ & 25 & 50& 100& $\quad$ & 25 & 50 & 100 & $\quad $ & 25 & 50 & 100 &$ \quad$ & 25 & 50 & 100\\
		\midrule		
		$\bar{\nu}_2$  & 0.47 &0.52 &0.54& & 0.48 & 0.46 &0.48 & & 0.66 & 0.51 & 0.37& & 0.42 & 0.38 & 0.34\\
		\midrule
		$\bar\xi(0.05)$	 & 0.30 & 0.31& 0.34& & 0.40&0.31&0.29 & & 0.60 & 0.47 & 0.33  & & 0.31 & 0.20 & 0.11 \\
$\bar\xi(0.50)$	 & 0.34 & 0.37 & 0.38 & & 0.40 & 0.33 & 0.31 & & 0.59 & 0.48 & 0.34 & & 0.32& 0.22& 0.13\\
$\bar\xi(0.95)$	 & 0.37& 0.40& 0.44 & & 0.41 & 0.34 & 0.33 & & 0.58  & 0.47  & 0.32  & &0.33  & 0.23 & 0.15 \\
  	\bottomrule
	\end{tabular}
\end{table}

The results for the simulation design detailed in Appendix C.1 
of the Online Supplement with the Normal copula are displayed in Table~\ref{tab:means}; additional results for the Cauchy copula and other values of $w$ may be found in Tables C1 and C2
 in the Online Supplement. First note that the results for distinct matrices $\T_i$ are not directly comparable, because the maximum error reduction that can possibly be achieved depends on $\T_i$. By looking at  $\bar{\nu}_2$, we see that the {mean squared error of $\th$} can be cut substantially by choosing the best structure on the path $\mathcal{P}$ returned by Algorithm \ref{algo:path} in each of the scenarios. The values of $\bar \xi(\alpha)$ suggest that Algorithm C1 often selects a structure that does nearly as well in terms of error reduction. The reactivity to changes in $\alpha$ is low, which is a sign that the gap in $\alpha^{(i)}$ between reasonable and poor models is often large. Algorithm C1 performs best when $\alpha$ is small once the sample size is large enough. Even if Algorithms \ref{algo:path} and C1 do not select the true coarsest partition $\mathcal{G}$ that satisfies \ref{ass:main}, Table \ref{tab:means} illustrates that the resulting estimator $\Tt$ can outperform $\Th$ considerably. This is particularly apparent in the block pertaining to $\T_4$. Even though hardly any path goes through $\mathcal{G}$ in this case (viz. Table C1 
 in the Online Supplement), Table \ref{tab:means} shows that the structure selected by Algorithms \ref{algo:path} and C1 leads to a substantial reduction of mean squared error.

The above findings suggest that selecting a structure simpler than the true one might be beneficial when $\Th$ is extremely noisy because even if it introduces a small bias, it reduces the variance considerably. To demonstrate this further, we challenge Algorithms \ref{algo:path} and C1 by considering $d\in \{50,100\}$ and four additional matrices $\T_5,\dots, \T_8$ also displayed in Figure~\ref{fig:gen-mats}. $\T_5$ and $\T_6$ correspond to $50$ variables and have a block structure with 10 and 19 blocks, respectively. The matrices $\T_7$ and $\T_8$ are $100 \times 100$ noisy versions of Toeplitz matrices with a constant decay in their entries as one moves away from the diagonal. This means that $\T_7$ and $\T_8$ are unstructured, i.e., only $\{\{1\},\dots, \{d\}\}$ satisfies the \ref{ass:main}. We also consider rather small sample sizes, viz.~$n \in \{25,50,100\}$. Because better results are achieved for higher values of the shrinkage parameter $w$ in small samples, we set $w=1$.

The results for $500$ simulation runs are also summarized in Table~\ref{tab:means}. They suggest that the method still performs well in higher-dimensional settings, even when $d$ is larger than $n$. In accordance with Table C3 
in the Online Supplement, it is better to choose a more conservative value for $\alpha$ in small samples, although the sensitivity to $\alpha$ seems small. The results are particularly interesting for $\T_7$ and $\T_8$, which do not have a block structure. However, they still possess a certain structure that Algorithm \ref{algo:path} is able to capture and $\Tt$ performs better than $\Th$ due to a bias-variance tradeoff when $n$ is small. Note that because $K=d$ in both cases, there is no gain asymptotically, so the improvement should decrease with the sample size. This is already apparent for $\T_7$ and $\T_8$ for the sample sizes considered here.

\section{Application to stocks returns} \label{sec:stock}

Consider the daily closing value of all $d = 107$ stocks included in the NASDAQ100 index from January 1 to September 30, 2017 (Source: Google Finance, December 27, 2017). The information on the components of the index (company name, sectors and industries) were taken from \url{www.nasdaq.com}. Rather than clustering together stocks whose returns have a similar distribution, our goal was to identify a block structure in the sample Kendall rank correlation matrix using the methodology proposed here. As is the norm in joint modeling of financial time series, we computed the log returns for series of stocks and applied our methodology to the series of residuals from a fitted stochastic volatility model \citep{Patton:2006,Patton:2012,Remillard:2013}. The stochastic volatility model that yielded satisfactory results was the GARCH$(1,1)$ model. This produced $n = 187$ residuals for each of the $d=107$ stocks. 

Note that even if the stochastic volatility model is appropriate, the residuals are not iid. However, as shown in Corollary 2 in \cite{Remillard:2017}, the empirical copula process based on the residuals has the same asymptotic behavior as the empirical copula process for iid observations. As a result, Eq.~\eqref{eq:asstau} continues to hold with the very same asymptotic variance $\S_\infty$. Even if $\S$ may not be the exact finite-sample variance of $\th$ based on the residuals, the result of \cite{Remillard:2017} implies that Eqs. \eqref{eq:consSigma} and \eqref{eq:Stildeconv} remain true. Consequently, the methodology developed here can be used as is.

For computational purposes and because $n/d = 1.75$ is small, we applied Algorithm~\ref{algo:path} with $w=1$. We then computed $\alpha^{(i)}$ for $i \in \{ 107,\dots,1\}$, as given by Eq.~\eqref{eq:alpha}, again with $w=1$. An excerpt of the plot is shown in Figure~\ref{fig:stocks-alphas}; $\alpha^{(i)}$ for $i \geq 20$ and $i \leq 10$ were essentially equal to $1$ and $0$, respectively. Among the $107$ candidate structures produced by Algorithm \ref{algo:path}, $\mathcal{G}^{(17)}$, $\mathcal{G}^{(16)}$ and $\mathcal{G}^{(15)}$ stand out as the most interesting, because they precede drops in $\alpha^{(i)}$ that are typical of questionable cluster mergers. In effect, the selection procedure is not very sensitive to the arbitrary choice of the value of $\alpha$: any value of $\alpha$ between 0.01 and 0.99 leads to one of $\mathcal{G}^{(17)}$, $\mathcal{G}^{(16)}$ and $\mathcal{G}^{(15)}$ among the 107 structures that were returned by Algorithm \ref{algo:path}. The reduction {of the number of unknown pair-wise correlations}{ is striking; for instance if we use $\mathcal{G}^{(16)}$ we go from $d(d-1)/2=5671$ pair-wise correlations down to $K(K+1)/2=136$.

\begin{figure}[t!]
\centering
	\includegraphics[width=.7\linewidth]{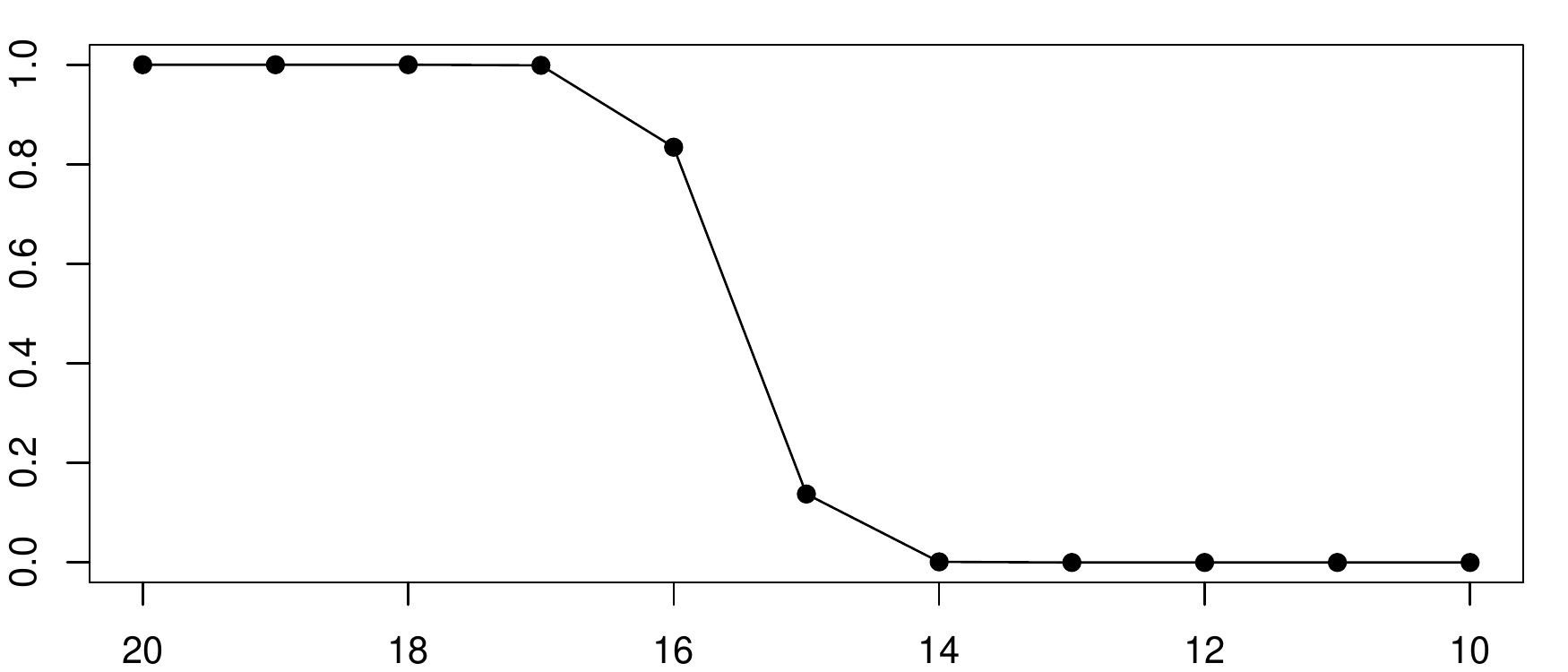}
\caption{The pairs $(i,\alpha^{(i)})$ for $i \in \{ 20,\dots,10\}$ corresponding to the path returned by Algorithm \ref{algo:path} applied to the stock returns residuals.}\label{fig:stocks-alphas}
\end{figure}

The left panel in Figure  \ref{fig:stocks-matrices} shows the empirical Kendall's tau matrix $\Th$ when the variables are labeled consecutively as in the NASDAQ100 index. Clearly, any pattern is very hard to discern from this matrix. The middle panel of the same figure shows $\Th$ once the variables have been relabeled according to the clusters in $\mathcal{G}^{(16)}$, i.e., so that the cluster membership matrix is block-diagonal. Now a structure is beginning to emerge. The right panel displays the improved estimator $\Tt^{(16)}$ that takes the cluster structure into account. The benefit of using Algorithm \ref{algo:path} is clearly visible here.

It may also be of interest to explore whether the clusters in $\mathcal{G}^{(16)}$ have an interpretation; a detailed list of the latter is provided in Appendix D
 of the Online Supplement. Most clusters have intuitive interpretations in terms of business sectors or known affiliations, particularly when the intra-cluster correlation is large.

\section{Conclusion} \label{sec:conclusion}

We have developed a new approach to identify {a block structure within the matrix of Kendall correlations}. Aside from a mild partial exchangeability assumption, the method is completely nonparametric and does not require any additional assumption on the joint or marginal distributions of the variables, insofar as they are continuous. This contribution has the potential to be useful in all areas where {dependence between variables is of interest}. While goodness-of-fit tests can only be performed a posteriori, the proposed method can {serve as a model specification guide at  the early stages of the data analysis}.

We have formally shown that taking advantage of the block structure of the correlation matrix can lead to improved inference on the correlation coefficients. Not only are the new estimators consistent and asymptotically Normal, but their asymptotic variance is smaller than that of the empirical Kendall coefficients. This is important, as {the latter} tend to be extremely noisy when $d$ gets large. Our simulations have shown that the new estimator has better mean squared error in finite samples even when the correlation structure is not known a priori and has to be estimated from data.

We have shown that as the sample size gets large, this algorithm is based on a loss function that will assign negligible loss to merges that agree with the true block structure and large loss to merges that do not, which ensures that the agglomerative process will yield a set of $d$ potential structures that includes the true one. We have also considered how to identify reasonable structures among those proposed by the algorithm. The asymptotic properties of the loss function and the simulation study indicate that this tool is adequate in both large and small samples. Block-exchangeable structures can thus be used as general-purpose shrinkage targets when $n$ is small compared to $d$.

The loss function used in the agglomerative algorithm depends on a variance matrix $\S$ that must be estimated. We have shown that under the Partial Exchangeability Assumption, this matrix and its inverse share a common structure property. We exploit this property to improve the consistent plug-in estimator $\Sh$ of $\S$. The new estimator {$\St_w$  is shown to possess} the same structural properties as $\S$.

Future work may follow from this proposal. The criterion $\alpha^{(i)}$ is intended as a guide and not as a formal test statistic or model selection criterion. Though such ad hoc selection tools are common in hierarchical clustering, perhaps a more formal statistic could be of value. For the shrinkage estimation of $\S$, investigating adaptive weights $w$ that would diminish with $n$ and with each iteration of Algorithm \ref{algo:path} may also lead to slightly improved inference. Working with some, but not all, off-diagonal entries of $\Sh$ may also improve the performance of the procedure. Finally, the Partial Exchangeability Assumption could perhaps be relaxed to allow for more general hierarchical structures.

\section*{Acknowledgments}
We would like to thank the Acting Editor Richard A. Lockhart, the Associate Editor and two reviewers for their careful reading and insightful comments. Special thanks are due to the Editor-in-Chief, Christian Genest, for stimulating conversations, encouragements, and help with copy editing. This research was funded by individual operating grants from the Natural Sciences and Engineering Research Council of Canada to TD (RGPIN-2016-05883) and JGN (RGPIN-2015-06801), a team grant from the Fonds de recherche du Qu\'ebec -- Nature et technologies to TD and JGN (2015--PR--183236), a team grant from the Canadian Statistical Sciences Institute to JGN, and a graduate scholarship from the Fonds de recherche du Qu\'ebec -- Nature et technologies to SP.

\appendix

\section{Estimating $\S$} \label{app:sigma} 
\setcounter{lemma}{0}
\setcounter{definition}{0}
\setcounter{proposition}{0}
\setcounter{equation}{0}
\setcounter{algorithm}{0}
\setcounter{example}{0}
\setcounter{remark}{0}
\setcounter{corollary}{0}
\setcounter{figure}{0}
\setcounter{table}{0}
\renewcommand{\thelemma}{\Alph{section}\arabic{lemma}}
\renewcommand{\theexample}{\Alph{section}\arabic{example}}
\renewcommand{\theremark}{\Alph{section}\arabic{remark}}
\renewcommand{\thedefinition}{\Alph{section}\arabic{definition}}
\renewcommand{\theproposition}{\Alph{section}\arabic{proposition}}
\renewcommand{\thecorollary}{\Alph{section}\arabic{corollary}}

\citet{Genest/Neslehova/BenGhorbal:2011} traced back explicit formulas for the diagonal elements of $\S$ to \citet{Lindeberg:1927,Lindeberg:1929}. Extending the results of \citet{Ehrenberg:1952}, they then provided a formula for the off-diagonal elements of $\S$. Using this formula, given in Eq.~\eqref{eq:sigma} below, we define a plug-in estimator $\Sh$ of $\S$. The estimator and its computation are presented in \ref{subapp:sigma-hat}. In most cases, it is not advisable to use $\Sh$ directly due to a high amount of noise in the estimation. More so because it needs to be inverted, which is known to amplify the estimation error when the original matrix is ill-conditioned. Fortunately, if $\mathcal{G}$ satisfies the \ref{ass:main}, $\S$ has a block structure as well. The latter is described and explained in \ref{subapp:sigma-structure}. We then use this block structure to improve the estimation of $\S$ by averaging entries of $\Sh$ block-wise. The resulting estimate may still contain too much noise to be useful. Inspired by the work of \citet{Ledoit/Wolf:2004}, we thus apply, in addition to the averaging just mentioned, a simple Stein-type shrinkage procedure which depends on the so-called shrinkage intensity $w$. The two shrinkage procedures are presented in \ref{subapp:sigma-tilde}.

\subsection{Plug-in estimator of $\S$} \label{subapp:sigma-hat}
Let $\X$ be a random vector with continuous univariate marginals $F_1,\dots, F_d$ and unique copula $C$, as in Section  \ref{sec:model}. Let $\U = (F_1(X_1),\dots, F_d(X_d))$ and recall that $\U$ has distribution function $C$. For any subset $\{i_1,\dots, i_k \} \subseteq \{1,\dots, d\}$ of indices, let $C_{i_1 \dotsm i_k}$ denote the unique copula  of the marginal $(X_{i_1},\dots, X_{i_k})$ of $\X$. As shown in \citet{Genest/Neslehova/BenGhorbal:2011}, for any $i_1 \neq j_1 \in \{1,\dots, d\}$ and $i_2 \neq j_2 \in \{1,\dots, d\}$,
\begin{align}
 \label{eq:sigma}
{\rm cov}(\Th_{i_1 j_1},\Th_{i_2j_2}) &= \left\{4/n(n-1)\right\}^2 \bigl\{\ n(n-1)(n-2) ( \theta_1 + \theta_2  + \theta_3 + \theta_4)\\\notag
& \quad +\ n(n-1)(\vartheta_1 + \vartheta_2 )\bigr\} -\ \{2(2n-3)/n(n-1)\} (\T_{i_1 j_1}+1)(\T_{i_2j_2}+1),
\end{align}
where
\begin{gather}\label{eq:thetas}
\theta_1 = {\rm E}\{C_{i_1 j_1}(U_{i_1}, U_{j_1}) C_{i_2j_2}(U_{i_2}, U_{j_2})\}, \quad \theta_2 =  {\rm E}\{\bar{C}_{i_1 j_1}(U_{i_1}, U_{j_1}) C_{i_2j_2}(U_{i_2}, U_{j_2})\}, \\\notag
\theta_3 = {\rm E}\{C_{i_1 j_1}(U_{i_1}, U_{j_1})\bar{C}_{i_2j_2}(U_{i_2}, U_{j_2})\}, \quad \theta_4 = {\rm E}\{\bar{C}_{i_1 j_1}(U_{i_1}, U_{j_1})\bar{C}_{i_2j_2}(U_{i_2}, U_{j_2})\},\\
\vartheta_1 = {\rm E}\{C_{i_1 j_1i_2j_2}(U_{i_1},U_{j_1}, U_{i_2}, U_{j_2})\}, \quad \vartheta_2 = {\rm E}\{\tilde{C}_{i_1 j_1i_2j_2}(U_{i_1},U_{j_1}, U_{i_2}, U_{j_2})\},\notag
\end{gather}
and $\bar{C}$ denotes the survival function corresponding to $C$, while $\tilde{C}_{i_1 j_1i_2j_2} = C_{i_1 j_1} - C_{i_1 j_1j_2} - C_{i_1 j_1i_2} + C_{i_1 j_1i_2j_2}$.

For arbitrary $r, s \in \{1,\dots, p\}$, let $\S_{rs} = {\rm cov}(\Th_{i_r j_r},\Th_{i_sj_s})$. From Eqs.~\eqref{eq:sigma}--\eqref{eq:thetas} and the fact that for all $i\neq j \in \{1,\dots, d\}$, ${\rm E}\{C_{ij}(U_i,U_j)\} = {\rm E}\{\bar C_{ij}(U_i,U_j)\}$, we have, as $n\to \infty$, $n \S \to \S_\infty$, where for any $r, s\in\{1,\dots, p\}$, the $(r,s)$th entry of $\S_{\infty}$ is given by
\begin{align}
 \label{eq:sigmainf}
(\S_{\infty})_{rs} & = 16(\theta_1 + \theta_2 + \theta_3 + \theta_4) -4 (\T_{i_r j_r} + 1)(\T_{i_s j_s} + 1) \\\notag
&= 16{\rm cov}\bigl\{C_{i_r, j_r} (U_{i_r}, U_{j_r}) +\bar C_{i_r, j_r} (U_{i_r}, U_{j_r}), C_{i_s, j_s} (U_{i_s}, U_{j_s })+\bar C_{i_s, j_s} (U_{i_s}, U_{j_s }) \bigr\}.
\end{align}

For any $i_1 \neq j_1 \in \{1,\dots, d\}$ and $i_2 \neq j_2 \in \{1,\dots, d\}$, a plug-in estimator of ${\rm cov}(\Th_{i_1 j_1},\Th_{i_2j_2}) $ can be defined by first replacing $\T_{i_1 j_1}$ and $\T_{i_2 j_2}$ by $\Th_{i_1 j_1}$ and $\Th_{i_2 j_2}$, respectively. Furthermore, the quantities in  Eq.~\eqref{eq:thetas} can be estimated as follows. For $k \in \{ 1, 2 \}$, let $\bss{I}^{(k)}$ be an $n \times n$ matrix with entries
\begin{align}\label{eq:matrixI}
	\bss{I}_{rs}^{(k)} = \boldsymbol{1}(X_{r i_k} < X_{s i_k}, X_{r j_k } < X_{s j_k}).
\end{align}
Similarly to the plug-in estimators considered in \cite{BenGhorbal/Genest/Neslehova:2009}, an unbiased estimator of $\theta_1$ is then given by
\begin{align*}
	\hat{\theta}_1 &= \frac{1}{n(n-1)(n-2)} \sum_{r \neq s \neq t} \boldsymbol{1}(X_{r i_1} < X_{s i_1}, X_{r j_1} < X_{s j_1}) \boldsymbol{1}(X_{t i_2} < X_{s i_2}, X_{t j_2 } < X_{s j_2}) = \frac{1}{n(n-1)(n-2)} \sum_{r \neq s \neq t} \bss{I}^{(1)}_{rs} \bss{I}^{(2)}_{ts}.
\end{align*}

Similar formulas can be derived for the other parameters, viz.
\begin{gather*}
\hat{\theta}_2 = \frac{1}{n(n-1)(n-2)} \sum_{r \neq s \neq t} \bss{I}_{rs}^{(1)} \bss{I}_{tr}^{(2)},	\quad \hat{\theta}_3 =  \frac{1}{n(n-1)(n-2)}\sum_{r \neq s \neq t} \bss{I}_{rs}^{(1)} \bss{I}_{st}^{(2)}, \quad \hat{\theta}_4 =  \frac{1}{n(n-1)(n-2)} \sum_{r \neq s \neq t} \bss{I}_{rs}^{(1)} \bss{I}_{rt}^{(2)}, \\
	\hat{\vartheta_1} = \frac{1}{n(n-1)}\sum_{r \neq s} \bss{I}_{rs}^{(1)} \bss{I}_{rs}^{(2)}, \quad	\hat{\vartheta_2} = \frac{1}{n(n-1)} \sum_{r \neq s} \bss{I}_{rs}^{(1)} \bss{I}_{sr}^{(2)}.
\end{gather*}

Given that $\hat\theta_1,\dots, \hat \theta_4$ and $\hat \vartheta_1, \hat \vartheta_2$ are $U$-statistics with square integrable kernels, they are consistent and asymptotically Normal. These properties carry over to the resulting plug-in estimator $\Sh$ of $\S$; in particular, as $n \to \infty$,
\begin{align} 
\label{eq:consSigma}
n \Sh \to \S_{\infty}
\end{align}
in probability. Similarly, one can define a consistent plug-in estimator of $\S_\infty$ by replacing $\theta_1,\dots, \theta_4$ and $\T_{i_1j_1}$ and $\T_{i_2 j_2}$ by their estimators in Eq.~\eqref{eq:sigmainf}.

Finally, note that $\Sh$ can be computed efficiently using matrix products. To this end, consider again arbitrary $i_1 \neq j_1 \in \{1,\dots, d\}$ and $i_2 \neq j_2 \in \{1,\dots, d\}$, and for $k \in \{ 1,2\}$, define $\bss{I}^{(1)}$ through Eq.~\eqref{eq:matrixI} and set $\bss{J}^{(1)} = (\bss{I}^{(1)})^\top$. Furthermore, let $\mathbf{1}$ be the $n$-dimensional vector of ones and $\circ$ denote the Hadamard product. Then because the diagonal entries of $\bss{I}^{(k)}$, $k \in \{ 1, 2\}$ are zero,
\begin{align*}
	n(n-1)(n-2) \sum_{\ell=1}^4 \hat{\theta}_\ell
		&= \sum_{r \neq s \neq t}  ( \bss{I}_{rs}^{(1)} \bss{J}_{st}^{(2)} + \bss{J}_{rs}^{(1)} \bss{J}_{st}^{(2)} + \bss{I}_{rs}^{(1)} \bss{I}_{st}^{(2)} + \bss{J}_{rs}^{(1)} \bss{I}_{st}^{(2)} )= \sum_{r \neq t}  [ (\bss{I}^{(1)} + \bss{J}^{(1)} ) (\bss{I}^{(2)} + \bss{J}^{(2)} ) ]_{rt}\\
	&= \mathbf{1}^\top  (\bss{I}^{(1)} + \bss{J}^{(1)} )  (\bss{I}^{(2)} + \bss{J}^{(2)} ) \mathbf{1} - \mathbf{1}^\top  \{ (\bss{I}^{(1)} + \bss{J}^{(1)} ) \circ  (\bss{I}^{(2)} + \bss{J}^{(2)} ) \} \mathbf{1},\\
n(n-1)(	\hat{\vartheta_1} + \hat{\vartheta_2}) 
	& = \sum_{r \neq s} \bss{J}_{rs}^{(1)}  [ \bss{I}^{(2)} + \bss{J}^{(2)}  ]_{sr}
	 = \mathbf{1}^\top \{ \bss{J}^{(1)} \circ  ( \bss{I}^{(2)} + \bss{J}^{(2)}  )  \} \mathbf{1},
\end{align*}
so that
\begin{align*}
n(n-1)(n-2) ( \hat{\theta}_1 + \cdots + \hat{\theta}_4 ) + n(n-1)(\hat{\vartheta_1} + \hat{\vartheta_2}) = \mathbf{1}^\top (\bss{I}^{(1)} + \bss{J}^{(1)} )  (\bss{I}^{(2)} + \bss{J}^{(2)} )\mathbf{1} - \mathbf{1}^\top  \{ \bss{J}^{(1)} \circ  ( \bss{I}^{(2)} + \bss{J}^{(2)}  )  \} \mathbf{1}.
\end{align*}

\subsection{{Structure of} $\S$ and $\S^{-1}$ implied by $\boldsymbol{\mathcal{G}}$}  \label{subapp:sigma-structure}
Suppose that the Partial Exchangeability Assumption (\ref{ass:main}) holds for some partition $\mathcal{G}$. In this section, we describe the block structure of $\S$ and $\S^{-1}$ induced by the \ref{ass:main}. In \ref{subapp:sigma-tilde} we then exploit this structure to derive the improved estimator $\St$ of $\S$. To this end, let us first focus on a single entry $\S_{rs}$ for some arbitrary fixed $r,s \in \{1,\dots, p\}$. Recall that $\mathcal{G}$ induces the partition $\mathcal{B}_{\mathcal{G}}= \{\mathcal{B}_1,\dots, \mathcal{B}_L\}$ of $\{1,\dots, p\}$, as given in Eq.~\eqref{eq:B-cal}. Eq.~\eqref{eq:sigma} suggests that the value of $\S_{rs}$ depends on the blocks in $\mathcal{B}_{\mathcal{G}}$ to which {$r$ and $s$} belong. To identify these blocks, let
$$
\bs{\Phi}_1 = \{(\ell_1,\ell_2) : 1 \leq \ell_1 \leq \ell_2 \leq L \}
$$
be the set of all ordered pairs of block indices and define the function
\begin{align*} 
\phi : \{1,\dots,p\}^2  \rightarrow \bs{\Phi}_1: 
(r,s)  \mapsto (\ell_1 \wedge \ell_2,\ell_1 \vee \ell_2) \text{ such that } (r,s) \in  \mathcal{B}_{\ell_1} \times \mathcal{B}_{\ell_2} ,
\end{align*}
where for any $a,b\in \mathbb{R}$, $a \wedge b = \min(a,b)$ and $a \vee b = \max(a,b)$.
Now recall from Section~\ref{sec:model} that $(i_r,j_r)$ is a pair of indices such that $\t_r = \T_{i_rj_r}$ and similarly for $(i_s,j_s)$. The value of $\S_{rs}$ does not depend only on $\phi(r,s)$, but also on the overlap between $(i_r,j_r)$ and $(i_s,j_s)$. To account for the latter, let
$
 \bs{\Phi}_2 = \{(k_1,k_2) : 0 \leq k_1 \leq k_2 \leq K \}
$ 
 and define the function $\varphi : \{1,\dots,p\}^2 \rightarrow \bs{\Phi}_2$ given by
\begin{align*} 
	\varphi(r,s) = \begin{cases} (0,0) & \text{ if } \{i_r,j_r\}\cap\{i_s,j_s\} = \emptyset, \\
	(0,k) & \text{ if } \{i_r,j_r\}\cap\{i_s,j_s\} = \{i\}, \; i \in \mathcal{G}_k,\\
		(k_1 \wedge k_2, k_1 \vee k_2) & \text{ if } \{i_r,j_r\}\cap\{i_s,j_s\} = \{i,j\}, \;(i,j) \in \mathcal{G}_{k_1} \times \mathcal{G}_{k_2}.
		\end{cases}
\end{align*}

Using this notation, we introduce, for any $\bs{\ell} = (\ell_1,\ell_2) \in \bs{\Phi}_1$ and $\bs{k} = (k_1,k_2) \in \bs{\Phi}_2$,
\begin{align*}
	\mathcal{C}_{\bs{\ell} \bs{k}} = \{ (r,s) \in \{1,\dots, p\}^2: \; r \le s,\, \phi(r,s) = \bs{\ell}, \,\ \varphi(r,s) = \bs{k} \},
\end{align*}
and set $\mathcal{C}_{\mathcal{G}} = \{ \mathcal{C}_{\bs{\ell} \bs{k}} : (\bs{\ell},\bs{k}) \in \bs{\Phi}_1 \times \bs{\Phi}_2\}$.  Similarly to $\mathcal{T}_\mathcal{G}$, we now define the set $\mathcal{S}_{\mathcal{G}}$ of matrices with a block structure given by $\mathcal{C}_{\mathcal{G}}$, i.e.,
\begin{align*}
	\mathcal{S}_{\mathcal{G}} = \{\SS \in \mathbb{R}^{p \times p}: \SS \text{ symmetric and } \forall_{(\bs{\ell},\bs{k}) \in\bs{\Phi}_1 \times \bs{\Phi}_2}\ (r_1,s_1),(r_2,s_2) \in \mathcal{C}_{\bs{\ell}\bs{k}}\; \Rightarrow \; \SS_{r_1s_1} = \SS_{r_2s_2} \}.
\end{align*}
Finally, for each $r \in \{1,\dots, p\}$, $\ell \in \{1,\dots, L\}$ and $\bs{k} \in \bs{\Phi}_2$, we shall also need the set
\begin{align}
\label{eq:slice}
	\mathcal{C}_{\ell \bs{k}}^{(r)} = \{ s\in \mathcal{B}_\ell : \varphi(r,s)   = \bs{k}\}.
\end{align}

The next proposition confirms that $\S$ and $\S_\infty$ have the block structure induced by $\mathcal{C}_{\mathcal{G}}$.
\begin{proposition} 
\label{prop:equal2}
Suppose that $\mathcal{G}$ is such that the \ref{ass:main} holds. Then $\S \in \mathcal{S}_{\mathcal{G}}$ and $\S_{\infty} \in  \mathcal{S}_{\mathcal{G}}$.
\end{proposition}
\begin{proof}
Fix arbitrary $(\bs{\ell},\bs{k}) \in\bs{\Phi}_1 \times \bs{\Phi}_2$ and $(r_1,s_1),(r_2,s_2) \in \mathcal{C}_{\bs{\ell}\bs{k}}$. To ease the notation, write $(i_1,j_1),(i_2,j_2),(i_3,j_3)$ and $(i_4,j_4)$ instead of $(i_{r_1},j_{r_1}), (i_{s_1},j_{s_1}),(i_{r_2},j_{r_2})$ and $(i_{s_2},j_{s_2})$, respectively. To prove the claim, it suffices to show that all expectations in Eq.~\eqref{eq:sigma} are identical
{when $(i_1,j_1)$ is changed to $(i_{3}, i_{3})$ and $(i_2,j_2)$ to $(i_4,j_4)$, respectively}.  Focusing on $\theta_1$, we need to show
\begin{align} 
\label{eq:expectations}
{\rm E}\{C_{i_1j_1}(U_{i_1},U_{j_1}) C_{i_2j_2}(U_{i_2},U_{j_2})\} = {\rm E}\{C_{i_3j_3}(U_{i_3},U_{j_3})\ C_{i_4j_4}(U_{i_4},U_{j_4})\}.
\end{align}

Let $\bss{I}$ be the set of unique indices from $(i_1,j_1,i_2,j_2)$ and $\bss{I}^{m}=(i_m,j_m)$ for $m \in \{ 1,2\}$. The left-hand side of Eq.~\eqref{eq:expectations} can then be rewritten as
\begin{align}
{\rm E}\{C_{i_1j_1}(U_{i_1},U_{j_1}) C_{i_2j_2}(U_{i_2},U_{j_2})\} &= \int C_{i_1j_1}(u_{i_1},u_{j_1}) C_{i_2j_2}(u_{i_2},u_{j_2})\  \d C_{i_1j_1i_2j_2} = \int C_{\bss{I}^1}(\bs{u}_{\bss{I}^1}) C_{\bss{I}^2}(\bs{u}_{\bss{I}^2})\  \d C_{\bss{I}}. \label{eq:exp-integral}
\end{align}

Now define $\bss{J}$ to be the set of distinct indices from $(i_3,j_3,i_4,j_4)$ and $\bss{J}^{m}=(i_{m+2},j_{m+2})$ for $m \in \{ 1, 2 \}$. Combining the facts that $\phi(r_1,s_1) = \phi(r_2,s_2)$ and $\varphi(r_1,s_1) = \varphi(r_2,s_2)$, we deduce that for any $k\in \{1,\dots, K\}$, $\bss{I}$ and $\bss{J}$ have the same number of entries coming from $\mathcal{G}_k$, with no repetition. The same can be deduced for $\bss{I}^m$ and $\bss{J}^m$ with $m \in \{ 1, 2 \}$. We can therefore use the \ref{ass:main} to replace {$C_{\bss{I}}$ by $C_{\bss{J}}$ and $C_{\bss{I}^{m}}$ by $C_{\bss{J}^{m}}$} for $m \in \{ 1, 2 \}$. Consequently,
\begin{align*}
\int C_{\bss{I}^1}(\bs{u}_{\bss{I}^1}) C_{\bss{I}^2}(\bs{u}_{\bss{I}^2})\  \d C_{\bss{I}} &= \int C_{\bss{J}^1}(\bs{u}_{\bss{J}^1}) C_{\bss{J}^2}(\bs{u}_{\bss{J}^2})\  \d C_{\bss{J}},
\end{align*}
thus showing that the right-hand side of Eq.~\eqref{eq:exp-integral} is indeed equal to the right-hand side of Eq.~\eqref{eq:expectations}.
Equalities for the other quantities $\theta_2,\dots, \theta_4$ and $\vartheta_1,\vartheta_2$ can be shown using the same technique.
\end{proof}
\begin{example} \label{ex:A.1} \em
For $\mathcal{G}$ as given in Eq.~\eqref{eq:ex-G} in Example \ref{ex:1}, there are $L = 6$ sets in $\mathcal{B}_{\mathcal{G}}$, viz.
\begin{align*}
\mathcal{B}_{\mathcal{G}} = \{ \mathcal{B}_{11}, \mathcal{B}_{12}, \mathcal{B}_{13}, \mathcal{B}_{22}, \mathcal{B}_{23}, \mathcal{B}_{33} \} \equiv \{\mathcal{B}_1,\dots, \mathcal{B}_6\},
\end{align*}
{where for $k_1,k_2 \in \{1,\dots, 3\}$, $\mathcal{B}_{k_1k_2}$ is as in Eq.~\eqref{eq:Bk}. The blocks in $\mathcal{B}_{\mathcal{G}}$} are displayed in the 
 left panel of Figure~\ref{fig:ex-T-sigma-structure}. For each $k_1,k_2 \in \{1,2, 3\}$, the cells $(i_r,j_r)$ and $(j_r,i_r)$ for $r \in \mathcal{B}_{k_1k_2}$ are colored the same, emphasizing that the entries are equal.

 \begin{figure}[t!]
	\centering
	\begin{minipage}{.35\textwidth}
	    \centering
    	    \includegraphics[width=1\linewidth]{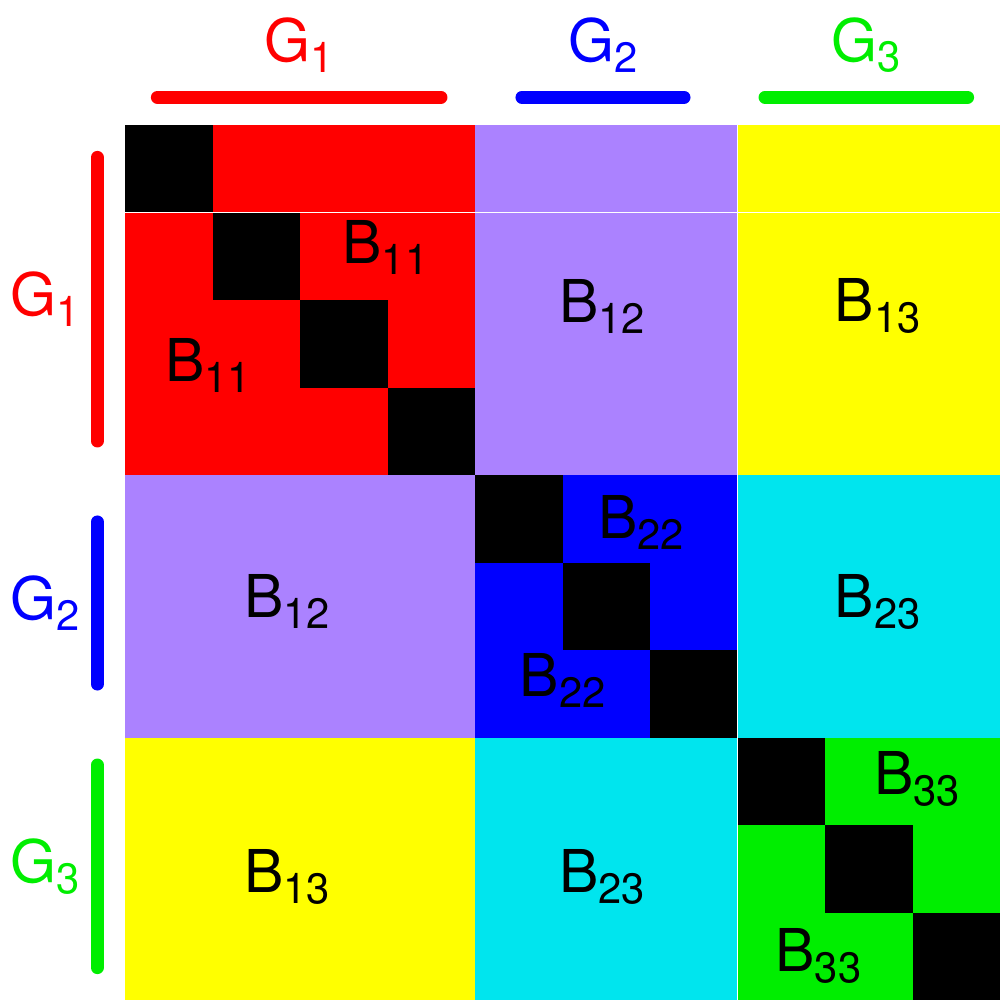}	\\
    	    \hspace{.18\textwidth} $\T$
	\end{minipage}
	\begin{minipage}{.35\textwidth}
	    \centering
    	    \includegraphics[width=1\linewidth]{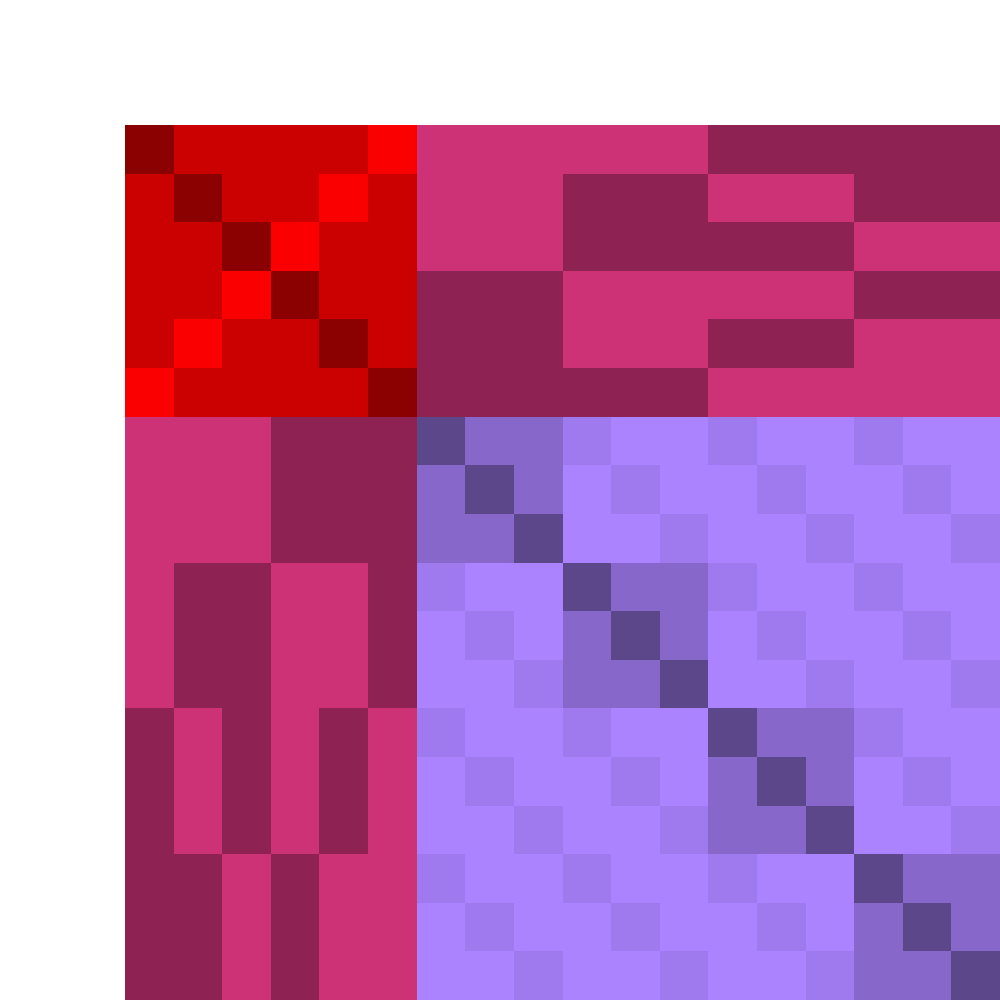}	\\
    	    \hspace{.05\textwidth} Sub-matrix of $\S$
	\end{minipage}
	\caption{The matrix $\T$ (left) and a sub-matrix of $\S$ (right) from Example \ref{ex:A.1}. The cells are tinted so that, in each matrix, all entries sharing the same value are of the same color and color intensity.} \label{fig:ex-T-sigma-structure}
\end{figure}

Given that $p=d(d+1)/2=9(10)/2=45$, the $45\times45$ matrix $\S$ is more cumbersome to visualize. To see its structure more clearly, we vectorize $\T$ not as in Eq.~\eqref{eq:vectorization}, but rather block by block. For instance the first $18$ entries of $\t$ are the six {entries in $\T$ corresponding to} $\mathcal{B}_{11}{\equiv\mathcal{B}_1 = \{1,\dots, 6\}}$ followed by the 12 {entries in $\T$ corresponding to} $\mathcal{B}_{12}{\equiv\mathcal{B}_2 =\{ 7, \dots,18 \}}$, i.e., $(\t_1,\ldots,\t_{18}) = (\T_{1,2},\T_{1,3},\T_{1,4}, \T_{2,3}, \T_{2,4}, \T_{3,4}, \T_{1,5},\dots, \T_{4,7})$. 

The $18\times 18$ dimensional sub-matrix of $\S$ displaying the pair-wise covariances of $\th_{1},\dots, \th_{18}$ is showed in the right panel of Figure~\ref{fig:ex-T-sigma-structure}. Distinct values are depicted using different colors and color intensity. The colors represent distinct values of $\phi$: for all $r,s\in\{1,\dots, 18\}$, the cell $(r,s)$ appears in red, magenta and violet if $\phi(r,s)$ equals $(1,1)$, $(1,2)$ and $(2,2)$, respectively. In other words, the color blocks are induced by $\mathcal{B}_{11}$ and $\mathcal{B}_{12}$. 

Next, notice that in each colored block, the values of $\S$ can differ, and this is depicted through different color intensity. This is because for any $r,s\in \{1,\dots, p\}$, the value of $\S_{rs}$ also depends on $\varphi(r,s)$. For example, the red block in the top left corner contains three distinct values, for if $r,s \in \{1,\dots,6\}$, $\{i_r,j_r\} \cap \{i_s,j_s\}$ is either empty (red), contains one element from $\mathcal{G}_1$ (light red), or contains two elements from $\mathcal{G}_1$ (dark red). To illustrate, take $r$ successively equal to $1, 2, 6$; then $(i_1,j_1) = (1,2)$, $(i_2,j_2) = (1,3)$, and $(i_{6},j_{6}) = (3,4)$, respectively. Consequently, $\phi(1,1) = \phi(1,2) = \phi(1,6)  = (1,1)$, and indeed the entries $\S_{11}$, $\S_{12}$ and $\S_{16}$ are red. However, $\varphi(1,1) = (1,1)$, $\varphi(1,2) = (0,1)$ and $\varphi(1,6) = (0,0)$, which is why $\S_{11}$, $\S_{12}$ and $\S_{16}$ are dark red, red and light red, respectively. One can indeed verify that $\S_{11}\neq\S_{12}\neq\S_{16}$. The block structure $\S$ is thus described by both $\phi$ and $\varphi$; the sets in $\mathcal{C}_{\mathcal{G}}$  correspond to the cells above the main diagonal of $\S$ with the same color and intensity.

Finally, the right panel in Figure~\ref{fig:ex-T-sigma-structure} can be used to visualize the sets $\mathcal{C}^{(r)}_{\ell \bs{k}}$ defined in Eq.~\eqref{eq:slice}. For a given $r \in \{1,\dots, p\}$, the union of the sets $\mathcal{C}^{(r)}_{\ell \bs{k}}$ for $\ell \in \{1,\dots, L\}$ and $\bs{k} \in \bs{\Phi}_2$ may be identified with the $r$th row (or equivalently the $r$th column) of $\S$: the index $\ell$ determines the color and $\bs{k}$ the intensity in that row (column). To illustrate in the context of this example, pick $r=1$ and $\ell=1$, say. Then
$\mathcal{C}^{(1)}_{1 (1,1)} = \{1\}$, $\mathcal{C}^{(1)}_{1 (0,1)} = \{2,3,4,5\}$,  and $\mathcal{C}^{(1)}_{1 (0,0)}=\{6\}$,
while $\mathcal{C}^{(1)}_{1 \bs{k}} =\emptyset$ for any other $\bs{k} \in \bs{\Phi}_2$. Note that  $\{\mathcal{C}^{(1)}_{1 (1,1)},\mathcal{C}^{(1)}_{1 (0,1)}, \mathcal{C}^{(1)}_{1 (0,0)}\}$ is a partition of $\mathcal{B}_{11}$.

\end{example}
The following proposition establishes that the structure of $\S^{-1}$ is the same as that of $\S$.
\begin{proposition} \label{prop:S-inv}
Suppose that $\mathcal{G}$ is a partition for which the \ref{ass:main} holds. An invertible matrix $\SS$ is an element of $ \mathcal{S}_{\mathcal{G}}$ if and only if $\SS^{-1} \in \mathcal{S}_{\mathcal{G}}$. That is, $\mathcal{S}_{\mathcal{G}}$ is closed under inversion.
\end{proposition}
\begin{proof}
From the Cayley--Hamilton Theorem as stated, e.g., on p.~583 of~\citep{Harville:2008}, $\SS$ satisfies its characteristic equation
\begin{align*}
	\SS^p + \sum_{s = 1}^{p-1} c_s \SS^{s} + (-1)^p |\SS| \Id{p} = 0,
\end{align*}
for some known coefficients $c_s$, $s\in\{1,\ldots,p-1\}$. As a consequence, the inverse of a $p \times p$ matrix $\SS$ can be represented by a linear function of its $p - 1$ first powers, viz.
\begin{align*}
	\SS^{-1} = \frac{1}{|\SS|} \sum_{s = 1}^p c_s \SS^{s-1}.
\end{align*}

Therefore, it suffices to show that if $\SS , \bss{Q} \in \mathcal{S}_{\mathcal{G}}  $ and $\SS\bss{Q}$ is symmetric, $ \SS \bss{Q} \in \mathcal{S}_{\mathcal{G}}$.
To this end, fix an arbitrary $\bs{\ell}=(\ell_1,\ell_2) \in \bs{\Phi}_1$, $\bs{k}=(k_1,k_2) \in \bs{\Phi}_2$, and arbitrary pairs  $(r_1,s_1),(r_2,s_2) \in \mathcal{C}_{\bs{\ell}\bs{k}}$. To show that
$[\SS \bss{Q}]_{r_1s_1} = [\SS \bss{Q}]_{r_2 s_2}$, first note that because $\SS\bss{Q}$ is symmetric by assumption, it can be assumed, without loss of generality, that $r_1,r_2 \in \mathcal{B}_{\ell_1}$ and $s_1,s_2 \in \mathcal{B}_{\ell_2}$. For any $\bs{\ell}^* \in \bs{\Phi}_1$ and any $\bs{k}^* \in \bs{\Phi}_2$, let $\SS^{\bs{\ell}^*\bs{k}^*}$ and $\bss{Q}^{\bs{\ell}^*\bs{k}^*}$ denote the unique values such that $\SS_{rs} = \SS^{\bs{\ell}^*\bs{k}^*}$ and $\bss{Q}_{rs} = \bss{Q}^{\bs{\ell}^*\bs{k}^*}$ whenever  $(r,s) \in \mathcal{C}_{\bs{\ell}^*\bs{k}^*}$.
Because $\mathcal{B}_{\mathcal{G}}$ given by Eq.~\eqref{eq:B-cal} is a partition of $\{1,\dots, p\}$, we can write
\begin{equation}
\label{eq:A10}
[\SS \bss{Q}]_{r_1s_1} = \sum_{t=1}^p \SS_{r_1 t} \bss{Q}_{t s_1} = \sum_{\ell=1}^L \sum_{t \in \mathcal{B}_{\ell}} \SS_{r_1 t} \bss{Q}_{t s_1}.
\end{equation}

For any fixed $\ell \in \{1,\dots, d\}$ and $t \in \mathcal{B}_\ell$, $\phi(r_1,t) = (\ell \wedge \ell_1, \ell \vee \ell_1) \equiv \bs{\ell}_{1\ell}$ and $\phi(s_1,t) = (\ell \wedge \ell_2, \ell \vee \ell_2)\equiv \bs{\ell}_{2\ell}$. However, $\SS_{r_1t}$ and $\bss{Q}_{t s_1}$ also depend on $\varphi(r_1,t)$ and $\varphi(s_1,t)$, and this requires further partitioning of $\mathcal{B}_{\ell}$ by means of the sets defined in Eq.~\eqref{eq:slice}. Specifically,
\[
	\mathcal{B}_\ell = \bigcup_{\bs{k_1}, \bs{k_2} \in \bs{\Phi}_2} \mathcal{C}_{\ell\bs{k}_1}^{(r_1)} \cap \mathcal{C}_{\ell\bs{k}_2}^{(s_1)}.
\]
Clearly, the sets $\mathcal{C}_{\ell\bs{k}_1}^{(r_1)} \cap \mathcal{C}_{\ell\bs{k}_2}^{(s_1)}$ are disjoint for distinct $\bs{k}_1, \bs{k}_2 \in \bs{\Phi}_2$, and for any given $\bs{k}_1, \bs{k}_2 \in \bs{\Phi}_2$ and $t \in \mathcal{C}_{\ell\bs{k}_1}^{(r_1)} \cap \mathcal{C}_{\ell\bs{k}_2}^{(s_1)}$, $\varphi(r_1,t) = \bs{k}_1$ and $\varphi(t,s_1) = \bs{k}_2$ so that $\SS_{r_1t} = \SS^{\bs{\ell}_{1\ell} \bs{k}_1}$ and $\bss{Q}_{ts_1} = \bss{Q}^{{\bs{\ell}_{2\ell} \bs{k}_2}}$. Consequently, the last expression in Eq.~\eqref{eq:A10} can be rewritten as
\begin{equation}\label{eq:A11}
\sum_{\ell=1}^{L}\sum_{\bs{k}_1,\bs{k}_2 \in \bs{\Phi}_2} \sum_{t \in \mathcal{C}_{\ell\bs{k}_1}^{(r_1)} \cap \mathcal{C}_{\ell\bs{k}_2}^{(s_1)}} \SS^{\bs{\ell}_{1\ell} \bs{k}_1}\bss{Q}^{{\bs{\ell}_{2\ell} \bs{k}_2}} = \sum_{\ell=1}^{L}\sum_{\bs{k}_1,\bs{k}_2 \in \bs{\Phi}_2} \left|\mathcal{C}_{\ell\bs{k}_1}^{(r_1)} \cap \mathcal{C}_{\ell\bs{k}_2}^{(s_1)}\right | \SS^{\bs{\ell}_{1\ell} \bs{k}_1}\bss{Q}^{{\bs{\ell}_{2\ell} \bs{k}_2}}.
\end{equation}
Now for any $\ell \in \{1,\dots, L\}$ and any $\bs{k}_1,\bs{k}_2 \in \bs{\Phi}_2$, Lemma~\ref{lem:inter} gives that $|\mathcal{C}_{\ell\bs{k}_1}^{(r_1)} \cap \mathcal{C}_{\ell\bs{k}_2}^{(s_1)} | = |\mathcal{C}_{\ell\bs{k}_1}^{(r_2)} \cap \mathcal{C}_{\ell\bs{k}_2}^{(s_2)}|$. Furthermore, for any $t \in \mathcal{C}_{\ell\bs{k}_1}^{(r_2)} \cap \mathcal{C}_{\ell\bs{k}_2}^{(s_1)}$, $\phi(r_2, t) = \bs{\ell}_{1\ell}$, $\phi(t,s_2) = \bs{\ell}_{2\ell}$ and $\varphi(r_2,t) = \bs{k}_1$, $\varphi(t,s_2) = \bs{k}_2$, so that $\SS_{r_2 t} = \SS^{\bs{\ell}_{1\ell} \bs{k}_1}$, and $\bss{Q}_{ts_2} = \bss{Q}^{{\bs{\ell}_{2\ell} \bs{k}_2}}$.  Consequently, the right-hand side in Eq.~\eqref{eq:A11} equals
\[
 \sum_{\ell=1}^{L}\sum_{\bs{k}_1,\bs{k}_2 \in \bs{\Phi}_2} \left|\mathcal{C}_{\ell\bs{k}_1}^{(r_2)} \cap \mathcal{C}_{\ell\bs{k}_2}^{(s_2)}\right | \SS^{\bs{\ell}_{1\ell} \bs{k}_1}\bss{Q}^{{\bs{\ell}_{2\ell} \bs{k}_2}} =  \sum_{\ell=1}^{L}\sum_{\bs{k}_1,\bs{k}_2 \in \bs{\Phi}_2} \sum_{t \in \mathcal{C}_{\ell\bs{k}_1}^{(r_2)} \cap \mathcal{C}_{\ell\bs{k}_2}^{(s_2)}} \SS_{r_2 t}\bss{Q}_{ts_2} = [\SS \bss{Q}]_{r_2s_2} ,
\]
as claimed.
\end{proof}
In view of Proposition \ref{prop:equal2}, the following result follows directly from Proposition \ref{prop:S-inv}.
\begin{corollary}\label{cor:A1}
Suppose that $\mathcal{G}$ is such that the \ref{ass:main} holds. If $\S$ is invertible, then $\S^{-1} \in  \mathcal{S}_{\mathcal{G}}$. Similarly, if $\S_\infty$ is invertible, then $\S_\infty^{-1} \in \mathcal{S}_{\mathcal{G}}$.
\end{corollary}

\subsection{An improved estimator of $\S$} \label{subapp:sigma-tilde}

Throughout this section, assume that $\mathcal{G}$ is a partition of $\{1,\dots, d\}$ such that the \ref{ass:main} holds. The empirical estimator $\Sh$ defined in \ref{subapp:sigma-hat} does not exploit the  structural information provided by $\mathcal{C}_{\mathcal{G}}$. It is natural to think that, as was the case for $\t$, we can improve the estimation of $\S$ by averaging its entries with respect to the sets in $\mathcal{C}_{\mathcal{G}}$. But we can do even better by exploiting the following decomposition of $\S$. To this end, write
\begin{align} \label{eq:Sigma-decomp}
	\S = \Ta - [2(2n - 3)/\{n(n-1)\}] (\t + \mathbf{1})(\t + \mathbf{1})^{\top},
\end{align}
where  $\mathbf{1}$ is the $p$-dimensional vector of ones and {$\Ta$ is a $p \times p$ matrix gathering the terms involving $\theta_1,\dots, \theta_4$ and $\vartheta_1,\vartheta_2$ in Eq.~\eqref{eq:sigma}. It easily follows from the proof of Proposition \ref{prop:equal2}  that $\Ta \in \mathcal{S}_{\mathcal{G}}$ as well as $(\t + \mathbf{1})(\t + \mathbf{1})^{\top} \in \mathcal{S}_{\mathcal{G}}$. However, the structure of $(\t + \mathbf{1})(\t + \mathbf{1})^{\top}$ is even simpler, because the overlaps between pairs of indices described by the function $\varphi$ need not be taken into account. Specifically,}  $(\t + \mathbf{1})(\t + \mathbf{1})^{\top} \in \mathcal{T}_{\mathcal{B}_{\mathcal{G}}} \subset \mathcal{S}_{\mathcal{G}}$, where
\[
{\mathcal{T}_{\mathcal{B}_{\mathcal{G}}} = \{ \bss{R} \in \mathbb{R}^{p \times p} : \forall_{\ell_1,\ell_2 \in \{1,\dots, L\}} \, r_1,r_2 \in \mathcal{B}_{\ell_1} \text{ and } s_1,s_2 \in \mathcal{B}_{\ell_2} \; \Rightarrow \; \bss{R}_{r_1 s_1} = \bss{R}_{r_2s_2}\}. }
\]
That is, $(\t + \mathbf{1})(\t + \mathbf{1})^{\top}$ possesses a block structure similar to $\T$, but defined in accordance with the clustering $\mathcal{B}_{\mathcal{G}}$ instead of $\mathcal{G}$. In particular, $(\t + \mathbf{1})(\t + \mathbf{1})^{\top}$ possesses $L$ diagonal blocks, as opposed to $K$ diagonal blocks for $\T$ including diagonal blocks that correspond to clusters in $\mathcal{G}$ of size $1$. The decomposition \eqref{eq:Sigma-decomp} of Eq. $\S$ is illustrated next.
\begin{example}\label{ex:A.2}\em
The decomposition of $\S$ from Example \ref{ex:A.1} according to Eq.~\eqref{eq:Sigma-decomp} is depicted in Figure~\ref{fig:Sigma-decomp}. The matrix $\S$ clearly inherits its structure from $\Ta \in \mathcal{S}_{\mathcal{G}}$; the structure of $(\t + \mathbf{1})(\t + \mathbf{1})^{\top}$ is considerably simpler.
\begin{figure}[t!]
	\centering
	\begin{minipage}{.25\textwidth}
	    \centering
    	    \includegraphics[width=1\linewidth]{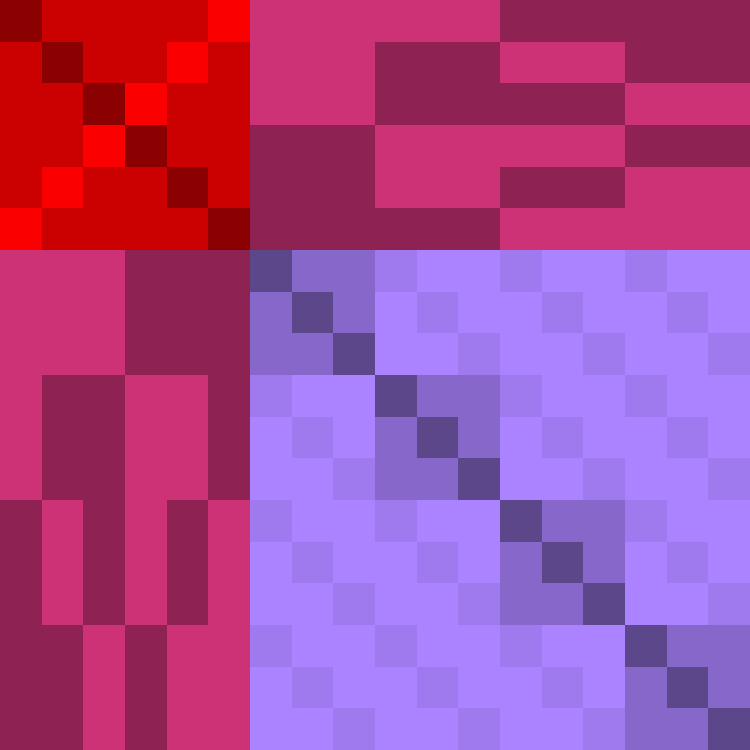}\\
    	    $\phantom{)^\top}\S$
	\end{minipage}
	\begin{minipage}{.25\textwidth}
	    \centering
    	    \includegraphics[width=1\linewidth]{example-theta}	\\
    	    $\phantom{)^\top}\Ta$
	\end{minipage}
	\begin{minipage}{.25\textwidth}
	    \centering
    	   \includegraphics[width=1\linewidth]{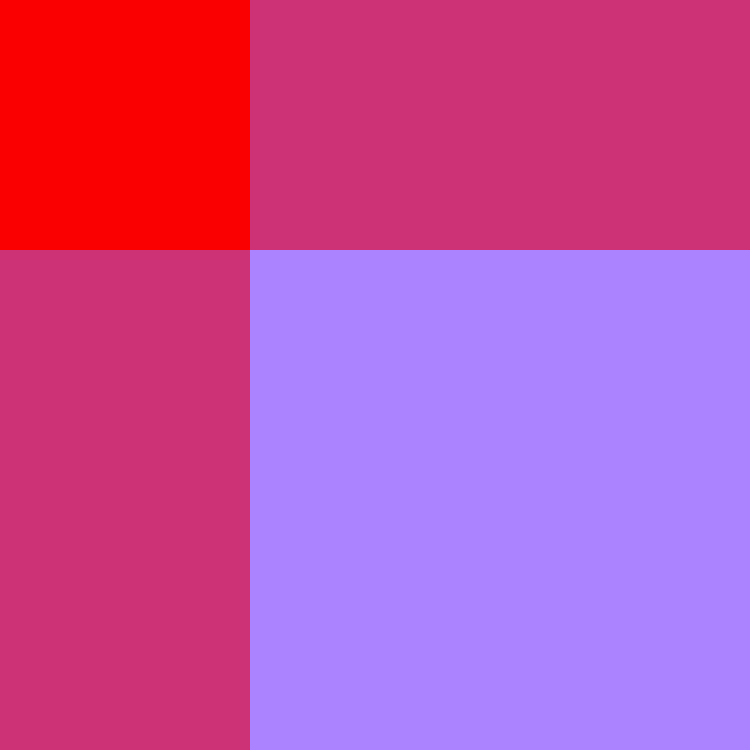}	\\
    	   $(\t+ \mathbf{1}) (\t+ \mathbf{1})^\top$
	\end{minipage}
	\caption{Submatrices of $\S$, $\Ta$ and $(\t + \mathbf{1})(\t + \mathbf{1})^\top$ from Example \ref{ex:A.2}. The same vectorization of $\T$ as in Example \ref{ex:A.1} is used.} \label{fig:Sigma-decomp}
\end{figure}
\end{example}
Let $\hat\Ta$ be {the plug-in empirical estimator of $\Ta$, defined by replacing, for each $(r,s) \in \{1,\dots, p\}^2$, the parameters $\theta_1,\dots, \theta_4$ and $\vartheta_1,\vartheta_2$ by their empirical estimates given in \ref{subapp:sigma-hat}.
Because $\Ta \in \mathcal{S}_{\mathcal{G}}$, we now define the improved estimator $\tilde{\Ta}$, which is in $\mathcal{S}_{\mathcal{G}}$ by construction. First, the upper triangular (including the diagonal) entries of $\tilde{\Ta}$ are simply the entries of $\hat \Ta$ averaged out over each block in $\mathcal{C}_{\mathcal{G}}$. Second, because $\tilde{\Ta}$ is symmetric, its lower triangular entries are obtained by symmetry.
Furthermore, let $\tt = \tt (\th | \mathcal{G})$ be as in Eq.~\eqref{eq:maxlike} and estimate $(\t + \mathbf{1}) (\t + \mathbf{1})^{\top}$ by $(\tt + \mathbf{1}) (\tt + \mathbf{1})^{\top} \in \mathcal{T}_{\mathcal{B}_\mathcal{G}}$, so that even more averaging is employed. The resulting estimator $\S$ is then
\begin{equation*}
\St = \tilde{\Ta} - [2(2n - 3)/\{n(n-1)\}] (\tt + \mathbf{1})(\tt + \mathbf{1})^{\top}.
\end{equation*}
Clearly, $\St \in \mathcal{S}_{\mathcal{G}}$. Using Eq.~\eqref{eq:sigmainf}, write $\S_\infty = \Ta_\infty - 4 (\t +1)(\t + 1)^\top$, where $\Ta_\infty$ has entries $16\,(\theta_1+\dots+ \theta_4)$. Note that from the proof of Proposition \ref{prop:equal2}, $\Ta_\infty \in \mathcal{S}_{\mathcal{G}}$. Hence, as $n \to \infty$, $n \tilde \Ta \to \Ta_\infty$ element-wise in probability. Furthermore, given that $\tt$ is a consistent estimator of $\t$ as per Theorem \ref{thm:reduced-variance}, $(\tt + \mathbf{1})(\tt + \mathbf{1})^{\top} \to (\t + \mathbf{1})(\t + \mathbf{1})^{\top}$. Put together,
\begin{align}
\label{eq:Stildeconv}
n\St \to \S_\infty
\end{align}
element-wise in probability. For the shrinkage estimator $\St_w=(1-w)\St + w\St_{\rm diag}$ it is easy to show that $\St_w \in \mathcal{S}_{\mathcal{G}}$. Because $\St$ is consistent, so is $\St_w$, as long as $w\to0$
as $n\to\infty$.

Finally, note that the estimators $\Sh$, $\St$, and $\St_w$ may not be positive definite, in particular when $n$ is small. For the methodology presented in this paper, $\St_w$ needs to be invertible, and this is often easily achieved by a larger value of the shrinkage parameter $w$ when $n$ is small. In the data illustration and simulation study conducted in this paper, no problems with invertibility were encountered. If the estimator of $\S$ further needs to be positive semi-definite, one can project any of the estimators  $\Sh$, $\St$, and $\St_w$ to the cone of positive semi-definite matrices; this projection, say $\bar \S$, can be computed using the alternative direction of multipliers algorithm as described, e.g., in Appendix~A of \cite{Datta/Zou:2017}. Since the projection onto the cone of positive semidefinite matrices is a continuous mapping, $\bar \S$ will be consistent. 

\section{Proofs} \label{app:proofs} 

\setcounter{lemma}{0}
\setcounter{definition}{0}
\setcounter{proposition}{0}
\setcounter{equation}{0}
\setcounter{algorithm}{0}
\setcounter{example}{0}
\setcounter{remark}{0}
\setcounter{corollary}{0}
\setcounter{figure}{0}
\setcounter{table}{0}

This section is divided as follows. \ref{subapp:lemmas} contains six auxiliary lemmas that pertain
to the structure of ${\mathcal{S}_{\mathcal{G}}}$ and that are invoked in the proofs of Theorem \ref{thm:maxlike}, Theorem \ref{thm:reduced-variance} and Proposition \ref{prop:stop}. These proofs, along with the proofs of Proposition \ref{prop:qform} and Corollary \ref{cor:on-path}, are detailed in \ref{subapp:B.2}.
 \ref{subapp:proof-T-inv} contains additional results on the structural properties of the inverse of correlation matrices used in Section~\ref{sec:bijective}.

\subsection{Description of $\bs{\mathcal{S}_{\mathcal{G}}}$} \label{subapp:lemmas} 
We first present three auxiliary lemmas that pertain to the cardinality of the sets $\mathcal{C}_{\ell \bs{k}}^{(r)}$ defined in Eq.~\eqref{eq:slice}.

\begin{lemma} \label{lem:slice-cardinality}
{Let $\mathcal{G}$ be an arbitrary partition of $\{1,\dots, d\}$ and $\mathcal{B}_{\mathcal{G}}$ as in Eq.~\eqref{eq:B-cal}. Assume that $r,s \in \mathcal{B}_\ell$ for some $\ell\in\{1,\dots, L\}$. Then for all $\lambda \in \{1,\dots, L\}$, and all $\bs{\kappa} \in \bs{\Phi}_2$,
$|\mathcal{C}_{\lambda\bs{\kappa}}^{(r)}| = |\mathcal{C}_{\lambda\bs{\kappa}}^{(s)}|$.}
\end{lemma}
\begin{remark}\em
Before proceeding with the proof of Lemma~\ref{lem:slice-cardinality}, let us illustrate the claim with the right panel of Figure~\ref{fig:ex-T-sigma-structure}. Because $\S \in \mathcal{S}_{\mathcal{G}}$, the sets  $\mathcal{C}_{\ell \bs{k}}^{(r)}$ for a fixed $r$ can be identified with the $r$th row of $\S$; different colors and intensities correspond to different values of $\ell$ and $\bs{k}$, respectively; see also Example \ref{ex:A.1}. Lemma~\ref{lem:slice-cardinality} implies that, for example, the number of cells with the same color and intensity in the first six rows (corresponding to ${\mathcal{B}}_1$) is the same.
\end{remark}
\begin{proof}
Fix arbitrary $\mathcal{G}$, $\ell\in\{1,\dots, L\}$ and $r,s \in \mathcal{B}_\ell$.  We will prove the assertion by showing that there exists a bijection $h:\{1,\dots,p\} \rightarrow \{1,\dots,p\}$ such that for all $\lambda \in \{1,\dots, L\}$ and $\bs{\kappa} \in \bs{\Phi}_2$,
\begin{align}\label{eq:implhrs}
	t \in \mathcal{C}_{\lambda\bs{\kappa}}^{(r)} \quad \Leftrightarrow \quad h(t) \in \mathcal{C}_{\lambda\bs{\kappa}}^{(s)},
\end{align}
for then obviously $|\mathcal{C}_{\lambda\bs{\kappa}}^{(r)}| = |\mathcal{C}_{\lambda\bs{\kappa}}^{(s)}|$. To this end, first identify $k_1,k_2, \in \{1,\dots, K\}$  such that $(i_r, j_r) \in \mathcal{G}_{k_1}\times  \mathcal{G}_{k_2}$. Note that then $\mathcal{B}_\ell = \mathcal{B}_{(k_1 \wedge k_2)(k_1 \vee k_2)}$. Because $s \in \mathcal{B}_\ell$ by assumption, either $(i_s, j_s) \in \mathcal{G}_{k_1}\times  \mathcal{G}_{k_2}$ or $(i_s, j_s) \in \mathcal{G}_{k_2}\times  \mathcal{G}_{k_1}$.

First assume that $(i_s, j_s) \in \mathcal{G}_{k_1}\times  \mathcal{G}_{k_2}$. Let $\pi$ be any permutation such that
\begin{equation}\label{eq:pibasic}
\forall_{i \in \{1,\dots, d\}} \; \forall_{k \in \{1,\dots, K\}} \quad i \in \mathcal{G}_k \: \Leftrightarrow\: \pi(i)\in \mathcal{G}_k,
\end{equation}
and further such that
\begin{equation}\label{eq:pi1}
\pi(i_r) = i_s, \quad \pi(j_r) = j_s.
\end{equation}
Because  $i_r \neq j_r$ and $i_s \neq j_s$, and $i_r, i_s \in \mathcal{G}_{k_1}$, $j_r, j_s \in \mathcal{G}_{k_2}$, such a permutation always exists, although it is generally not unique. Now define $h$ by
\begin{align}\label{eq:hrs}
h : \{1,\dots, p\} \to \{1,\dots, p\}: 
t  \mapsto t^* \text{ such that } i_{t^*} = \pi(i_t) \wedge \pi(j_t) \text{ and }  j_{t^*} = \pi(i_t) \vee \pi(j_t).
\end{align}

First observe that $h$ is well defined because for any $t \in \{1,\dots, p\}$, $i_t < j_t$, $\pi(i_t) \neq \pi(j_t)$ and hence $i_{t^*} < j_{t^*}$, so that $t^*$ indeed exists. Furthermore, $h$ is a bijection. This is because, for any $t^* \in \{1,\dots, p\}$ and $t$ such that $i_t = \pi^{-1}(i_{t^*}) \wedge  \pi^{-1}(j_{t^*})$ and $j_t = \pi^{-1}(i_{t^*}) \vee  \pi^{-1}(j_{t^*})$, $h(t) = t^*$. Next, Eq.~\eqref{eq:pibasic} implies that for any $\lambda \in \{1,\dots, L\}$, $t \in \mathcal{B}_\lambda$ if and only if $h(t) \in \mathcal{B}_\lambda$. To prove that $h$ satisfies Eq.~\eqref{eq:implhrs}, it thus remains to show that for any $\bs{\kappa} \in \bs{\Phi}_2$, $\varphi(r,t) = \bs{\kappa}$ if and only if $\varphi \{ s,h(t)\} = \bs{\kappa}$. This follows from the following facts, each of which is an immediate consequence of Eq.~\eqref{eq:pi1}:
\begin{itemize}
\item[(i)] For any $ i\in \{1,\dots,d\}$, $\{i\} \cap\{i_r, j_r\} = \emptyset$ if and only if $\{\pi(i)\} \cap\{i_s, j_s\} = \emptyset$.
\item[(ii)] For any $ i\in \{1,\dots,d\}$, $\{i\} \cap\{i_r, j_r\} = \{i_r\}$ if and only if $\{\pi(i)\} \cap\{i_s, j_s\} = \{i_s\}$; $i_r$ and $i_s$ are in the same cluster $\mathcal{G}_{k_1}$.
\item[(iii)] For any $ i\in \{1,\dots,d\}$, $\{i\} \cap\{i_r, j_r\} = \{j_r\}$ if and only if $\{\pi(i)\} \cap\{i_s, j_s\} = \{j_s\}$; $j_r$ and $j_s$ are in the same cluster~$\mathcal{G}_{k_2}$.
\end{itemize}

This concludes the proof in the case when $(i_s, j_s) \in \mathcal{G}_{k_1}\times  \mathcal{G}_{k_2}$. When $(i_s, j_s) \in \mathcal{G}_{k_2}\times  \mathcal{G}_{k_1}$, one can proceed analogously by constructing $h$ from an arbitrary fixed  permutation $\pi$ satisfying Eq.~\eqref{eq:pibasic} and such that $\pi(i_r) = j_s$ and $\pi(j_r) = i_s$.
\end{proof}
\begin{lemma} \label{lem:ratio}
Let $\mathcal{G}$ be an arbitrary partition of $\{1,\dots, d\}$ and $\mathcal{B}_{\mathcal{G}}$ as in Eq.~\eqref{eq:B-cal}. Then for any $\ell_1,\ell_2\in\{1,\dots, L\}$, $r \in \mathcal{B}_{\ell_1}$, $s \in \mathcal{B}_{\ell_2}$, and $\bs{k} \in \bs{\Phi}_2$, ${|\mathcal{C}_{\ell_2\bs{k}}^{(r)}|}/{|\mathcal{B}_{\ell_2}|} ={|\mathcal{C}_{\ell_1\bs{k}}^{(s)}|}/{|\mathcal{B}_{\ell_1}|}$.
\end{lemma}
\begin{proof} Fix arbitrary $\ell_1,\ell_2\in\{1,\dots, L\}$, $r \in \mathcal{B}_{\ell_1}$, $s \in \mathcal{B}_{\ell_2}$, and $\bs{k} \in \bs{\Phi}_2$.
The case $\ell_1 = \ell_2$ trivially follows from Lemma~\ref{lem:slice-cardinality}.  For $\ell_1 \neq \ell_2$, set $\bs{\ell} = (\ell_1 \wedge \ell_2, \ell_1 \vee \ell_2)$. Using Lemma~\ref{lem:slice-cardinality} again, we then have that
\begin{align*}
|\mathcal{C}_{\bs{\ell} \bs{k}}| = \sum_{t \in \mathcal{B}_{\ell_1}} |\mathcal{C}_{\ell_2\bs{k}}^{(t)}| = \sum_{t \in \mathcal{B}_{\ell_1}} |\mathcal{C}_{\ell_2\bs{k}}^{(r)}| = |\mathcal{C}_{\ell_2\bs{k}}^{(r)}|\ |\mathcal{B}_{\ell_1}|.
\end{align*}
Similarly, summing over elements in $\mathcal{B}_{\ell_2}$, $|\mathcal{C}_{\bs{\ell} \bs{k}}| = |\mathcal{C}_{\ell_1\bs{k}}^{(s)}|\ |\mathcal{B}_{\ell_2}|$, which proves the claim.
\end{proof}

\begin{lemma} \label{lem:inter}
Let $\mathcal{G}$ be an arbitrary partition of $\{1,\dots, d\}$ and $\mathcal{B}_{\mathcal{G}}$ as in Eq.~\eqref{eq:B-cal}. Assume that $r_1, r_2 \in \mathcal{B}_{\ell_1}$ and $s_1, s_2 \in \mathcal{B}_{\ell_2}$ for some $\bs{\ell} = (\ell_1,\ell_2) \in \bs{\Phi}_1$. Further assume that $(r_1,s_1), (r_2, s_2) \in \mathcal{C}_{\bs{\ell} \bs{k}}$ for some $\bs{k} \in \bs{\Phi}_2$. Then for all $\lambda \in \{1,\dots, L\}$, and all $\bs{\kappa}_1, \bs{\kappa}_2 \in \bs{\Phi}_2$, one has $|\mathcal{C}_{\lambda\bs{\kappa}_1}^{(r_1)} \cap \mathcal{C}_{\lambda\bs{\kappa}_2}^{(s_1)}| = |\mathcal{C}_{\lambda\bs{\kappa}_1}^{(r_2)} \cap \mathcal{C}_{\lambda\bs{\kappa}_2}^{(s_2)}|$, i.e.,
\begin{align*}
	|\{ t \in \mathcal{B}_\lambda : \varphi(r_1, t) = \bs{\kappa}_1 ,  \varphi(s_1, t) = \bs{\kappa}_2\}| = |\{ t \in \mathcal{B}_\lambda : \varphi(r_2, t) = \bs{\kappa}_1 ,  \varphi(s_2, t) = \bs{\kappa}_2\}|.
\end{align*}
\end{lemma}
\begin{proof}
Proceed similarly as in the proof of Lemma~\ref{lem:slice-cardinality} and define a function $h$ through Eq.~\eqref{eq:hrs} from a certain permutation $\pi$ of $(1,\dots, d)$ satisfying Eq.~\eqref{eq:pibasic}. As argued in the proof of Lemma~\ref{lem:slice-cardinality}, $h$ is then a well-defined bijection such that for all $\lambda \in \{1,\dots, L\}$, $t \in \mathcal{B}_\lambda$ if and only if $h(t) \in \mathcal{B}_\lambda$. If $h$ is further such that for each $t \in \{1,\dots, p\}$ and $\bs{\kappa}_1, \bs{\kappa}_2 \in \bs{\Phi}_2$,
\begin{equation}\label{eq:hreq}
\varphi(r_1, t) = \bs{\kappa}_1   \quad \Leftrightarrow \quad \varphi \{r_2, h(t)\} = \bs{\kappa}_1  \qquad \text{and} \qquad
\varphi(s_1, t) = \bs{\kappa}_2   \quad \Leftrightarrow \quad \varphi \{ s_2, h(t)\} = \bs{\kappa}_2,
\end{equation}
then one has that, for all $t \in \{1,\dots, p\}$, $t \in \mathcal{C}_{\lambda\bs{\kappa}_1}^{(r_1)} \cap \mathcal{C}_{\lambda\bs{\kappa}_2}^{(s_1)}$ if and only if $h(t) \in \mathcal{C}_{\lambda\bs{\kappa}_1}^{(r_2)} \cap \mathcal{C}_{\lambda\bs{\kappa}_2}^{(s_2)}$, and consequently that $|\mathcal{C}_{\lambda\bs{\kappa}_1}^{(r_1)} \cap \mathcal{C}_{\lambda\bs{\kappa}_2}^{(s_1)}| = |\mathcal{C}_{\lambda\bs{\kappa}_1}^{(r_2)} \cap \mathcal{C}_{\lambda\bs{\kappa}_2}^{(s_2)}|$, as claimed.

In contrast to the proof of  Lemma~\ref{lem:slice-cardinality}, the properties of the permutation $\pi$ require a more cumbersome case distinction. To this end, fix $\bs{\ell} = (\ell_1,\ell_2) \in \bs{\Phi}_1$, $\bs{k} \in \bs{\Phi}_2$, and $(r_1,s_1), (r_2, s_2) \in \mathcal{C}_{\bs{\ell} \bs{k}}$  such that $r_1, r_2 \in \mathcal{B}_{\ell_1}$ and $s_1, s_2 \in \mathcal{B}_{\ell_2}$.  Now let $k_{11},k_{12},k_{21},k_{22} \in \{1,\dots, K\}$ be such $\mathcal{B}_{\ell_1} = \mathcal{B}_{k_{11}k_{12}}$  and $\mathcal{B}_{\ell_2} = \mathcal{B}_{k_{21}k_{22}}$. Without loss of generality, assume that
\begin{align}\label{eq:wlog}
	(i_{r_1},j_{r_1}),(i_{r_2},j_{r_2}) \in \mathcal{G}_{k_{11}} \times \mathcal{G}_{k_{12}} \quad \text{and} \quad (i_{s_1},j_{s_1}),(i_{s_2},j_{s_2}) \in \mathcal{G}_{k_{21}} \times \mathcal{G}_{k_{22}}.
\end{align}
\smallskip
{\it Case I.} $\bs{k} = \varphi(r_1,s_1) = \varphi(r_2,s_2) = (0,0)$. Here,
\begin{equation}\label{eq:card4}
|\{i_{r_1}, j_{r_1}, i_{s_1}, j_{s_1}\}| = |\{i_{r_2}, j_{r_2}, i_{s_2}, j_{s_2}\}| =4,
\end{equation}
and $\pi$ is an arbitrary fixed permutation with the property \eqref{eq:pibasic} and such that
\begin{equation}\label{eq:pi2}
\pi(i_{r_1}) = i_{r_2}, \quad \pi(j_{r_1}) = j_{r_2}, \quad \pi(i_{s_1}) = i_{s_2}, \quad \pi(j_{s_1}) = j_{s_2}.
\end{equation}

Such a permutation exists because of Eqs. \eqref{eq:wlog} and \eqref{eq:card4}, although it is generally not unique. Because of Eq.~\eqref{eq:pi2}, one has that, for any $i \in \{1,\dots, d\}$,
\begin{itemize}
\item[(i)] $\{i\} \cap \{i_{r_1}, j_{r_1}\} = \emptyset$ if and only if $\{\pi(i)\} \cap \{i_{r_2}, j_{r_2}\} = \emptyset$, and  $\{i\} \cap \{i_{s_1}, j_{s_1}\} = \emptyset$ if and only if $\{\pi(i)\} \cap \{i_{s_2}, j_{s_2}\} = \emptyset$;
\item[(ii)] $\{i\} \cap\{i_{r_1}, j_{r_1}\} = \{i_{r_1}\}$ if and only if $\{\pi(i)\} \cap\{i_{r_2}, j_{r_2}\} = \{i_{r_2}\}$; $i_{r_1}$ and $i_{r_2}$ are elements of the same cluster $\mathcal{G}_{k_{11}}$.
\item[(iii)] $\{i\} \cap\{i_{r_1}, j_{r_1}\} = \{j_{r_1}\}$ if and only if $\{\pi(i)\} \cap\{i_{r_2}, j_{r_2}\} = \{j_{r_2}\}$; $j_{r_1}$ and $j_{r_2}$ are elements of the same cluster $\mathcal{G}_{k_{12}}$.
\item[(iv)] $\{i\} \cap\{i_{s_1}, j_{s_1}\} = \{i_{s_1}\}$ if and only if $\{\pi(i)\} \cap\{i_{s_2}, j_{s_2}\} = \{i_{s_2}\}$; $i_{s_1}$ and $i_{s_2}$ are elements of the same cluster $\mathcal{G}_{k_{21}}$.
\item[(v)] $\{i\} \cap\{i_{s_1}, j_{s_1}\} = \{j_{s_1}\}$ if and only if $\{\pi(i)\} \cap\{i_{s_2}, j_{s_2}\} = \{j_{s_2}\}$; $j_{s_1}$ and $j_{s_2}$ are elements of the same cluster $\mathcal{G}_{k_{22}}$.
\end{itemize}
Hence $h$ fulfills Eq.~\eqref{eq:hreq} and the proof is complete in this case.

\medskip
\noindent
{\it Case II.} $\bs{k} =\varphi(r_1,s_1) = \varphi(r_2,s_2) = (0,k)$ for some $k\in \{1,\dots, K\}$. In this case,
\begin{align*}
|\{i_{r_1}, j_{r_1}, i_{s_1}, j_{s_1}\}| = |\{i_{r_2}, j_{r_2}, i_{s_2}, j_{s_2}\}| =3.
\end{align*}

Observe that the three distinct elements of $\{i_{r_1}, j_{r_1}, i_{s_1}, j_{s_1}\}$ are $\{a_1,a_2, a_3\}$, say, such that $a_1 \in \{i_{r_1}, j_{r_1}\} \cap \{i_{s_1}, j_{s_1}\}$, $a_2 \in \{i_{r_1}, j_{r_1}\}\setminus \{a_1\}$, and $a_3 \in \{i_{s_1}, j_{s_1}\}\setminus \{a_1\}$. Similarly, $\{i_{r_2}, j_{r_2}, i_{s_2}, j_{s_2}\}=\{b_1,b_2,b_3\}$ such that $b_1 \in \{i_{r_2}, j_{r_2}\} \cap \{i_{s_2}, j_{s_2}\}$, $b_2 \in\{i_{r_2}, j_{r_2}\}\setminus \{b_1\} $, and $b_3 \in \{i_{s_2}, j_{s_2}\}\setminus \{b_1\}$. Note that one must then have
$$
\{i_{r_1}, j_{r_1}\} = \{a_1, a_2\}, \quad \{i_{s_1}, j_{s_1}\}=\{a_1, a_3\}, \quad \{i_{r_2}, j_{r_2}\} = \{b_1, b_2\}, \quad \{i_{s_2}, j_{s_2}\} = \{b_1, b_3\}.
$$

Furthermore, for $m \in \{1,2,3\}$, $a_m$ and $b_m$ are members of the same cluster. Indeed, given that $\bs{k} = (0,k)$, $a_1, b_1 \in \mathcal{G}_k$. From Eq.~\eqref{eq:wlog} it further follows that $a_2,b_2$ are in $\mathcal{G}_{k_{11}}$ and $\mathcal{G}_{k_{12}}$, respectively, if $a_2 = i_{r_1}$, $b_2 = i_{r_2}$, and $a_2 = j_{r_1}$, $b_2=j_{r_2}$, respectively. If $a_{2} = i_{r_1}$ and $b_2 = j_{r_2}$, then $a_1 = j_{r_1}$ and $b_1 = i_{r_2}$ and Eq.~\eqref{eq:wlog} and the fact that $a_1, b_1 \in \mathcal{G}_k$ together imply that $k =k_{11} = k_{12}$. Hence, $a_2, b_2 \in \mathcal{G}_{k}$. Similarly, if  $a_{2} = j_{r_1}$ and $b_2 = i_{r_2}$, we also have that $a_2, b_2 \in \mathcal{G}_{k}$. The verification of the fact that $a_3$ and $b_3$ are in the same cluster is analogous.

Now let $\pi$ be any permutation with the property \eqref{eq:pibasic} and such that
\begin{equation}\label{eq:pi3}
\pi(a_1) = b_1, \quad \pi(a_2) = b_2, \quad \pi(a_3) = b_3.
\end{equation}
Such a permutation indeed exists, although it is not unique; existence is guaranteed because the $a_m$s are all distinct and for $m \in \{1,2,3\}$, $a_m$ and $b_m$ are members of the same cluster. Because of Eq.~\eqref{eq:pi3}, one has that, for any $i \in \{1,\dots, d\}$,
\begin{itemize}
\item[(i)] $\{i\} \cap \{i_{r_1}, j_{r_1}\} = \emptyset$ if and only if $\{\pi(i)\} \cap \{i_{r_2}, j_{r_2}\} = \emptyset$, and  $\{i\} \cap \{i_{s_1}, j_{s_1}\} = \emptyset$ if and only if $\{\pi(i)\} \cap \{i_{s_2}, j_{s_2}\} = \emptyset$;
\item[(ii)] $\{i\} \cap\{i_{r_1}, j_{r_1}\} = \{a_1\}$ if and only if $\{\pi(i)\} \cap\{i_{r_2}, j_{r_2}\} = \{b_1\}$; $a_1$ and $b_1$ are elements of the same cluster $\mathcal{G}_{k}$.
\item[(iii)] $\{i\} \cap\{i_{r_1}, j_{r_1}\} = \{a_2\}$ if and only if $\{\pi(i)\} \cap\{i_{r_2}, j_{r_2}\} = \{b_2\}$; $a_2$ and $b_2$ are elements of the same cluster $\mathcal{G}_{k_{11}}$, $\mathcal{G}_{k_{12}}$ or $\mathcal{G}_{k}$, as the case may be.
\item[(iv)] $\{i\} \cap\{i_{s_1}, j_{s_1}\} = \{a_1\}$ if and only if $\{\pi(i)\} \cap\{i_{s_2}, j_{s_2}\} = \{b_1\}$; $a_1$ and $b_1$ are elements of the same cluster $\mathcal{G}_{k}$.
\item[(v)] $\{i\} \cap\{i_{s_1}, j_{s_1}\} = \{a_3\}$ if and only if $\{\pi(i)\} \cap\{i_{s_2}, j_{s_2}\} = \{b_3\}$; $a_3$ and $b_3$ are elements of the same cluster $\mathcal{G}_{k_{21}}$, $\mathcal{G}_{k_{22}}$ or $\mathcal{G}_{k}$, as the case may be.
\end{itemize}
Hence $h$ fulfills Eq.~\eqref{eq:hreq} and the proof is complete in this case.

\medskip
\noindent
{\it Case III.} $\bs{k} =\varphi(r_1,s_1) = \varphi(r_2,s_2) = (k_1,k_2)$ where $k_1,k_2 > 0$. In this case,
$|\{i_{r_1}, j_{r_1}, i_{s_1}, j_{s_1}\}| = |\{i_{r_2}, j_{r_2}, i_{s_2}, j_{s_2}\}| =2$.
Write $\{i_{r_1}, j_{r_1}, i_{s_1}, j_{s_1}\} = \{a_1 ,a_2\}$, and $\{i_{r_2}, j_{r_2}, i_{s_2}, j_{s_2}\} = \{b_1, b_2\}$, with $a_1, b_1 \in \mathcal{G}_{k_1}$ and $a_2, b_2 \in \mathcal{G}_{k_2}$. Now let $\pi$ be any permutation with the property \eqref{eq:pibasic} and such that
\begin{equation}\label{eq:pi4}
\pi(a_1) = b_1, \quad \pi(a_2) = b_2.
\end{equation}
Because of Eq.~\eqref{eq:pi4} and the fact that $ \{a_1 ,a_2\}=\{i_{r_1}, j_{r_1}\} \cap \{ i_{s_1}, j_{s_1}\}$, and $\{b_1, b_2\} = \{i_{r_2}, j_{r_2}\} \cap \{i_{s_2}, j_{s_2}\}$, one has, for any $i \in \{1,\dots, d\}$,
\begin{itemize}
\item[(i)] $\{i\} \cap \{i_{r_1}, j_{r_1}\} = \emptyset$ if and only if $\{\pi(i)\} \cap \{i_{r_2}, j_{r_2}\} = \emptyset$, and  $\{i\} \cap \{i_{s_1}, j_{s_1}\} = \emptyset$ if and only if $\{\pi(i)\} \cap \{i_{s_2}, j_{s_2}\} = \emptyset$;
\item[(ii)] $\{i\} \cap\{i_{r_1}, j_{r_1}\} = \{a_1\}$ if and only if $\{\pi(i)\} \cap\{i_{r_2}, j_{r_2}\} = \{b_1\}$; $a_1$ and $b_1$ are elements of the same cluster $\mathcal{G}_{k_1}$.
\item[(iii)] $\{i\} \cap\{i_{r_1}, j_{r_1}\} = \{a_2\}$ if and only if $\{\pi(i)\} \cap\{i_{r_2}, j_{r_2}\} = \{b_2\}$; $a_2$ and $b_2$ are elements of the same cluster $\mathcal{G}_{k_{2}}$.
\item[(iv)] $\{i\} \cap\{i_{s_1}, j_{s_1}\} = \{a_1\}$ if and only if $\{\pi(i)\} \cap\{i_{s_2}, j_{s_2}\} = \{b_1\}$; $a_1$ and $b_1$ are elements of the same cluster $\mathcal{G}_{k_1}$.
\item[(v)] $\{i\} \cap\{i_{s_1}, j_{s_1}\} = \{a_2\}$ if and only if $\{\pi(i)\} \cap\{i_{s_2}, j_{s_2}\} = \{b_2\}$; $a_2$ and $b_2$ are elements of the same cluster $\mathcal{G}_{k_2}$.
\end{itemize}
Hence $h$ fulfills Eq.~\eqref{eq:hreq} and the proof is complete in this case as well.
\end{proof}

The following lemma is the cornerstone of the proof of Theorem \ref{thm:maxlike}.
\begin{lemma} \label{lem:maxlike}
{Let $\mathcal{G}$ be an arbitrary partition of $\{1,\dots, d\}$. If $\SS \in \mathcal{S}_{\mathcal{G}}$} and $\B$ is the block membership matrix corresponding to $\mathcal{B}_{\mathcal{G}}$ in Eq.~\eqref{eq:B-cal}, then
$\B^\top \SS = \B^\top \SS \B \B^{+}$.

\end{lemma}
\begin{proof}
First note that, for any $\ell\in \{1,\dots, L\}$ and $r \in \{1,\dots, p\}$,
\begin{equation}\label{eq:B3}
[\B^\top \SS]_{\ell r} = \sum_{s=1}^p \B_{s \ell} \SS_{s r} = \sum_{s \in \mathcal{B}_\ell} \SS_{s r}
\end{equation}
and that $[\B^+]_{\ell r} = \boldsymbol{1}(r \in \mathcal{B}_\ell) |\mathcal{B}_\ell |^{-1}$. Also, for any $s \in \mathcal{B_\ell}$,
\begin{align} \label{eq:BBplus}
	[\B \B^+]_{rs} =  \boldsymbol{1}(r \in \mathcal{B}_\ell)|\mathcal{B}_\ell |^{-1}.
\end{align}
{Now fix an arbitrary $\ell\in \{1,\dots, L\}$ and $r \in \{1,\dots, p\}$ and find $\ell^*$ so that $r \in \mathcal{B}_{\ell^*}$. Then from Eqs. \eqref{eq:B3} and \eqref{eq:BBplus},
\begin{equation}\label{eq:B5}
[\B^\top \SS \B \B^{+}]_{\ell r} = \sum_{s=1}^p [\B^\top \SS]_{\ell s} [\B\B^+]_{s r}
	= \sum_{s=1}^p\sum_{t \in \mathcal{B}_\ell} \{\SS_{t s}/ |\mathcal{B}_{\ell^*}|\} \boldsymbol{1}(s \in \mathcal{B}_{\ell^*})
	= \sum_{s \in \mathcal{B}_{\ell^*}} \sum_{t \in \mathcal{B}_\ell} \{\SS_{t s}/|\mathcal{B}_{\ell^*}|\} = \sum_{t \in \mathcal{B}_\ell}  \frac{1}{{ |\mathcal{B}_{\ell^*}|}}\sum_{s \in \mathcal{B}_{\ell^*}} \SS_{t s}.
\end{equation}
Now set $\bs{\ell} = (\ell \wedge \ell^*, \ell \vee \ell^*)$  and for any $\bs{k} \in \bs{\Phi}_2$, let $\SS^{\bs{\ell}\bs{k}}$ denote the unique value such that $\SS_{ts} = \SS^{\bs{\ell}\bs{k}}$ whenever  $(t,s) \in \mathcal{C}_{\bs{\ell}\bs{k}}$. Because $\mathcal{B}_{\ell^*} = \cup_{\bs{k} \in \bs{\Phi}_2} \mathcal{C}_{\ell^* \bs{k}}^{(t)}$ for any $t \in \mathcal{B}_\ell$ and $\SS \in \mathcal{S}_{\mathcal{G}}$ by assumption, one has
$$
\sum_{s \in \mathcal{B}_{\ell^*}} \SS_{t s} = \sum_{\bs{k} \in \bs{\Phi}_2} \sum_{s \in \mathcal{C}_{\ell^* \bs{k}}^{(t)}} \SS_{t s} = \sum_{\bs{k} \in \bs{\Phi}_2} \sum_{s \in \mathcal{C}_{\ell^* \bs{k}}^{(t)}} \SS^{\bs{\ell} \bs{k}} = \sum_{\bs{k} \in \bs{\Phi}_2} |\mathcal{C}_{\ell^* \bs{k}}^{(t)}| \SS^{\bs{\ell} \bs{k}}.
$$
Using this and Lemma~\ref{lem:ratio}, the right-hand side in Eq.~\eqref{eq:B5} equals
$$
\sum_{t \in \mathcal{B}_\ell} \sum_{\bs{k} \in \bs{\Phi}_2}  \{|\mathcal{C}_{\ell^* \bs{k}}^{(t)}| /{ |\mathcal{B}_{\ell^*}|}\}\SS^{\bs{\ell} \bs{k}} = \sum_{t \in \mathcal{B}_\ell} \sum_{\bs{k} \in \bs{\Phi}_2}  \{|\mathcal{C}_{\ell \bs{k}}^{(r)}| /{ |\mathcal{B}_{\ell}|}\}\SS^{\bs{\ell} \bs{k}} =   \sum_{\bs{k} \in \bs{\Phi}_2}  |\mathcal{C}_{\ell \bs{k}}^{(r)}| \SS^{\bs{\ell} \bs{k}}.
$$
Now use the fact that $\mathcal{B}_{\ell} = \cup_{\bs{k} \in \bs{\Phi}_2} \mathcal{C}_{\ell \bs{k}}^{(r)}$ and that $\SS \in \mathcal{S}_{\mathcal{G}}$ to rewrite the right-hand side as
$$
\sum_{\bs{k} \in \bs{\Phi}_2}  \sum_{t \in \mathcal{C}_{\ell \bs{k}}^{(r)}} \SS^{\bs{\ell} \bs{k}} = \sum_{t \in \mathcal{B}_\ell} \SS_{tr},
$$
which is equal to the right-hand side in Eq.~\eqref{eq:B3}, as claimed.
}
\end{proof}
The next result is an immediate consequence of the properties of the Moore--Penrose pseudo-inverse.
\begin{lemma} \label{lem:idem}
{Let $\mathcal{G}$ be an arbitrary partition of $\{1,\dots, d\}$, $\B$ the block membership matrix corresponding to $\mathcal{B}_{\mathcal{G}}$ in Eq.~\eqref{eq:B-cal}, and $\G = \B\B^+$. Then $\G$ and   $\Id{p} - \G$ are idempotent, i.e., $\G \G = \G$ and $(\Id{p} - \G)(\Id{p} - \G) = \Id{p} - \G$.}
\end{lemma}
\begin{lemma} \label{lem:gamma-sigma-rel}
Let $\mathcal{G}$ be an arbitrary partition of $\{1,\dots, d\}$, $\B$ the block membership matrix corresponding to $\mathcal{B}_{\mathcal{G}}$ in Eq.~\eqref{eq:B-cal}, and $\G = \B\B^+$. Then for any $\SS \in \mathcal{S}_{\mathcal{G}}$, $\G \SS = \G \SS \G = \SS \G$.
\end{lemma}
\begin{proof}
{For arbitrary $\bs{\ell} \in \bs{\Phi}_1$ and $\bs{k} \in \bs{\Phi}_2$, let $\SS^{\bs{\ell}\bs{k}}$ denote the unique value such that $\SS_{rs} = \SS^{\bs{\ell}\bs{k}}$ whenever  $(r,s) \in \mathcal{C}_{\bs{\ell}\bs{k}}$. Now fix arbitrary $r,s \in \{1,\dots, p\}$ and let $\ell_1,\ell_2$ be such that $r \in \mathcal{B}_{\ell_1}$ and $s \in \mathcal{B}_{\ell_2}$. Set $\bs{\ell} = (\ell_1 \wedge \ell_2, \ell_1 \vee \ell_2)$.} From Eq.~\eqref{eq:BBplus} and Lemma~\ref{lem:ratio} we can conclude that
\begin{align*}
	[\G \bss{S}]_{rs} =  \sum_{t \in \mathcal{B}_{\ell_1}} \{\SS_{t s}/|\mathcal{B}_{\ell_1}|\} = \sum_{\bs{k}\in \bs{\Phi}_2} \sum_{t \in \mathcal{C}_{\ell_1\bs{k}}^{(s)}} \{\SS^{\bs{\ell} \bs{k}}/|\mathcal{B}_{\ell_1}|\} = \sum_{\bs{k}\in \bs{\Phi}_2} \{|\mathcal{C}_{\ell_1\bs{k}}^{(s)}|/|\mathcal{B}_{\ell_1}|\}\ \SS^{\bs{\ell} \bs{k}}  =  \sum_{\bs{k}\in \bs{\Phi}_2} \{|\mathcal{C}_{\ell_2\bs\varphi}^{(r)}|/|\mathcal{B}_{\ell_2}|\} \SS^{\bs{\ell} \bs{k}} = [\bss{S} \G]_{rs},
\end{align*}
proving that $\G \SS = \SS \G$. Furthermore,  $\G \bss{S} = \bss{S} \G$ and Lemma~\ref{lem:idem} together imply that $\G \G \bss{S} = \G \bss{S} \G$. Using the idempotence of $\G$ again, this simplifies to $ \G \bss{S} = \G \bss{S} \G$.
\end{proof}
\begin{remark} \label{rem:gamma-sigma-rel}
{Suppose that $\mathcal{G}$ satisfies the \ref{ass:main}. Proposition \ref{prop:equal2} and Corollary \ref{cor:A1} along with Lemma~\ref{lem:idem} imply that $(\Id{p} - \G) \bss{S} = (\Id{p} - \G) \bss{S} (\Id{p} - \G) = \bss{S} (\Id{p} - \G)$ holds for both $\bss{S}=\S$ and $\bss{S}=\S^{-1}$}.
\end{remark}

\subsection{Proofs from Sections \ref{sec:improved} and \ref{sec:learning}}\label{subapp:B.2}
 
\begin{proof}[Proof of Theorem \ref{thm:maxlike}]
Observe that any $\boldsymbol{t} \in \mathcal{T}_{\mathcal{G}}$ can be expressed as $\B\boldsymbol{t}^*$ for some $\boldsymbol{t}^* \in [-1,1]^L$. Solving the equation $\partial \mathcal{L}(\B \boldsymbol{t}^*|\th,\S) /\partial \boldsymbol{t}^* = \mathbf{0}$ for $\boldsymbol{t}^*$ gives as the unique solution
$\tt^* = ( \B^\top \S^{-1} \B)^{-1} \B^\top \S^{-1} \th$.
From Proposition \ref{prop:S-inv} and Lemma~\ref{lem:maxlike}, we have that $\B^\top \S^{-1} = \B^\top \S^{-1} \B \B^{+}$. Consequently, $\tt^*= \B^{+} \th$ and $\tt = \G \th$ given that $\tt = \B \tt^*$. The expression for $\tt(\th|\mathcal{G})_r$ follows immediately from Eq.~\eqref{eq:BBplus} in the proof of Lemma~\ref{lem:maxlike}. 
\end{proof}

\begin{proof}[Proof of Theorem \ref{thm:reduced-variance}] Because $\tt - \t = \G(\th - \t)$, (i) follows from Eq.~\eqref{eq:asstau} and the fact that $\G \S_\infty \G = \G \S_\infty$ by Lemma~\ref{lem:gamma-sigma-rel}. From Lemma~\ref{lem:idem}, $\S_\infty - \G\S_\infty = (\Id{p} - \G) \S_\infty = (\Id{p} - \G)^\top \S_\infty (\Id{p} - \G)$, where $\Id{p}$ denotes the $p\times p$ identity matrix.  Consequently, (ii) follows from the fact that $\S_\infty$ is nonnegative definite. 
\end{proof}

\begin{proof}[Proof of Proposition \ref{prop:uniqueG}]
Suppose that $\mathcal{G}=\{\mathcal{G}_1,\ldots, \mathcal{G}_K\}$ and $\mathcal{G}^*=\{\mathcal{G}_1^*,\ldots, \mathcal{G}_K^*\}$ are distinct partitions with $|\mathcal{G}| = |\mathcal{G}^*|$ that satisfy the \ref{ass:main}.  The claim follows from the fact that there must then necessarily exist a coarser partition for which the \ref{ass:main} also holds. Assume, without loss of generality, that  $A = \mathcal{G}_1 \cap \mathcal{G}_{1}^* \neq \emptyset$ and $B = \mathcal{G}_2 \cap \mathcal{G}_{1}^* \neq \emptyset$.  We will now show that in that case, $\{\mathcal{G}_1 \cup \mathcal{G}_2,\ldots, \mathcal{G}_K\}$ satisfies the \ref{ass:main}. To this end, let $\U \sim C$ and fix two arbitrary indices $i \in \mathcal{G}_1$, $j \in \mathcal{G}_2$. It suffices to prove that for the permutation $\pi$ given by $\pi(i) = j$, $\pi(j) =i$ and $\pi(k)=k$ for all $k \not \in \{i,j\}$, 
\begin{equation}\label{eq:Upi}
(U_1, \ldots U_d) \eqdis (U_{\pi(1)}, \ldots, U_{\pi(d)}).
\end{equation}

\medskip
\noindent
{\it Case I.} If $i \in A$ and $j \in B$, Eq.~\eqref{eq:Upi} follows at once from the fact that $\mathcal{G}^*$ satisfies the \ref{ass:main}. 

\medskip
\noindent
{\it Case II.} If $i \in A$ and $j \not \in B$, pick an arbitrary $j^* \in B$ and define the permutations $\pi_1$ and $\pi_2$ by $\pi_1(j) = j^*$, $\pi_1(j^*) = j$, $\pi_1(k) =k$ for all $k \not \in\{j,j^*\}$, and $\pi_2(i) = j^*$, $\pi_2(j^*) = i$, $\pi_2(k) = k$ for all $k \not \in \{i, j^*\}$. Because $\mathcal{G}$ and $\mathcal{G}^*$ satisfy the \ref{ass:main}, 
\begin{equation}\label{eq:Upii}
(U_1, \ldots U_d) \eqdis (U_{\pi_\ell(1)}, \ldots, U_{\pi_\ell(d)})
\end{equation}
for $\ell \in \{1,2\}$. The validity of Eq.~\eqref{eq:Upi} follows from the fact that $\pi_1 \circ \pi_2 \circ \pi_1 = \pi$.

\medskip
\noindent
{\it Case III.} If $i \not \in A$ and $j \in B$, Eq.~\eqref{eq:Upi}  follows analogously as in the previous case by picking $i^* \in A$. 

\medskip
\noindent
{\it Case IV.} If $i \not \in A$ and $j \not \in B$, pick $i^* \in A$, $j^* \in B$ and define the permutations $\pi_1$, $\pi_2$ and $\pi_3$ as follows. For all $k \not \in \{i, i^*\}$, $\pi_1(k) = k$ while $\pi_1(i) = i^*$ and $\pi_1(i^*) =i$; for all $k \not \in \{j, j^*\}$, $\pi_2(k) =k$ while $\pi_2(j) = j^*$ and $\pi(j^*) = j$; for all $k \not \in \{i^*, j^*\}$, $\pi_3(k) =k$ while $\pi_3(i^*) = j^*$ and $\pi_3(j^*) = i^*$. Then Eq.~\eqref{eq:Upii} holds for $\ell \in \{1,2\}$ and $\ell=3$ because $\mathcal{G}$ and $\mathcal{G}^*$ satisfy the \ref{ass:main}, respectively. The identity \eqref{eq:Upi} follows because $\pi= \pi_1\circ \pi_3 \circ \pi_2 \circ \pi_3 \circ \pi_1$.
\end{proof}

\begin{proof}[Proof of Proposition \ref{prop:qform}]
First, note that as $n\to \infty$, $\S^{-1}/n \to \S^{-1}_\infty$ in probability given that $A \mapsto A^{-1}$ is a continuous map for nonsingular matrices \citep{Stewart:1969}. Now write
$(\th - \tt)^\top (\hat\S^{-1}/n) (\th - \tt)= \th^\top(\Id{p} - \G) (\hat\S^{-1}/n) (\Id{p} - \G)\th$.
Because as $n\to \infty$, $\th \to \t$ in probability by Eq.~\eqref{eq:asstau}, $\th^\top(\Id{p} - \G) (\hat\S^{-1}/n) (\Id{p} - \G)\th$ converges to $\t^\top (\Id{p} - \G) \S_\infty^{-1} (\Id{p} - \G) \t$ in probability, proving the statement (ii). When $\mathcal{G}$ fulfills the \ref{ass:main}, $\G \t = \t$ and hence $(\Id{p} - \G) \t = \boldsymbol{0}_p$, proving (i).
\end{proof}

\begin{proof}[Proof of Corollary \ref{cor:on-path}]
Suppose, without loss of generality, that $\mathcal{G}\neq \mathcal{G}^{(d)}$ and $\mathcal{G} \neq \mathcal{G}^{(1)}$; otherwise, it is included in any path by design. Let $\mathcal{A}$ be the set of all partitions for which the \ref{ass:main} holds and $\mathcal{N}$ be the set of all partitions for which the \ref{ass:main} does not hold. For an arbitrary partition $\mathcal{G}^*$, let $\G^*=\B \B^+$, where $\B$ is the block membership matrix corresponding to $\mathcal{B}_{\mathcal{G}^*}$, and set
\begin{align*}
	M = \min_{\mathcal{G}^* \in \mathcal{N}} \t^\top (\Id{p} - \G^*) \S_\infty^{-1} (\Id{p} - \G^*) \t.
\end{align*}
Observe that by the assumptions of Corollary \ref{cor:on-path}, $M>0$. For an arbitrary partition $\mathcal{G}^*$, let  $\mathcal{A}_{\succ \mathcal{G}^*}$ denote the set of all refinements of $\mathcal{G}^*$ in $\mathcal{A}$ whose cardinality is $|\mathcal{G}^*|+1$; the set $\mathcal{A}_{\succ \mathcal{G}^*}$ may of course be empty if $\mathcal{G}^*$ does not satisfy the \ref{ass:main}. Recall from \ref{subapp:sigma-tilde}, that $n\tilde\S( \Sh\mid \th, \mathcal{G}^\dagger, w ) \to \S_{\infty}$ element-wise in probability for any $G^\dagger \in \mathcal{A}_{\succ \mathcal{G}^*}$.

For arbitrary $\epsilon \in (0,M)$, (i) of Proposition \ref{prop:qform} implies that 
\begin{multline*}
\lim_{n\to \infty} \Pr \left[\bigcup_{\mathcal{G}^* \in \mathcal{A}} \bigcup_{\mathcal{G}^\dagger \in \mathcal{A}_{\succ \mathcal{G}^*}} \left\{ \frac{1}{n} \,\mathcal{L}\{\tt(\th,\mathcal{G}^*)\mid\th,\tilde\S( \Sh\mid \th, \mathcal{G}^\dagger, w )\} > \epsilon \right\} \right] \\
\le  \sum_{\mathcal{G}^* \in\mathcal{A}} \sum_{\mathcal{G}^\dagger \in \mathcal{A}_{\succ \mathcal{G}^*}}  \lim_{n\to \infty}\Pr \left[ \frac{1}{n} \, \mathcal{L}\{\tt(\th,\mathcal{G}^*)\mid\th,\tilde\S( \Sh\mid \th, \mathcal{G}^\dagger, w )\}> \epsilon \right] = 0.
\end{multline*}
Similarly, using (ii) of Proposition \ref{prop:qform}, we get
\begin{multline*}
\lim_{n\to \infty} \Pr \left[ \bigcup_{\mathcal{G}^* \in \mathcal{N}} \bigcup_{\mathcal{G}^\dagger \in \mathcal{A}_{\succ \mathcal{G}^*}} \left\{ \frac{1}{n} \,\mathcal{L} \{ \tt(\th,\mathcal{G}^*)\mid\th,\tilde\S( \Sh\mid \th, \mathcal{G}^\dagger, w )\} < M-\epsilon \right\} \right]\\
\le  \sum_{\mathcal{G}^*\in\mathcal{N}} \sum_{\mathcal{G}^\dagger \in \mathcal{A}_{\succ \mathcal{G}^*}}  \lim_{n\to \infty}\Pr \left[\left| \frac{1}{n} \, \mathcal{L}\{\tt(\th,\mathcal{G}^*)\mid\th,\tilde\S( \Sh\mid \th, \mathcal{G}^\dagger, w )\}  -\t^\top (\Id{p} - \G^*) \S_\infty^{-1} (\Id{p} - \G^*) \t \right| > \epsilon \right] = 0.
\end{multline*}

Now fix an arbitrary $\epsilon \in (0,M/2)$. Then,
\begin{multline*}
	\Pr(\mathcal{G}^* \not\in \mathcal{P}) \le \Pr \left[ \{ \mathcal{G}^* \not\in \mathcal{P}\}  \bigcap_{\mathcal{G}^* \in \mathcal{A}} \bigcap_{\mathcal{G}^\dagger \in \mathcal{A}_{\succ \mathcal{G}^*}} \left\{ \frac{1}{n} \, \mathcal{L}\{\tt(\th,\mathcal{G}^*)\mid\th,\tilde\S( \Sh\mid \th, \mathcal{G}^\dagger, w )\} \le \epsilon \right\} \right. \\
	 \left. \qquad \bigcap_{\mathcal{G}^* \in \mathcal{N}} \bigcap_{\mathcal{G}^\dagger \in \mathcal{A}_{\succ \mathcal{G}^*}} \left\{ \frac{1}{n} \, \mathcal{L}\{\tt(\th,\mathcal{G}^*)\mid\th,\tilde\S( \Sh\mid \th, \mathcal{G}^\dagger, w )\} \ge M-\epsilon \right\} \right]\\
		 \qquad\quad + \Pr\left [ \bigcup_{\mathcal{G}^* \in \mathcal{A}} \bigcup_{\mathcal{G}^\dagger \in \mathcal{A}_{\succ \mathcal{G}^*}} \left\{ \frac{1}{n} \, \mathcal{L}\{\tt(\th,\mathcal{G}^*)\mid\th,\tilde\S( \Sh\mid \th, \mathcal{G}^\dagger, w )\} > \epsilon \right\} \right]\\
	 \qquad\qquad + \Pr \left[ \bigcup_{\mathcal{G}^* \in \mathcal{N}} \bigcup_{\mathcal{G}^\dagger \in \mathcal{A}_{\succ \mathcal{G}^*}} \left \{ \frac{1}{n} \, \mathcal{L}\{\tt(\th,\mathcal{G}^*)\mid\th,\tilde\S( \Sh\mid \th, \mathcal{G}^\dagger, w )\} < M-\epsilon \right\} \right].
\end{multline*}

As $n \to \infty$, the last two terms tend to $0$ as shown above. The first term is however exactly $0$, because $\epsilon < M-\epsilon$ and hence partitions in $\mathcal{A}$ are necessarily selected in any iteration step $i \in \{d-1, \dots, |\mathcal{G}| + 1\}$ of~Algorithm \ref{algo:path}. When $i = |\mathcal{G}|$, the only partition in $\mathcal{A}$ that remains is $\mathcal{G}$, and so again it is selected, because $\mathcal{G}^{(i+1)} \in \mathcal{A}_{\succ \mathcal{G}}$ and  $\mathcal{G}$ is the only partition of cardinality $|\mathcal{G}|$ whose refinement is $\mathcal{G}^{(i+1)}$ and whose loss is at most $\epsilon$.
\end{proof}

\begin{proof}[Proof of Proposition \ref{prop:stop}]
Because $\mathcal{G}$ satisfies the \ref{ass:main}, $\G \t = \t$. Therefore,
\begin{align*}
(\th -\tt)^{\top} \Sh^{-1} (\th -\tt) & = \{\th - \t - \G(\th -\t)\}^{\top} \Sh^{-1} \{\th - \t - \G(\th -\t)\} = (\th - \t)^\top (\Id{p} - \G) \Sh^{-1} (\Id{p} - \G) (\th - \t).
\end{align*}
Because $\bss{A} \mapsto \bss{A}^{-1}$ is a continuous mapping on the space of nonsingular matrices \citep{Stewart:1969}, then as $n \to \infty$, $\Sh^{-1}/n$ converges element-wise to $\S_{\infty}^{-1}$ in probability. The asymptotic normality \eqref{eq:asstau} implies, together with Slutsky's lemma, that
$$
(\th -\tt)^{\top} \Sh^{-1} (\th -\tt) \rightsquigarrow \bss{V}^{\top} \S_{\infty}^{-1} \bss{V},
$$
where $\bss{V} \sim \mathcal{N}[\bss{0}_p,  (\Id{p} - \G) \S_{\infty}  (\Id{p} - \G)]$. Now set $\bss{M} = (\Id{p}-\G) \S_\infty (\Id{p}-\G)$ and $\bss{A} = \S^{-1}_\infty$. Following Lemma~\ref{lem:idem} and Remark \ref{rem:gamma-sigma-rel}, we easily get that
$\bss{M} \bss{A} = (\Id{p}-\G) \S_\infty (\Id{p}-\G) \S^{-1}_\infty = (\Id{p} - \G)$
is idempotent and has trace $\mathrm{tr}(\Id{p} - \G) = p - \mathrm{tr}(\G)$. An application of Theorem~8.6 in \cite{Severini:2005} thus yields that $\bss{V}^{\top} \S_{\infty}^{-1} \bss{V}$ is $ \chi_{p -\mathrm{tr}(\G)}^2$. Finally,
\begin{align*}
\mathrm{tr}(\G) = \sum_{r=1}^p \G_{rr} = \sum_{\ell = 1}^{L} \sum_{r \in \mathcal{B}_\ell} \frac{1}{|\mathcal{B}_{\ell}|} = L,
\end{align*}
and hence $(\th -\tt)^{\top} \Sh^{-1} (\th -\tt) \rightsquigarrow \chi_{p -L}^2$, as claimed.
\end{proof}

\subsection{Additional results for Section~\ref{sec:bijective}} \label{subapp:proof-T-inv}

In this section, we consider inverses of matrices in $\mathcal{T}_{\mathcal{G}}$. To this end, consider an arbitrary partition $\mathcal{G}$ of $\{1,\dots, d\}$ and introduce constraints on the diagonal entries of the matrices in $\mathcal{T}_{\mathcal{G}}$ through the set
\begin{align*}
	\mathcal{T}_{\mathcal{G}}^\dag = \{\R \in \mathcal{T}_{\mathcal{G}} : \text{for any $i,j$ such that } \D_{ij} = 1, \R_{ii} = \R_{jj}\}.
\end{align*}
A direct consequence of this definition is that $\mathcal{T}_{\mathcal{G}}^\dag \subset \mathcal{T}_{\mathcal{G}}$. Furthermore, $\T \in \mathcal{T}_{\mathcal{G}}^\dag$ since $\T_{ii} = 1$ for all $i \in \{1,\dots, d\}$.

\begin{lemma}\label{lem:T-inv}
If $\R \in \mathcal{T}_{\mathcal{G}}^\dag$ is invertible, then $\R^{-1} \in \mathcal{T}_{\mathcal{G}}^\dag$.
\end{lemma}

\begin{proof}
Invoking once again the Cayley--Hamilton Theorem as in the proof of Proposition \ref{prop:S-inv}, we need only show that if $\R , \bss{Q} \in \mathcal{T}_{\mathcal{G}}^\dag$, then $\R \bss{Q} \in \mathcal{T}_{\mathcal{G}}^\dag$. To this end, fix arbitrary $k_1$, $k_2 \in \{1,\dots, K\}$ and let $(i_1,j_1),(i_2,j_2) \in \mathcal{G}_{k_1} \times \mathcal{G}_{k_2}$ be such that $i_1=j_1$ if and only if $i_2 = j_2$. Then for any $k \in\{1,\dots, K\}$,
\begin{align*}
	\sum_{s \in \mathcal{G}_k} \R_{i_1 s} \bss{Q}_{s j_1} = \sum_{s \in \mathcal{G}_k} \R_{i_2 s} \bss{Q}_{s j_2}.
\end{align*}
Therefore,
\begin{align*}
	[\R \bss{Q}]_{i_1j_1} &= \sum_{s = 1}^{d} \R_{i_1 s} \bss{Q}_{s j_1} = \sum_{k = 1}^{K} \sum_{s \in \mathcal{G}_k} \R_{i_1 s} \bss{Q}_{s j_1} = \sum_{k = 1}^{K} \sum_{s \in \mathcal{G}_k} \R_{i_2 s} \bss{Q}_{s j_2} = [\R \bss{Q}]_{i_2j_2},
\end{align*}
which proves the claim.
\end{proof}



\end{document}